\documentclass[reqno]{amsart}

\usepackage[normalem]{ulem}

\usepackage{amsmath,amssymb,color}
\usepackage{amsfonts, amscd, epsfig, amsmath, amssymb,enumerate, bbm}
\usepackage{graphicx}
\usepackage{graphics}
\usepackage{stackengine}
\usepackage{color}
\usepackage{mathrsfs}
\usepackage{stmaryrd}
\usepackage{todonotes}
\usepackage{soul}
\usepackage{dsfont} 
\usepackage{BOONDOX-calo} 
\usepackage{enumitem} 

 \usepackage[usenames,dvipsnames]{pstricks}
 \usepackage{pstricks-add}
 \usepackage{epsfig}
 \usepackage{pst-grad} 
 \usepackage{pst-plot} 
 \usepackage[space]{grffile} 
 \usepackage{etoolbox} 
\usepackage[margin=3.5cm]{geometry}
 \makeatletter 
 \patchcmd\Gread@eps{\@inputcheck#1 }{\@inputcheck"#1"\relax}{}{}
 \makeatother

\newtheorem{theorem}{Theorem}[section]
\newtheorem{lemma}[theorem]{Lemma}
\newtheorem{proposition}[theorem]{Proposition}
\newtheorem{corollary}[theorem]{Corollary}
\newtheorem{remark}[theorem]{Remark}

\newtheorem{definition}{Definition}[section]

\usepackage{tikz}
\usetikzlibrary{backgrounds}
\usetikzlibrary{patterns,fadings}
\usetikzlibrary{arrows,decorations.pathmorphing}
\usetikzlibrary{decorations}
\usetikzlibrary{calc}
\usetikzlibrary{shapes.misc}

\usepackage{setspace}

\definecolor{navyblue}{rgb}{0.0, 0.0, 0.5}
\definecolor{light-gray}{gray}{0.95}
\definecolor{green}{RGB}{51,204,0}
\definecolor{red}{RGB}{211,0,0}
\definecolor{blue}{RGB}{0,59,210}
\usepackage{float}
\usepackage[colorlinks=true,linkcolor=blue,citecolor=purple,urlcolor=navyblue]{hyperref}
\def\centerarc[#1](#2)(#3:#4:#5){\draw[#1] ($(#2)+({#5*cos(#3)},{#5*sin(#3)})$) arc (#3:#4:#5);}

\newcommand{\mm}[1]{{\color{magenta}#1}}

\newcommand{\mc}[1]{{\mathcal #1}}

\newcommand{\bb}[1]{{\mathbb #1}}
\newcommand{\bs}[1]{{\boldsymbol #1}}

\renewcommand{\epsilon}{\varepsilon}

\newcommand{\R}{\mathbb R}

\newcommand{\Z}{\mathbb Z}
\newcommand{\N}{\mathbb N}
\newcommand{\Prob}{\mathbb P}

\newcommand{\E}{\mathbb E}
\newcommand{\Q}{\mathbb Q}

\newcommand{\diff}{\,\mathrm{d}}
\newcommand{\ind}{\mathds{1}}
\newcommand{\gene}{\mathcal{L}}

\renewcommand{\bar}{\overline}
\renewcommand{\tilde}{\widetilde}

\newcommand{\normm}[1]{{\left\vert\kern-0.1ex\left\vert\kern-0.1ex\left\vert\; #1 \; \right\vert\kern-0.1ex\right\vert\kern-0.1ex\right\vert}}

\newcommand{\dlangle}{\langle\!\langle}
\newcommand{\drangle}{\rangle\!\rangle}
\newcommand{\bulletx}{\mbox{\stackunder[1pt]{$\bullet$}{\tiny $x$}}}
\newcommand{\bulletone}{\mbox{\stackunder[1pt]{$\bullet$}{\tiny $1$}}}
\newcommand{\circx}{\mbox{\stackunder[1pt]{$\circ$}{\tiny $x$}}}
\newcommand{\circone}{\mbox{\stackunder[1pt]{$\circ$}{\tiny $1$}}}
\newcommand{\circs}{\mbox{$\circ$}}
\newcommand{\bullets}{\mbox{$\bullet$}}

\renewcommand{\leq}{\leqslant}
\renewcommand{\geq}{\geqslant}
\renewcommand{\le}{\leqslant}
\renewcommand{\ge}{\geqslant}
\newcommand{\eq}[2]{\begin{equation}\label{#1}#2\end{equation}}

\definecolor{MSBlue}{rgb}{.15,.0.35,.85}
\newcommand{\nota}[1]{{\color{MSBlue}#1}}
\newcommand{\act}{\mathfrak{a}}
\newcommand{\integers}[2]{{\llbracket #1 ,#2\rrbracket}}

\allowdisplaybreaks 

\newcommand\blfootnote[1]{%
	\begingroup
	\renewcommand\thefootnote{}\footnote{#1}%
	\addtocounter{footnote}{-1}%
	\endgroup
}

\title[Hydrodynamics for a facilitated exclusion process with slow/fast boundaries]{Hydrodynamic limit for an open Facilitated Exclusion Process with slow and fast boundaries}

\author{Hugo Da Cunha}
\address{Hugo Da Cunha -- Department of Mathematics, \'Ecole Normale Sup\'erieure (ENS) de Lyon,  69007 Lyon,  France}
\email{\href{mailto:hugo.da_cunha@ens-lyon.fr}{hugo.da\_cunha@ens-lyon.fr}}
\author{Cl\'ement Erignoux}
\address{Cl\'ement Erignoux -- Inria DRACULA, ICJ UMR5208, CNRS, \'Ecole Centrale de Lyon, INSA Lyon, Université Claude Bernard Lyon 1,
Université Jean Monnet, 69603 Villeurbanne, France.}
\email{\href{mailto:clement.erignoux@inria.fr}{clement.erignoux@inria.fr}}
\author{Marielle Simon}
\address{Marielle Simon -- Université Claude Bernard Lyon 1, ICJ UMR5208, CNRS, \'Ecole Centrale de Lyon, INSA Lyon, Université Jean Monnet,
69622 Villeurbanne, France \emph{and} GSSI, Viale Francesco Crispi 7, 67100 L'Aquila, Italy.}
\email{\href{mailto:msimon@math.univ-lyon1.fr}{msimon@math.univ-lyon1.fr}}

\begin{document}

\begin{abstract}
We study the symmetric \emph{facilitated exclusion process} (FEP) on the finite one-dimensional lattice $\lbrace 1,\hdots, N-1\rbrace$ when put in contact with boundary reservoirs, whose action is subject to an additional kinetic constraint in order to enforce ergodicity, and whose speed is of order $N^{-\theta}$ for some parameter $\theta$. We derive its hydrodynamic limit as $N\to\infty$, in the diffusive space-time scaling, when the initial density profile is supercritical. More precisely, the macroscopic density of particles evolves in the bulk according to a fast diffusion equation as in the periodic case, which is now subject to boundary conditions that can be of Dirichlet, Robin or Neumann type depending on the parameter $\theta$. In the Dirichlet case, the FEP exhibits a very peculiar behaviour: unlike for the classical SSEP, and due to the two-phased nature of FEP, the reservoirs impose boundary densities which do not coincide with their equilibrium densities. The proof is based on the classical entropy method, but requires significant adaptations to account for the FEP's non-product stationary states and to deal with the non-equilibrium setting.
\end{abstract}

\maketitle


\section{Introduction}

\blfootnote{\textsc{Acknowledgements:} We warmly thank Kirone Mallick for very enlightening discussions about the stationary states of the FEP in contact with stochastic reservoirs. We also thank the anonymous referee who highlighted to us the most natural physical mechanism to consider at the boundaries.}

\subsection{The facilitated exclusion process}

Over the last century, there has been a rapidly growing interest in describing macroscopic features of the physical world at the
microscopic level. In particular, a variety of models has been introduced to describe the evolution of a \emph{multiphased media}, as for instance the joint evolution of liquid and solid phases. Such complex phenomena often feature
absorbing phase transitions, which have been closely investigated by both physicists and mathematicians over the last decades.

In particular, the class of \emph{kinetically constrained stochastic lattice gases}, which has been put forward in the 80's (see \textit{e.g.}~\cite{ritortS2003glassy} for a review), is known to accurately illustrate some microscopic mechanisms at the origin of liquid/solid interfaces. In these systems, particles are situated on the sites of a discrete lattice, and jump at random times to neighbouring sites, following microscopic rules: we consider here in particular the \emph{exclusion rule}, which prevents two particles from being on the same site, and an additional \emph{kinetic constraint}, which makes a given jump possible or not depending on the \emph{local} configuration around the jump edge. Such kinetically constrained lattice gases can be seen as the Kawasaki-type counterparts to the Glauber-type non-conservative kinetically constrained spin models (see \textit{e.g.}~\cite{biroli2007} and \cite{reviewKCM} for a more exhaustive review), whose dynamics involves particle creation/annihilation rather than jumps.

\medskip

One of these models, called \emph{facilitated exclusion process} (FEP) has been proposed by physicists in \cite{RossiPastorVespignani00}
and further investigated by physicists and mathematicians in, for instance, \cite{lubeck, oliveira, BasuMohanty09, BaikBarraquandCorwinToufic16,  Chen2019, blondel2020hydrodynamic, blondel2021stefan, Goldstein_2024}. 
The \emph{symmetric one-dimensional FEP} on the discrete lattice $\Lambda \subset \Z$ is defined as follows: it is an exclusion process, meaning that each site is either empty or occupied by one particle. Besides, a particle is considered  \emph{active} if at least one of its two neighbouring sites is occupied. Then,  active particles jump randomly at rate 1 to any empty nearest neighbour. Because of the kinetic constraint, the FEP exhibits a phase separation with critical density $\rho_c=\frac12$: more precisely, it remains active at supercritical densities $\rho>\frac12$ (\textit{i.e.} there is always at least an active particle in the system), whereas if $\rho<\frac12$, it   reaches an absorbing state after some \emph{transience time}\footnote{See also Section \ref{subsec:transience} for more detailed explanations., \textit{i.e.}~the final particle configuration has no active particle}.
The FEP  is \emph{cooperative}, in the sense that there is no \emph{mobile cluster}\footnote{A mobile cluster is a set of particles able to move autonomously in the system under the kinetic constraint, which provides strong local mixing for the system.} of particles in
the system. The cooperative nature of the FEP
distorts its equilibrium measures, which are no longer product, thus generating significant mathematical difficulties. Nevertheless, in the supercritical regime, the grand-canonical states $\pi_\rho$
are explicit, and supported by the \emph{ergodic component}, namely the set of configurations where each empty site is surrounded by particles. Those grand-canonical states $\pi_\rho$ are translation invariant and can be defined sequentially, through a Markovian construction, by filling an arbitrary site with probability $\rho>\frac12$, and following each particle by another particle with probability $\mathfrak{a}(\rho)=(2\rho-1)/\rho$. This quantity $\mathfrak{a}(\rho)$ represents the \emph{density of active particles}, namely particles with at least one occupied neighbour, in the system at density $\rho$. To make sure empty sites are isolated, each empty site is instead followed by a particle with probability $1$.

In \cite{blondel2020hydrodynamic, blondel2021stefan}, the hydrodynamic limit of the symmetric FEP with
periodic boundary conditions is derived, and takes the form of a Stefan problem (also called free boundary problem) with non-linear diffusion coefficient $D(\rho) = \mathfrak{a}'(\rho)\ind_{\{\rho>\frac12\}}$. In other words, the diffusive
supercritical phase (\emph{i.e.~}the macroscopic regions where the initial density profile satisfies $\rho^\mathrm{ini}(u)>\frac12$) progressively invades the initial frozen subcritical phase (where $\rho^\mathrm{ini}(u)<\frac12$), until one of the phases disappears (depending on the total mass $\int\rho^\mathrm{ini}$ of the initial profile being super- or sub-critical). 
For \emph{asymmetric jump rates} (namely the jump rate to the right is different from the one to the left), the hyperbolic 
Stefan problem hydrodynamic limit was derived in \cite{erignoux_mapping_2024}. More recently, the stationary macroscopic
equilibrium fluctuations have been characterized in the symmetric, weakly asymmetric and asymmetric cases in \cite{erignoux2023fluctuations}, and the transience time was studied in details in \cite{EM24}. All these results rely in parts on mapping arguments (\textit{i.e.}~the FEP is mapped onto an auxiliary process) which all fail in dimension
higher than $1$, and in the presence of boundaries. The stationary and absorbing states for the FEP were also
extensively studied both in the symmetric and asymmetric cases \cite{2019JSMTE, Goldstein2020, Goldstein2022, Chen2019}, and
once again rely on mapping arguments. 

\subsection{Effect of boundary interactions}

As the effect of boundary interactions on lattice gases has been under considerable scrutiny in recent years, it is now natural to investigate the macroscopic effect of boundary dynamics on the FEP.  Adding reservoir-type interactions at the extremities of microscopic systems is a classical way to induce boundary
conditions at the macroscopic level (\emph{e.g.}~in the hydrodynamic PDE), see for instance \cite{goncalves2019book} for a recent review
in the case of symmetric simple exclusion (SSEP). In turn, these boundary effects give access to the macroscopic non-equilibrium features of the model considered \cite{Derrida_2007,Derrida_2011}. In the FEP, particles  injected by  reservoirs may become blocked by the kinetic constraint, and therefore change
the effective stationary density imposed by the reservoirs, so that the effect of reservoir dynamics on the FEP is far from trivial.

\medskip

In this work, we consider the \emph{boundary-driven one-dimensional symmetric FEP} on the finite lattice $\Lambda_N:=\{1,\dots,N-1\}$, with two stochastic reservoirs at the extremities, whose dynamics is illustrated in Figure \ref{fig:dynamics}.  Precisely, let us fix $\alpha,\beta\in(0,1)$, and let $\theta\in\R$ and $\kappa>0$ be two parameters which  regulate the speed of the boundary dynamics. We assume the following: 
\begin{enumerate} \item the stochastic reservoirs at both ends inject particles at the boundary sites $1$ and $N-1$, if the latter are empty, at respective rates $\alpha\kappa N^{-\theta}$ and $\beta\kappa N^{-\theta}$ ;
	\item they can also remove boundary particles at sites $1$ and $N-1$, at respective rates $(1-\alpha)\kappa N^{-\theta}$, $(1-\beta)\kappa N^{-\theta}$, \emph{only if} the boundary particle is followed by another particle. For instance, see Figure \ref{fig:dynamics}: the particle situated at site $1$ in the middle of the bottom illustration  cannot be absorbed by the reservoir, since its neighbouring site $2$ is empty ;
	\item the particles in contact with reservoirs are always \emph{active}, meaning that if a particle is situated at one of the two extremities $x=1$ or $x=N-1$, then it  can always jump towards the bulk. Moreover, the rate to jump from site $1$ to $2$ is equal to $\alpha$, while the rate to jump from site $N-1$ to $N-2$ is equal to $\beta$. Equivalently, this mechanism can be interpreted as the boundary particles being active at any given time with respective probabilities  $\alpha$ and  $\beta$.
	\item for the particles situated between sites $x=2$ and $x=N-2$, the jump rates are the same as in the standard FEP, namely: a particle jump to an empty neighbour at rate $1$ provided that the other neighbour is occupied by another particle.	
	\end{enumerate} 

\begin{figure}
	\begin{tikzpicture}[scale=.9]
		\draw [-|] (0,0)--(10,0);
		\foreach \x in {0,...,10} \draw (\x,-.1)node[below]{\x} -- (\x,0.1);
		\shade[ball color=teal](1,.25) circle (0.2);
		\shade[ball color=teal](3,.25) circle (0.2);
		\shade[ball color=teal](4,.25) circle (0.2);
		\shade[ball color=teal](7,.25) circle (0.2);
		\shade[ball color=teal](9,.25) circle (0.2);
		\centerarc[thick,<-](1.5,0.4)(10:170:0.45);
		\draw (1.5,0.9) node[anchor=south] {${\color{red}\alpha}$};
		
		\centerarc[thick,->](2.5,0.4)(10:170:0.45);
		\draw (2.5,0.9) node[anchor=south] {$1$};
		\centerarc[thick,->, dashed](3.5,0.4)(10:170:0.45);
		\draw (3.5,.85) node[thick, red] {\LARGE $\times$};
		\centerarc[thick,<-](4.5,0.4)(10:170:0.45);
		\draw (4.5,0.9) node[anchor=south] {$1$};
		\centerarc[thick,->, dashed](6.5,0.4)(10:170:0.45);
		\draw (6.5,.85) node[thick, red] {\LARGE $\times$};
		\centerarc[thick,<-, dashed](7.5,0.4)(10:170:0.45);
		\draw (7.5,.85) node[thick, red] {\LARGE $\times$};
		\centerarc[thick,->](8.5,0.4)(10:170:0.45);
		\draw (8.5,0.9) node[anchor=south] {${\color{violet}\beta}$};
		
		
		\draw (-0.6,-0.7) arc (-80:80:0.7);
		\shade[ball color=black](-0.5,0.4) circle (0.2);
		\shade[ball color=black](-0.25,0.1) circle (0.2);
		\shade[ball color=black](-0.5,0) circle (0.2);
		\shade[ball color=black](-0.3,-0.2) circle (0.2);
		\shade[ball color=black](-0.5,-0.4) circle (0.2);
		\draw (-2.1,0) node[anchor=south] {\textsf{Left reservoir}};
		\draw (-2.1,-0) node[anchor=north] {\textsf{with density} ${\color{red}\alpha}$};
		\draw (10.6,0.7) arc (100:260:0.7);
		\shade[ball color=black](10.5,0.4) circle (0.2);
		\shade[ball color=black](10.5,0) circle (0.2);
		\shade[ball color=black](10.5,-0.4) circle (0.2);
		\shade[ball color=black](10.25,0.1) circle (0.2);
		\shade[ball color=black](10.3,-0.2) circle (0.2);
		\draw (12.1,0) node[anchor=south] {\sf Right reservoir};
		\draw (12.1,-0) node[anchor=north] {\sf with density ${\color{violet}\beta}$};
	\end{tikzpicture}
	
	\vspace{.3cm}
	
	\begin{tikzpicture}
		\draw [->] (0,0)--(2.5,0);
		\foreach \x in {0,...,2} \draw (\x,-.1)node[below]{\x } -- (\x,0.1);
		\draw[dashed] (2,0.25) circle (0.2);
		\centerarc[thick,<-](.5,0.4)(10:170:0.45);
		\draw (0.5,0.9) node[anchor=south] {${\color{red}\alpha} \kappa N^{-\theta}$};
		
		\draw (-0.6,-0.7) arc (-80:80:0.7);
		\shade[ball color=black](-0.5,0.4) circle (0.2);
		\shade[ball color=black](-0.25,0.1) circle (0.2);
		\shade[ball color=black](-0.5,0) circle (0.2);
		\shade[ball color=black](-0.3,-0.2) circle (0.2);
		\shade[ball color=black](-0.5,-0.4) circle (0.2);
		
		\draw [->] (5,0)--(7.5,0);
		\foreach \x in {0,...,2} \draw (\x +5,-.1)node[below]{\x } -- (\x +5,0.1);
		\centerarc[thick,->, dashed](5.5,0.4)(10:170:0.45);
		\draw (5.5,0.85) node[thick, red] {\LARGE $\times$};
		\draw (4.4,-0.7) arc (-80:80:0.7);
		\shade[ball color=teal](6,.25) circle (0.2);
		\shade[ball color=black](4.5,0.4) circle (0.2);
		\shade[ball color=black](4.75,0.1) circle (0.2);
		\shade[ball color=black](4.5,0) circle (0.2);
		\shade[ball color=black](4.7,-0.2) circle (0.2);
		\shade[ball color=black](4.5,-0.4) circle (0.2);
		
		\draw [->] (10,0)--(12.5,0);
		\foreach \x in {0,...,2} \draw (\x +10,-.1)node[below]{\x } -- (\x +10,0.1);
		\centerarc[thick,->](10.5,0.4)(10:170:0.45);
		\draw (10.5,0.9) node[anchor=south] {$(1-{\color{red}\alpha})\kappa N^{-\theta}$};
		
		\draw (9.4,-0.7) arc (-80:80:0.7);
		\shade[ball color=teal](11,.25) circle (0.2);
		\shade[ball color=teal](12,.25) circle (0.2);
		\shade[ball color=black](9.5,0.4) circle (0.2);
		\shade[ball color=black](9.75,0.1) circle (0.2);
		\shade[ball color=black](9.5,0) circle (0.2);
		\shade[ball color=black](9.7,-0.2) circle (0.2);
		\shade[ball color=black](9.5,-0.4) circle (0.2);
	\end{tikzpicture}
	
	\caption{Illustration (for the value $N=10$) of the bulk dynamics  (above) and of the boundary exchange dynamics with the left reservoir (below).  Allowed jumps are provided with their corresponding rates. Forbidden jumps (with rate $0$) are denoted by {\color{red}$\times$}.}
	\label{fig:dynamics}
\end{figure}
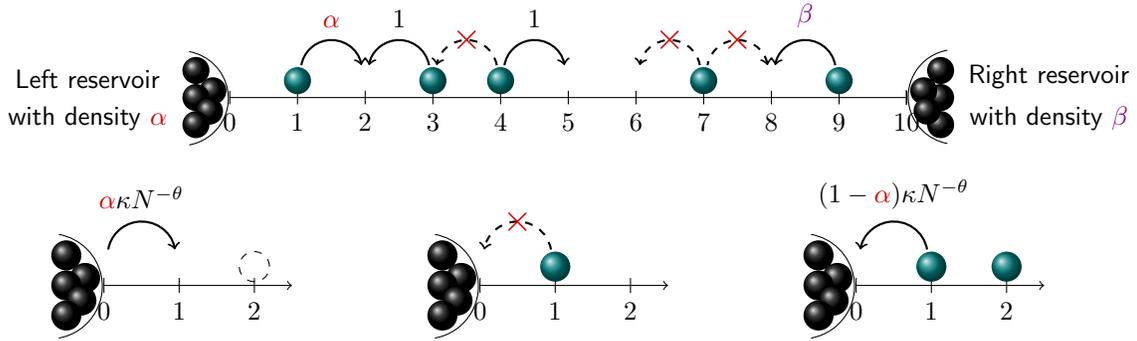

 Let us comment on the choices made to define the reservoir dynamics: the kinetic constraint (2) imposed at both reservoirs is not standard. It is made with the main purpose of preserving ergodic configurations: namely, if the particle system starts from an ergodic state  $\eta(0)$ (where every empty site is surrounded by particles), then it is not difficult to see that at any time $t>0$, $\eta(t)$ remains ergodic. With another definition of the reservoirs, the ergodic component would not remain stable under the dynamics, and this would  raise considerable difficulties from a hydrodynamic limit standpoint. Besides, the different rates at sites $1$ and $N-1$, as explained in the rule (3) above, are the most natural ones for a physical reason: this choice amounts to coupling a finite FEP with infinite reservoirs of particles which are themselves FEP systems. As a result, in the equilibrium case $\alpha=\beta$, we will see that the unique equilibrium measure in the open system is simply the marginal of an infinite grand-canonical measure.

\subsection{Active density and macroscopic limit}

As we mentioned earlier, the quantity $\mathfrak{a}(\rho)$ called \emph{active density} plays an important role. 
 This important relation between density and active density, namely $\mathfrak{a}(\rho)=(2\rho-1)/\rho$ can easily be inverted: we denote by ${\bar\rho} : [0,1]\longrightarrow [\frac12,1]$  the inverse function
\begin{equation}\label{eq:barrhointro}
	\bar\rho (a)=\act^{-1}(a)=\frac{1}{2-a}.
\end{equation}
We are now ready to state our main result. To focus on the most salient challenges of boundary interactions, we start our process straight from the ergodic component, in order to avoid some issues related to the transience time for the FEP. This implies that  the microscopic system is assumed to be initially already one-phased, with a uniformly-supercritical density. More precisely, we consider here the boundary-driven symmetric FEP in the diffusive time scale, started from an ergodic initial configuration fitting a supercritical density profile $\rho^{\rm ini}:[0,1]\to (\frac12,1]$. Our main result states that the empirical particle density converges as $N\to\infty$ to its hydrodynamic limit $\rho(t,u)$, which is the unique (weak) solution to 
\begin{equation}
	\label{eq:Diffeq}
	\partial_t \rho = \partial_u^2 \mathfrak{a}(\rho), \qquad \rho(0,\cdot) = \rho^{\rm ini}
\end{equation}
with different boundary conditions depending on the speed of the boundaries' exchanges: more precisely, \begin{itemize} \item  if $\theta<1$, the macroscopic density satisfies the Dirichlet boundary conditions
\[ \rho(\cdot,0) = \bar\rho(\alpha), \qquad \rho(\cdot,1)= \bar\rho(\beta) \; ; \]
\item  if $\theta=1$,  the macroscopic density satisfies the Robin boundary conditions,
\[\big(\partial_u \mathfrak{a}(\rho)\big)(\cdot,0) = \kappa \big( \act(\rho)(\cdot,0)-\alpha\big), \qquad \big(\partial_u \mathfrak{a}(\rho)\big)(\cdot,1) = \kappa \big( \beta - \act(\rho)(\cdot,1)\big) \; ;  \]
\item if $\theta >1$,  the macroscopic density satisfies the Neumann boundary conditions
\[\big(\partial_u \mathfrak{a}(\rho)\big)(\cdot,0) = 0, \qquad \big(\partial_u \mathfrak{a}(\rho)\big)(\cdot,1) = 0.  \]
\end{itemize}
As seen from the definition of $\bar\rho$ given in \eqref{eq:barrhointro}, the \emph{effective} boundary density imposed on the FEP by a reservoir with density $\alpha$ is equal to the density for which the \emph{active} density is equal to $\alpha$. In other words, the reservoirs' densities correspond to the boundary active densities, rather than the standard densities, as it is the case for the SSEP for instance. The corresponding result for the SSEP, with classical boundary conditions obtained in the same different regimes of $\theta$,  can be found in \cite{baldasso2017exclusion} and \cite{goncalves2019book}. The precise meaning of solution to the previous boundary-driven PDE will be given in Definitions \ref{defin:weaksolDirichlet} -- \ref{defin:weaksolNeumann}.

The general case where the initial profile takes any values in $[0,1]$ is technically challenging, partly because mappings are not available in the presence of reservoirs since the total number of particles is no longer conserved. We fully expect, however, that the hydrodynamic limit holds in that case as well and takes the form of a Stefan problem as in the periodic case. This is left for future work. Remarkably, a simple argument shows that \emph{unconstrained} reservoirs with density $\alpha$, \emph{i.e.~}the classical ones which remove particles at rate $1-\alpha$ without requiring another neighbouring particle, create the same unusual boundary conditions at the level of the macroscopic density. However, this framework raises significant technical difficulties, in particular very few information is available on the stationary states, thus it is also left for future work.  

\subsection{Strategy of the proof and outline} 

Let us now present briefly the strategy of the proof and its main novelties. As the hydrodynamic limit plays the role of a law of large numbers for the empirical density of particles, the detailed knowledge of the  stationary states is a crucial element in the proof. As expected, this is particularly challenging for the FEP (whose equilibrium states are not product \cite{blondel2020hydrodynamic}), and even more so in the non-equilibrium setting $\alpha\neq \beta$, which induces long-range correlations.  We first explicitly derive the equilibrium stationary state  in the presence of two reservoirs with  $\alpha=\beta$, which turns out to be the restriction to the finite system of the grand-canonical state $\pi_{\bar\rho(\alpha)}$ of the FEP with equilibrium density $\bar\rho(\alpha)$. Inspired by previous work \cite{erignoux2023fluctuations},  we represent it by a Markov construction.  Then, we construct an approximation of the stationary state in the non-equilibrium case $\alpha\neq \beta$, following the same Markov construction, and using the fact that the \emph{active} density in the bulk should interpolate linearly between its two boundary values.

The rest of the proof then follows Guo,  Papanicolaou and Varadhan's \emph{entropy method} \cite{guo1988nonlinear}, which relies on the classical \emph{one-block} and \emph{two-blocks} estimates, in order to replace microscopic observables by functions of the empirical measure. Since we are not in a periodic setup, and because the FEP's invariant states are not product (they charge  the ergodic component only),  some care is required. In particular, the approximated stationary states do not lend themselves easily to conditioning to local boxes. The adaptation of the entropy method to non-product, non-explicit distributions with strong local correlations is the  major contribution of our work.

\medskip

This article is organized as follows. In Section \ref{sec:modelres}, we introduce the FEP in contact with constrained reservoirs starting from the ergodic component, and state our main result, Theorem \ref{thm: hydrodynamic limit}, namely the hydrodynamic limit  for the boundary-driven FEP. In Section \ref{sec: stationary measures}, we study  its \emph{local} stationary states. Section \ref{sec:approx} is dedicated to building an approximate stationary state for the non-equilibrium FEP, inspired by the explicit results obtain in the previous section. Once the approximate stationary state is built, we obtain in the rest of Section \ref{sec:approx} the associated density field and dynamical Dirichlet estimates. In Section \ref{sec:HDL}, we finally exploit those Dirichlet estimates, in order to adapt the classical entropy method and complete the proof of the hydrodynamic limit. Since significant adaptations need to be made, we expose in detail the proof of the fundamental replacement lemmas, namely Lemma \ref{lemma: replacement bulk} and Lemma \ref{lemma:replacement_boundary}, in Section \ref{sec:replacement_boundary} (for the boundary replacement) and Section \ref{sec: Replacement lemma bulk} (for the bulk replacement). 


\subsection{General notations}

We gather here some general  notations that will be used throughout this article.

\begin{itemize}
	\item We use double brackets $\llbracket$ to denote integer segments, \emph{e.g.~}if $a,b\in\N$ are such that $a<b$,  $\llbracket a,b\rrbracket = \{ a,a+1,\dots ,b-1,b\}$.
	\item The integer $N\in\N$, $N\neq 0$ is a scaling parameter that shall go to infinity.
	\item Given two functions $f,g\in L^2\big( [0,1]\big)$,  we denote by
	\[\langle f,g\rangle = \int_0^1 f(u)g(u)\diff u\]
	their $L^2$ scalar product.  If $m$ is a finite measure on $[0,1]$ and $f\in L^2(m )$,  we also denote by $\langle m,f\rangle$ the integral of $f$ with respect to $m$. 
	\item For any non-negative sequence $(u_k)_{k\in\N}$ possibly depending on other parameters than $k$,  we will denote by $\mathcal{O}_k(u_k)$ (resp. $\mathcal{o}_k(u_k)$) an arbitrary sequence $(v_k)_{k\in\N}$ for which there exists a constant $C>0$ (resp. a vanishing sequence $(\varepsilon_k)_{k\in\N}$) -- possibly depending on other parameters -- such that
	\[\forall k\in\N ,\qquad |v_k|\le Cu_k \qquad (\mbox{resp.  }|v_k|\le\varepsilon_ku_k ).\]
	In the absence of ambiguity in the parameters, we simply write $\mathcal{O}(u_k)$ and $\mathcal{o}(u_k)$.
	\item A particle configuration is an element $\eta\in\{0,1\}^\Lambda$ for some $\Lambda \subset \Z$. Given a function $g(\eta)$, and given a time trajectory $\eta(t),\;t\geq 0$, whenever convenient we will simply write $g(t)$ for $g(\eta(t)).$
	\item If $\Lambda$ is a finite subset of $\Z$, we denote by $|\Lambda |$ its cardinality.
	\item When a new notation is introduced inside of a paragraph and is going to be used throughout, we colour it \nota{in blue}.
\end{itemize}

\section{Model and main results}
\label{sec:modelres}

\subsection{Definition of the model}

Let us introduce the \emph{boundary-driven facilitated exclusion process} which is investigated in this paper. This particle system is evolving on a finite one-dimensional lattice of size $N-1$, called its \emph{bulk} \nota{$\Lambda_N =\integers{1}{N-1}$}. The extreme sites $1$ and $N-1$ are called \emph{boundaries}. A \emph{particle configuration} is a variable $\eta = (\eta_x)_{x\in\Lambda_N}\in \nota{\Omega_N:=\lbrace 0,1\rbrace^{\Lambda_N}}$, where, as usual for exclusion processes, $\eta_x=1$ (resp.~$\eta_x=0$) means that site $x\in\Lambda_N$ is occupied by a particle (resp.~empty). 

We consider here the \emph{symmetric} Facilitated Exclusion Process (FEP), where particles jump at rate $1$ to each neighbouring site provided the target site is empty (this is the \emph{exclusion rule}) and that its other neighbouring site is occupied (this is the \emph{kinetic constraint}).  In the bulk, the dynamics is ruled by the Markov generator $\mathcal{L}_0$ which is defined as follows: for any $f:\Omega_N\longrightarrow\R$, and any $\eta\in\Omega_N$,
\begin{equation}
\label{def:genbulk}
	\mathcal{L}_0f(\eta ) = \sum_{x=1}^{N-2}c_{x,x+1}(\eta )\big[ f(\eta^{x,x+1})-f(\eta )\big] ,
\end{equation}
where
\begin{equation}
\label{eq:bulkjumprate}
	c_{x,x+1}(\eta ):=\eta_{x-1}\eta_x(1-\eta_{x+1})+(1-\eta_x)\eta_{x+1}\eta_{x+2},
\end{equation}
is the jump rate encompassing both constraints, and $\eta^{x,y}$ is the configuration where the values at sites $x$ and $y$ have been exchanged, namely
\begin{equation}
\label{eq:Defetaxy}
\forall z\in\Lambda_N,\quad \eta^{x,y}_z =\begin{cases}
	\eta_z & \mbox{ if }z\neq x,y, \\ 
	\eta_y & \mbox{ if }z=x, \\ 
	\eta_x & \mbox{ if }z=y.
	\end{cases}
\end{equation}
Let \nota{$\alpha ,\beta\in (0,1)$} be two parameters which will encode the respective densities of two stochastic particle reservoirs in contact with the boundaries.  In \eqref{eq:bulkjumprate}, for $x=1$, $x=N-1$ we define by convention 
\begin{equation}
\label{def:conventions1}
	\eta_0 \equiv \alpha\qquad \mbox{ and }\qquad \eta_N\equiv \beta.
\end{equation}
In particular, this convention implies that a particle at one of the boundaries (\textit{i.e.}~$1$ or $N-1$) is always able to jump to the neighbouring site if the latter is empty (thanks to the presence of the stochastic reservoir), and the rate of this jump is equal to the corresponding \emph{reservoir density}.

Besides, the stochastic reservoirs are able to exchange particles with the bulk. Let us first introduce  two additional parameters \nota{$\theta\in\R$} and \nota{$\kappa >0$} which will rule the speed of those exchanges. Then, particles can be either created or absorbed by the reservoirs, as follows:
\begin{itemize}
	\item if site $x=1$ (resp.~$x=N-1$) is empty, then the left (resp.~right) reservoir injects a particle at rate $\kappa \alpha N^{-\theta}$ (resp.~$\kappa\beta N^{-\theta}$) at this site ;
	\item if there is a particle at site $x=1$ (resp.~$x=N-1$), then the reservoir absorbs it at rate $\kappa (1-\alpha )N^{-\theta}$ (resp.~ $\kappa (1-\beta )N^{-\theta}$) \emph{only if} site $x=2$ (resp.~$x=N-2$) is also occupied. This additional kinetic constraint in case of absorption is consistent with the bulk kinetic constraint: in order to leave the system, a particle also needs an occupied neighbour.
\end{itemize}
In other words, the boundary dynamics is ruled by the generator
\[ \frac{\kappa}{N^\theta} (\mathcal{L}_\ell+\mathcal{L}_r),\]
where, for any $f:\Omega_N\longrightarrow\R$ and any $\eta\in\Omega_N$,
\begin{equation}
\label{eq:boundarygen}
	\mathcal{L}_\ell f(\eta ) = b_\ell (\eta ) \big[ f(\eta^1)-f(\eta )\big] \qquad \mbox{ and }\qquad \mathcal{L}_rf(\eta ) =  b_r(\eta )\big[ f(\eta^{N-1})-f(\eta )\big]
\end{equation}
with boundary rates given by
\begin{equation}
\label{def:boundaryrates}
	b_\ell (\eta ) = \alpha ( 1-\eta_1) +(1-\alpha )\eta_1\eta_{2} \qquad \mbox{ and }\qquad b_r(\eta ) = \beta ( 1-\eta_{N-1}) +(1-\beta )\eta_{N-1}\eta_{N-2}
\end{equation}
and where $\eta^x$ is the configuration obtained from $\eta$ by flipping the coordinate $x$:
\begin{equation}
	\eta^x_z=\begin{cases}\eta_z & \mbox{ if }z\neq x, \\ 
		1-\eta_x & \mbox{ if }z=x.
		\end{cases}
\end{equation}
Finally, the \emph{boundary-driven symmetric FEP} considered in this paper is  ruled by the total generator
\begin{equation}\label{def:L_N}
	\mathcal{L}_N:=\mathcal{L}_0+\frac{\kappa}{N^\theta}(\mathcal{L}_\ell +\mathcal{L}_r).
\end{equation}

\bigskip

As already pointed out in \cite{blondel2020hydrodynamic}, the FEP belongs to the class of \emph{gradient models} because the instantaneous current of particles in the bulk, namely 
\begin{equation}\label{def:current}
	j_{x,x+1}(\eta )=c_{x,x+1}(\eta )(\eta_x-\eta_{x+1}), \qquad  x\in\integers{1}{N-2}
\end{equation}
 can be written under the form
\begin{equation}\label{eq:gradientcondition}
	j_{x,x+1}(\eta )=h_x(\eta )-h_{x+1}(\eta )
\end{equation}
where $h_x$ is the following \emph{local}\footnote{\textit{i.e.}~which depends on a finite number of coordinates.} function  defined for $x\in\integers{1}{N-1}$ by
\begin{equation}\label{def:hx}
	h_x(\eta )=\eta_{x-1}\eta_x+\eta_x\eta_{x+1} - \eta_{x-1}\eta_x\eta_{x+1},
\end{equation}
with the convention $\eta_0=\alpha, \eta_N=\beta$. Note that here, the gradient decomposition is valid for any $x\in\integers{1}{N-2}$. At the boundaries, we have
\begin{equation}\label{def:currentboundaries}
	j_{0,1}(\eta )=\frac{\kappa}{N^\theta} b_\ell (\eta )(1-2\eta_1)\qquad\mbox{ and }\qquad j_{N-1,N}(\eta )=\frac{\kappa}{N^\theta} b_r(\eta)(2\eta_{N-1}-1)
\end{equation}
and therefore we can write a similar decomposition as in \eqref{eq:gradientcondition}, namely
\begin{equation}
	\label{eq:gradientboundaries}
	j_{0,1}(\eta) = \frac{\kappa}{N^\theta} (h_0(\eta)-h_1(\eta)), \qquad j_{N-1,N}(\eta) = \frac{\kappa}{N^\theta} (h_{N-1}(\eta)-h_N(\eta))
\end{equation}
if we further define by convention
\begin{equation}
	\label{def:conventions2}
	h_0\equiv \alpha\qquad\mbox{ and }\qquad h_N\equiv\beta.
\end{equation}

\begin{definition}[Active particle]\label{de:active}\upshape
	A particle at site $x\in\Lambda_N$ is said to be \emph{active} if it has at least one occupied neighbour, or it is situated at one of the boundaries ($x=1$ or $x=N-1$). In particular, if $x\neq 1$ and $x \neq N-1$, the function $h_x$ can be interpreted as the indicator function that an active particle lies at site $x$.
\end{definition}



\subsection{Phase transition and grand-canonical equilibrium measures}

\subsubsection{Frozen and ergodic configurations}

\label{subsec:transience}

In the absence of boundary interactions, the symmetric FEP is now quite well understood, see \cite{blondel2020hydrodynamic,blondel2021stefan} for a detailed study in the periodic case. In particular, this one-dimensional model exhibits a phase separated behaviour, which depends on the local particle density. 

More precisely, let us define its critical density $\nota{\rho_c:=\frac12}$. If the process starts from an initial state with subcritical total density $\rho<\rho_c$, then, after a \emph{transience time}, almost surely every particle becomes isolated (surrounded by empty sites), \textit{i.e.}~the FEP reaches its \emph{frozen component}, where configurations contain no active particles. If instead, the process is started from a state with  supercritical density $\rho >\rho_c$, then, after a transience time, it  reaches its \emph{ergodic component}, made of configurations in which all empty sites are isolated (surrounded by occupied ones).  The time to reach the ergodic component, starting from a product state, has been estimated in \cite{blondel2020hydrodynamic}, while the general case has been recently studied in detail in \cite{EM24}, starting from the worst possible configuration.
The definition of the ergodic component stems from the fact that pairs of neighbouring empty sites can be separated by the dynamics, but not created, which makes the ergodic component  irreducible for the Markov process. We give  in Figure \ref{fig: active particles} an example (in our boundary-driven setting) of an ergodic configuration in $\Omega_{13}$ where we highlight its active particles (recall Definition \ref{de:active}). 

 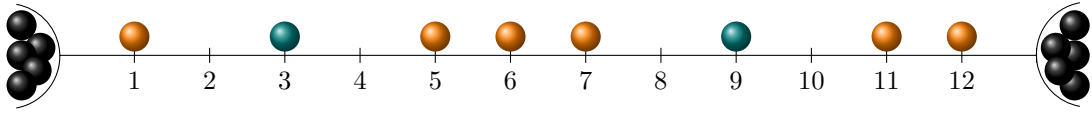
\begin{figure}
	\centering
	\begin{tikzpicture}
		\node[color=white] at (0,1.5) {nothing};
		\draw [-|] (0,0)--(13,0);
		\foreach \x in {1,...,12} \draw (\x,-.1)node[below]{\x} -- (\x,0.1);
		\shade[ball color=orange](1,.25) circle (0.2);
		\shade[ball color=teal](3,.25) circle (0.2);
		\shade[ball color=orange](5,.25) circle (0.2);
		\shade[ball color=orange](6,.25) circle (0.2);
		\shade[ball color=orange](7,.25) circle (0.2);
		\shade[ball color=teal](9,.25) circle (0.2);
		\shade[ball color=orange](11,.25) circle (0.2);
		\shade[ball color=orange](12,.25) circle (0.2);
		\draw (-0.57,-0.7) arc (-80:80:0.7);
		\shade[ball color=black](-0.5,0.4) circle (0.2);
		\shade[ball color=black](-0.25,0.1) circle (0.2);
		\shade[ball color=black](-0.5,0) circle (0.2);
		\shade[ball color=black](-0.3,-0.2) circle (0.2);
		\shade[ball color=black](-0.5,-0.4) circle (0.2);
		\draw (13.57,0.7) arc (100:260:0.7);
		\shade[ball color=black](13.5,0.4) circle (0.2);
		\shade[ball color=black](13.5,0) circle (0.2);
		\shade[ball color=black](13.5,-0.4) circle (0.2);
		\shade[ball color=black](13.25,0.1) circle (0.2);
		\shade[ball color=black](13.3,-0.2) circle (0.2);
	\end{tikzpicture}
	\caption{An ergodic configuration and its active particles (in orange).}
	\label{fig: active particles}
\end{figure}
\medskip
Throughout, given a set $B\subset \Z$, we define the ergodic component  on $B$, denoted by \nota{$\mathcal{E}_B$} as the set of configurations on $B$ where two neighbouring sites in $B$ contain at least one particle, namely
\eq{eq:ergB}{\mathcal{E}_B:=\Big\{ \eta\in \{0,1\}^B : \eta_x+\eta_{x+1}\ge 1\mbox{ for any }\{x,x+1\}\subset B\Big\}.}
The presence of reservoirs which can create and destroy particles on both sides prevents the system from evolving towards frozen configurations,  since the reservoirs are always able to create active particles at the boundaries, even in a frozen configuration.  In particular, the boundary-driven FEP almost surely ultimately reaches the ergodic component
\begin{equation}\label{def: E_N}
\mathcal{E}_N :=\mathcal{E}_{\Lambda_N}=\Big\{ \eta\in \Omega_N\; : \; \; \eta_x+\eta_{x+1}\ge 1, \;\;\; \forall x\in \integers{1}{N-2} \Big\},
\end{equation} 
however in this case where boundaries dynamics are involved, no sharp estimate on the transience time is available so far.
In fact, this is why the additional constraint of the reservoirs which can absorb particles only if the neighbouring site is occupied is very important: this ensures that the ergodic component remains stable under the dynamics. More precisely, we prove in Appendix \ref{sec:appergodic} that the FEP is irreducible on $\mathcal{E}_N$, meaning that two ergodic configurations can be linked by a series of particle jumps/creations/annihilations. As a consequence,  the generator $\gene_N$ has a unique stationary measure $\bar\mu^N$ which is concentrated on the ergodic component $\mathcal{E}_N$. We will see in Section \ref{sec: stationary measures} a more precise local description of this stationary state.

\subsubsection{Grand-canonical measures on $\Z$ and active density field}
\label{sec:grancan}
Let us now recall the \emph{grand-canonical} measures for the facilitated exclusion process in the supercritical phase,  that have been studied in details in \cite[Section 6.2]{blondel2020hydrodynamic}. There exists a collection of supercritical reversible probability distributions $(\pi_\rho )_{\frac{1}{2}<\rho \le 1}$ for the FEP on the infinite line $\Z$, driven by the generator
\begin{equation}
	\label{def:geninf}
	\mathcal{L}_\infty f(\eta ) = \sum_{x\in \Z} c_{x,x+1}(\eta ) \big[ f(\eta^{x,x+1})-f(\eta )\big] ,
\end{equation}
with the same jump rates given by \eqref{eq:bulkjumprate}.
Those measures are translation invariant, and have support on the (infinite volume) ergodic component ${\mc E}_\Z$ (see \eqref{eq:ergB}). Let us fix, for $\ell\geq 1$, a box $B_\ell :=\llbracket 1,\ell \rrbracket $. Then, given a local configuration $\sigma\in\lbrace 0,1\rbrace^{B_\ell}$, the grand-canonical states for the FEP are defined by their local marginals
\begin{equation}
	\label{def:pi_rho}
	\pi_\rho\big( \eta_{\mid B_\ell}=\sigma\big) = (1-\rho ) (1-\act(\rho))^{\ell -1-p}\act(\rho)^{2p-\ell+1-\sigma_1-\sigma_\ell}\ind\{\sigma\in\mathcal{E}_{B_\ell}\},
\end{equation}
where $p=\nota{|\sigma|:=\sum_{x=1}^{\ell}\sigma_x}$ is $\sigma$'s number of particles in $B_\ell$, and \nota{$\act (\rho)$} is the density of active particles (or \emph{active density}), defined as follows: it is the increasing function $\act : [\frac12,1]\longrightarrow [0,1]$  given  by
	\begin{equation}\label{eq:act}
		\act(\rho)=\frac{2\rho-1}{\rho}.
	\end{equation}
The name \emph{active density} will become more clear in the next paragraph.

In practice, this formula is not very convenient in applications, because it describes the distribution $\pi_\rho$ ``globally'' in a fixed  box, rather than sequentially. For this reason, we give the following interpretation of $\pi_\rho$: we set $\eta_0\sim \mathrm{Ber}(\rho)$, and we define two Markov chains started from $\eta_0$, with the same transition probabilities, but the first one, denoted by $(\eta_{x})_{x\geq 0}$, goes \emph{forward from the origin}, while the second one denoted by $(\eta_{-x})_{x\geq 0}$, goes \emph{backward from the origin} (and, once $\eta_0$ is chosen, they evolve independently of each other). More precisely we have, for any $x \geq 0$, 
\begin{equation}\label{transition grand-canonical}
	\pi_\rho \big(\eta_{x+1}=1\big|\eta_x=1\big) = \act(\rho)\qquad\mbox{ and }\qquad \pi_\rho\big(\eta_{x+1}=1\big|\eta_x=0\big) =1,
\end{equation}
and similarly for the backward chain. As expected, this Markovian construction, starting from an arbitrary site, only charges the ergodic component since as soon as a site is empty,  the next one is occupied with probability $1$. It is then straightforward to check that the resulting chain $(\eta_x)_{x\in \Z}$ has local marginals given by \eqref{def:pi_rho}, so that its distribution is indeed $\pi_\rho$. 

Finally, note that the function  $\act(\rho)$ defined in \eqref{eq:act} is indeed the \emph{active density} under $\pi_\rho$, since one can easily check that: for any $x \in \Z$, 
\begin{equation}\label{eq:relationdensities_pirho}
	\pi_\rho \big( h_x(\eta )=1\big) = \act(\rho)
\end{equation}
where $h_x$ was defined in \eqref{def:hx} as the indicator function that $x$ is occupied by an active particle.

Finally, for future reference we denote the inverse function  $\nota{\bar\rho} : [0,1]\longrightarrow [\frac12,1]$  given by
\begin{equation}
	\label{def:rhobar}
	\bar\rho (a)=\act^{-1}(a)=\frac{1}{2-a}.
\end{equation}

\subsection{Main result}
In this section we present the main results of this paper. First, let \nota{$T>0$} be an arbitrary fixed time horizon. Given a probability measure $\mu$ on $\Omega_N$, we denote by \nota{$\mathbb{P}_\mu$} the distribution on the Skorokhod space $\mathcal{D}([0,T],\Omega_N)$ of the process $(\eta (t))_{t\ge 0}$ driven by the \emph{diffusively accelerated} generator $N^2\mathcal{L}_N$, and with initial distribution $\mu$. We denote by \nota{$\E_\mu$} the corresponding expectation. Note that, even though the process $\eta (t)$ strongly depends on $N$, \textit{via} the timescale and the state space, this dependence does not appear in our notation, for the sake of clarity. 

The main result of this paper consists in proving that the empirical density associated with the configuration of particles converges in the diffusive timescale, as $N\to+\infty$, towards a density profile which is solution to some hydrodynamic equation with suitable boundary conditions, as we now explain.

\subsubsection{Initial distribution}
\label{sec:nu0N}

Although we strongly conjecture that our main result holds in a fairly general setting, in order to focus on the main technical challenges we consider in this article the case of a FEP starting from a \emph{supercritical} ergodic configuration $\eta$, and 
we now explain how to construct our initial distribution. We consider in this paper an initial probability law $\mu\equiv \nu_0^N$ whose support is included in  the ergodic component $\mathcal{E}_N$, and which also fits a given initial density profile $\rho^\mathrm{ini}:[0,1]\to(\frac12,1]$ assumed to be supercritical and continuous.

 In order to construct $\nu_0^N$, we first define the discrete \emph{active density field} associated with $\rho^{\rm ini}$, as follows:
\begin{equation}
	\label{eq:axini}
	a_x^\mathrm{ini} := \frac{\rho^\mathrm{ini}\big( \frac{x}{N}\big) +\rho^\mathrm{ini}\big( \frac{x-1}{N}\big) -1}{\rho^\mathrm{ini}(\frac{x-1}{N}\big)}, \qquad x\in \llbracket2,N-1\rrbracket.
\end{equation}
Note that, as its name suggests, $a_x^\mathrm{ini}$ is close to $\act (\rho^\mathrm{ini}(\frac xN ))$ as $N\to\infty$, with $\act(\cdot)$ defined in \eqref{eq:act}.  

\begin{definition} \label{de:initial} \upshape
	The \emph{initial condition} $\nota{\nu_0^N}$ is the probability distribution on $\Omega_N$ given by the law of an inhomogeneous Markov chain $(\eta_x)_{x\in\Lambda_N}$ with state-space $\{0,1\}$, started from $\eta_1\sim \mathrm{Ber}(\rho^\mathrm{ini}(\frac1N))$, and with transition probabilities 
\begin{equation} \label{eq:transition}
	\nu_0^N\big( \eta_{x+1}=1\big|\eta_x=1\big) =a_{x+1}^\mathrm{ini} \quad\mbox{ and }\quad \nu_0^N\big( \eta_{x+1}=1\big|\eta_x=0\big) =1,
\end{equation}
for any $x \in \llbracket 1, N-2 \rrbracket$. 
\end{definition}

Note that under the transition probabilities \eqref{eq:transition}, an empty site is followed by a particle with probability $1$, so that the support of $\nu_0^N$ is included in the ergodic component $\mathcal{E}_N$, as we wanted. Furthermore, by the Markov property and induction it is immediate to check that for any $x \in \Lambda_N$, 
\begin{equation}
	\nu_0^N(\eta_x=1)=\rho^\mathrm{ini}\Big(\frac{x}{N}\Big).
\end{equation}
We will prove in Appendix \ref{Appendix:CorrDecay} that under $\nu_0^N$, spatial correlations decay exponentially (cf.~\eqref{eq:corrdecay}), therefore by the law of large numbers, $\nu_0^N $ fits the macroscopic profile $\rho^{\rm ini}$, in the sense that for any smooth function $G$ on $[0,1]$
\begin{equation}
	\label{eq:initialprofile}
	\frac{1}{N}\sum_{x=1}^{N-1} G\Big(\frac{x}{N}\Big)\eta_x\xrightarrow[N\to +\infty]{}\int_0^1 G(u)\rho^{\rm ini}(u) \diff u
\end{equation}
in $\nu_0^N$--probability.

\subsubsection{Hydrodynamic equations}

Before stating our main result, we introduce some notations and definitions, starting with the spaces of functions that we will use:
\begin{itemize}
	\item $\mathcal{C}^{k,\ell}\big( [0,T]\times [0,1]\big)$ is the space of functions \[G:(t,u)\in [0,T]\times [0,1]\longmapsto G(t,u)= \nota{G_t(u)}\in\R\] which are of class $\mathcal{C}^k$ with respect to the time variable, and of class $\mathcal{C}^\ell$ with respect to the space variable.
	\item $\mathcal{C}_c^{k,\ell}\big( [0,T]\times (0,1)\big)$ is the space of functions $G\in \mathcal{C}^{k,\ell}\big( [0,T]\times [0,1]\big )$ such that, for any time $t\in [0,T]$, the function $G_t:[0,1]\longrightarrow\R$ has compact support included in $(0,1)$.
	\item $\mathcal{H}^1$ is the Sobolev space of locally integrable functions $g:[0,1]\longrightarrow\R$ such that there exists a function $\partial_ug\in L^2\big( [0,1]\big)$ for which $\langle \partial_ug,\varphi\rangle = -\langle g,\partial_u\varphi\rangle$ for all $\varphi\in\mathcal{C}_c^\infty \big( (0,1)\big)$. We endow it with the norm
	\begin{equation*}
		\| g \|_{\mathcal{H}^1} := \Big( \| g\|_{L^2}^2 + \|\partial_ug\|_{L^2}^2\Big)^{1/2}.
	\end{equation*}
	We also let $\mathcal{H}_0^1$ be the closure of $\mathcal{C}_c^\infty\big( (0,1)\big)$ with respect to the topology of this norm.
	\item $L^2\big( [0,T],\mathcal{H}^1\big)$ is the space of measurable functions $G:[0,T]\longrightarrow\mathcal{H}^1$ such that
	\begin{equation*}
		\| G\|_{L^2([0,T],\mathcal{H}^1)}^2 := \int_0^T \| G_t\|_{\mathcal{H}^1}^2\diff t <+\infty .
	\end{equation*}
	\item The usual inner product of $L^2\big( [0,T]\times [0,1]\big)$ is denoted by
	\begin{equation*}
		\dlangle G,H\drangle := \int_0^T\langle G_s,H_s\rangle\diff s, \qquad \forall G,H\in L^2\big( [0,T]\times [0,1]\big).
	\end{equation*}
\end{itemize}

We are now ready to define the notions of solutions to the hydrodynamic equations that will be derived for the boundary-driven FEP. Recall that we have fixed  a supercritical continuous initial profile $\rho^\mathrm{ini} : [0,1] \longrightarrow \big(\frac12,1\big]$, and recall the definitions \eqref{eq:act} and \eqref{def:rhobar} for $\act(\cdot)$ and its inverse.

\begin{definition}[Weak solution with Dirichlet boundary conditions]\label{defin:weaksolDirichlet}\upshape
	Let $\rho_-,\rho_+\in (\frac12,1]$ be two \emph{supercritical} boundary densities. We say that a measurable function $\rho : [0,T]\times [0,1]\longrightarrow [0,1]$ is a \emph{weak solution} to the following \emph{fast diffusion equation with Dirichlet boundary conditions}, and initial condition $\rho^\mathrm{ini}$
	\begin{equation}
		\label{eq:fast_diffusion_equation_Dirichlet}
		\begin{cases}
		\partial_t\rho = \partial_u^2\act(\rho )  & \mbox{ on } [0,T]\times [0,1], \\ 
		\rho_0(\cdot )=\rho^\mathrm{ini}(\cdot ), &  \\ 
		\rho_t (0)=\rho_-,  \;\; \rho_t (1)=\rho_+  & \mbox{ for all }t\in [0,T],
		\end{cases}
	\end{equation}
	if the following three conditions are satisfied:
	\begin{itemize}
		\item[(i)] $\mathfrak{a}(\rho)\in L^2\big( [0,T], \mathcal{H}^1\big)$ ;
		\item[(ii)] for any $t\in [0,T]$, and any test function $G\in\mathcal{C}_c^{1,2}\big([0,T]\times (0,1)\big)$, we have
		\begin{equation}
		\label{eq:weakformulation_Dirichlet}
			\langle \rho_t,G_t\rangle - \langle\rho^\mathrm{ini},G_0\rangle - \int_0^t\langle \rho_s,\partial_tG_s\rangle\diff s = \int_0^t \big\langle\act (\rho_s),\partial_u^2G_s\rangle\diff s \; ; 
		\end{equation}
		\item[(iii)] for all $t\in [0,T]$, $\rho_t(0)=\rho_-$ and $\rho_t(1)=\rho_+$.
	\end{itemize}
\end{definition}

\begin{definition}[Weak solution with Robin boundary conditions]\label{defin:weaksolRobin}\upshape
	We say that a measurable function $\rho : [0,T]\times [0,1]\longrightarrow [0,1]$ is a \emph{weak solution} to the following \emph{fast diffusion equation with Robin boundary conditions} and initial condition $\rho^\mathrm{ini}$
	\begin{equation}
		\label{eq:fast_diffusion_equation_Robin}
		\begin{cases}
		\partial_t\rho = \partial_u^2\act(\rho )  & \mbox{ on } [0,T]\times [0,1], \\ 
		\rho_0(\cdot )=\rho^\mathrm{ini}(\cdot ), &  \\ 
		\partial_u\act (\rho_t)(0)=\kappa\big( \act (\rho_t(0))-\alpha\big) & \mbox{ for all }t\in [0,T],\\
		\partial_u\act (\rho_t)(1)=\kappa\big(\beta -\act (\rho_t(1))\big) & \mbox{ for all }t\in [0,T],
		\end{cases}
	\end{equation}
	if the following two conditions hold :
	\begin{itemize}
		\item[(i)] $\mathfrak{a}(\rho)\in L^2\big( [0,T], \mathcal{H}^1\big)$ ;
		\item[(ii)] for any $t\in [0,T]$ and any test function $G\in\mathcal{C}^{1,2}\big([0,T]\times [0,1]\big)$, we have
		\begin{multline}
		\label{eq:weakformulation_Robin}
			\langle\rho_t,G_t\rangle -\langle\rho^\mathrm{ini},G_0\rangle - \int_0^t \langle\rho_s,\partial_tG_s\rangle\diff s \\= \int_0^t\big\langle \act(\rho_s),\partial_u^2G_s\big\rangle\diff s  -\int_0^t\Big\lbrace \act(\rho_s(1))\partial_uG_s(1)-\act(\rho_s(0))\partial_uG_s(0)\Big\rbrace\diff s \\+ \kappa\int_0^t\Big\lbrace \big( \beta - \act(\rho_s(1))\big) G_s(1)-\big(\act (\rho_s(0))-\alpha\big)G_s(0)\Big\rbrace\diff s.
		\end{multline}
	\end{itemize}
\end{definition}

\begin{definition}[Weak solution with Neumann boundary conditions]\label{defin:weaksolNeumann}\upshape
	We say that a measurable function $\rho : [0,T]\times [0,1]\longrightarrow [0,1]$ is a \emph{weak solution} to the following \emph{fast diffusion equation with Neumann boundary conditions} and initial condition $\rho^\mathrm{ini}$
	\begin{equation}
		\label{eq:fast_diffusion_equation_Neumann}
		\begin{cases}
		\partial_t\rho = \partial_u^2\act(\rho )  & \mbox{ on } [0,T]\times [0,1], \\ 
		\rho_0(\cdot )=\rho^\mathrm{ini}(\cdot ), &  \\ 
		\partial_u\act (\rho_t)(0)=\partial_u\act(\rho_t)(1)=0 & \mbox{ for all }t\in [0,T],
		\end{cases}
	\end{equation}
	if $\mathfrak{a}(\rho)\in L^2\big( [0,T], \mathcal{H}^1\big)$ and equation \eqref{eq:weakformulation_Robin} with $\kappa =0$ is satisfied for any $t\in [0,T]$ and any test function $G\in\mathcal{C}^{1,2}\big([0,T]\times [0,1]\big)$.
\end{definition}

\begin{remark}\upshape
	All the partial differential equations above admit a unique weak solution for their respective notion of solutions. We prove it in Appendix \ref{sec:appuniqueness}.
\end{remark}

\subsubsection{Hydrodynamic limits}
\label{sec:HDLstatement}

We are now ready to state our main result.

\begin{theorem}[Hydrodynamic limit for the boundary-driven FEP]
	\label{thm: hydrodynamic limit}
	Let $\rho^{\rm ini}:[0,1]\to(\frac12,1]$ be a continuous initial profile and recall the initial distribution $\nu_0^N$ defined in Section \ref{sec:nu0N}, associated with $\rho^{\rm ini}$. 
	
	Then, for all continuous function $G:[0,1]\longrightarrow\R$,  all $\delta >0$ and all $t \in [0,T]$ we have
	\begin{equation}\label{lim Hydrodynamic limit}
	\lim_{N\to +\infty}\Prob_{\nu_0^N} \left( \bigg|\frac{1}{N}\sum_{x\in\Lambda_N}G\Big( \frac{x}{N}\Big)\eta_x(t)-\int_0^1G(u)\rho_t(u)\diff u\bigg|>\delta\right) =0
	\end{equation}
	where $\rho :(t,u) \in [0,T] \times [0,1] \longmapsto \rho_t(u)\in [0,1]$ is the unique weak solution of
	\begin{itemize}
		\item if $\theta <1$, the fast diffusion equation \eqref{eq:fast_diffusion_equation_Dirichlet} with Dirichlet boundary conditions
		\begin{equation}\label{eq:DirichletBC}
			\rho_- =\bar\rho (\alpha )=\frac{1}{2-\alpha} \qquad\mbox{ and }\qquad \rho_+=\bar\rho (\beta ) =\frac{1}{2-\beta}
		\end{equation}
		in the sense of Definition \ref{defin:weaksolDirichlet} ;
		\item if $\theta =1$, the fast diffusion equation with Robin boundary conditions \eqref{eq:fast_diffusion_equation_Robin} in the sense of Definition \ref{defin:weaksolRobin} ;
		\item if $\theta >1$, the fast diffusion equation with Neumann boundary conditions \eqref{eq:fast_diffusion_equation_Neumann} in the sense of Definition \ref{defin:weaksolNeumann}.
	\end{itemize}
\end{theorem}

\begin{figure}[hbtp]
	\begin{tikzpicture}
	  \draw [thick, blue] (-3,0)--(1,0);
	  \draw [thick, green,->] (1,0)--(4,0);
	  \draw [very thick, red] (1,-.1)--(1,.1);
	  \draw [very thick, blue] (0,-.1)--(0,.1);
	  \draw (4.2,-.1) node[anchor=north] { $\theta$};
	  \draw (0,-.1) node[anchor=north,blue] { $0$};
	  \draw (1,-.1) node[anchor=north,red] { $1$};
	  \draw (-1.5,0) node[anchor=south,blue] {Dirichlet};
	  \draw (2.5,0) node[anchor=south,green] {Neumann};
	  \draw (-.1,.7) node[anchor=south,red] {Robin};
	  \draw [thick, red, ->] (.5,.8)--(1,.2);
	\end{tikzpicture}
	\caption{Diagram of the boundary conditions imposed by each value of $\theta$}
  \end{figure}
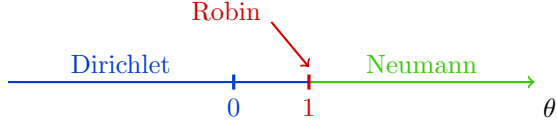

To conclude this section we briefly explain the important classical ingredients of the proof of Theorem \ref{thm: hydrodynamic limit}. Define as usual the empirical measure on $[0,1]$
\begin{equation}\label{def: empirical measure}
m_t^N(\mathrm{d} u) =\frac{1}{N}\sum_{x\in\Lambda_N}\eta_x(t) \delta_\frac{x}{N}(\mathrm{d} u)
\end{equation}
where $\delta_y$ stands for the Dirac mass at $y\in [0,1]$.  Endowing the space $\mathcal{M}_+$ of non-negative measures on $[0,1]$ with the topology of weak convergence of measures,  we see from \eqref{eq:initialprofile} that for our choice of initial distribution, 
$m_0^N$ converges to 
$ \rho^\mathrm{ini}(u)\diff u$ in probability as $N\to+\infty$.
Proving the hydrodynamic limit amounts to showing that 
\[m_t^N(\mathrm{d}u) \xrightarrow[N\to +\infty]{} \rho_t(u)\diff u\]
 in probability for all $t\in [0,T]$,  where $\rho$ is given in Theorem \ref{thm: hydrodynamic limit}.

See the empirical measure $m^N$ as a mapping from $\mathcal{D}\big( [0,T], \Omega_N\big)$ to $\mathcal{D}\big( [0,T],\mathcal{M}_+\big)$, and denote by  $\nota{\Q^N:=\Prob_{\nu_0^N}\circ (m^N)^{-1}}$ the pushforward distribution on $\mathcal{D}\big( [0,T],\mathcal{M}_+\big)$ of the empirical measure's trajectory, corresponding to its law.  The strategy of the proof is the following:
\begin{enumerate}
	\item First,  we prove that the sequence $(\Q^N)_{N\ge 1}$ is tight so that we can consider a limit point $\Q$,  which can be seen as the law of a random variable $m$ with values in $\mathcal{D}\big( [0,T],\mathcal{M}_+\big)$.
	\item Then,  we prove that $\Q$ is concentrated on trajectories of measures which are absolutely continuous with respect to the Lebesgue measure.  This implies that, $\Q$-almost surely~$m$ writes as $m_t(\mathrm{d} u)=\rho_t(u)\diff u$ for some density profile $\rho$.
	\item Finally, we show that $\rho$ is a weak solution to the hydrodynamic equation corresponding to the value of $\theta$ in Theorem \ref{thm: hydrodynamic limit}. By the uniqueness of weak solutions,  we deduce that the sequence $(\Q^N)_{N\ge 1}$ admits a unique limit point,  which is concentrated on the trajectory $\big( t\longmapsto \rho_t(u)\diff u\big)$ whose density is the unique weak solution of the expected hydrodynamic equation.  It proves that the random variables $(m^N)_{N\ge 1}$ converge in distribution to the trajectory $\big( t\longmapsto \rho_t(u)\diff u\big)$,  and therefore in probability since this limit is deterministic. 
\end{enumerate}

Although points (1) and (2) in our context follow straightforwardly from classical arguments \cite[Chapter 5, Section 1]{KL}, point (3) above is very delicate in general, and is tackled  here using  Guo,  Papanicolaou and Varadhan's \emph{entropy method} \cite{guo1988nonlinear}.  This requires understanding the local invariant measure of the process, in particular at the boundaries.  

In the next section, we describe the stationary states for the boundary-driven FEP.

\section{Stationary states}
\label{sec: stationary measures}

\subsection{The equilibrium case $\alpha =\beta$}

We assume in this section that $\alpha =\beta$. From Proposition \ref{prop: ergodic component irreducible} in the appendix, the ergodic component $\mathcal{E}_N$ is irreducible for the dynamics of the boundary-driven FEP. As a consequence, this process admits a unique stationary state \nota{$\mu_\alpha^N$}, and it turns out that it is the restriction to the box $\Lambda_N$ of the grand-canonical state whose density is the one imposed by the reservoirs, namely $\bar\rho (\alpha )$. We state and prove it in the following result.

\begin{proposition}
	In the equilibrium case $\alpha =\beta$, the unique stationary measure $\mu_\alpha^N$ of the boundary-driven FEP is nothing but the restriction of the measure $\pi_{\bar\rho (\alpha )}$ to the box $\Lambda_N$, \textit{i.e.}
	\[\forall \sigma \in\Omega_N, \qquad \mu_\alpha^N (\sigma )=\pi_{\bar\rho (\alpha )}(\eta_{|\Lambda_N}=\sigma ).\]
\end{proposition}

\begin{proof}
We will check that this measure is reversible with respect to both the bulk dynamics, and the boundary dynamics. Consider a jump occurring inside the bulk over an edge $\lbrace x,x+1\rbrace$ that does not touch the boundaries (\textit{i.e.}~$x\neq 1,N-2$). This means that the local configuration $(\eta_{x-1},\eta_x,\eta_{x+1},\eta_{x+2})$ must be either $\bullets\bulletx\circs\bullets$ or $\bullets\circx\bullets\bullets$ (where $\bullets$ stands for a particle and $\circs$ for an empty site), in order for the configuration to be ergodic, and to satisfy $c_{x,x+1}(\eta )\neq 0$. The probabilities of observing these two local configurations under $\pi_{\bar\rho (\alpha )}$ are respectively $\bar\rho (\alpha)\alpha (1-\alpha)$ and $\bar\rho (\alpha )(1-\alpha )\alpha$. Since both jumps (back and forth) occur at the same rate 1, this proves that this measure is reversible with respect to any FEP jump occurring inside $\integers{2}{N-2}$. The same is true for jumps between sites 1 and 2: the latter can only occur if the local configuration $(\eta_1,\eta_2,\eta_3)$ is ergodic and has non-zero jump rate $c_{1,2}(\eta )\neq 0$, so it must be given by $\bulletone\circs\bullets$ or $\circone\bullets\bullets$. On the one hand, we have
\[ \pi_{\bar\rho (\alpha )}(\bulletone\circs\bullets )\times c_{1,2}(\bulletone\circs\bullets ) = \bar\rho (\alpha )(1-\alpha )\times \alpha,\]
and on the other hand
\[ \pi_{\bar\rho (\alpha)}(\circone\bullets\bullets)\times c_{1,2}(\circone\bullets\bullets)=\big( 1-\bar\rho (\alpha)\big)\alpha\times 1,\]
and these quantities are equal by definition of $\bar\rho (\alpha )$ (recall  \eqref{def:rhobar}). This proves that the measure is reversible with respect to any jump through the edge $\lbrace 1,2\rbrace$, and we can prove the same for the edge $\lbrace N-2,N-1\rbrace$. In other words, $\mu_\alpha^N$ is reversible with respect to the generator $\mathcal{L}_0$, and we only have to prove that it is also reversible with respect to the boundary generators $\mathcal{L}_\ell$ and $\mathcal{L}_r$.

Let us consider only the case of the generator $\mathcal{L}_\ell$ because the generator $\mathcal{L}_r$ can be treated in the exact same way. Once again, to ensure that the boundary rate $b_\ell (\eta )$ is not zero and the configuration is ergodic, the local configuration $(\eta_1,\eta_2)$ must be either $\bulletone\bullets$ or $\circone\bullets$. The probability of the first one is $\bar\rho (\alpha )\alpha$, and that of the second one $1-\bar\rho (\alpha )=(1-\alpha )\bar\rho (\alpha )$. As the transition rate $\bulletone\bullets\mapsto \circone\bullets$ is $\kappa N^{-\theta}(1-\alpha)$, and the transition rate $\circone\bullets\mapsto\bulletone\bullets$ is $\kappa N^{-\theta}\alpha$, the reversibility with respect to the boundary dynamics is proved.
\end{proof}

\begin{remark}\upshape
	This result is analogous to the following well-known result about the SSEP. The grand-canonical measures for the SSEP are known to be the Bernoulli product measures with constant density. In the equilibrium case of the boundary-driven SSEP, the unique stationary measure is the restriction of this grand-canonical measure whose density is the one of the reservoirs.
\end{remark}

\subsection{The non-equilibrium case $\alpha\neq\beta$} 
We now get back to the general case where $\alpha\neq\beta$. As in the equilibrium case, there exists a unique stationary measure \nota{$\bar\mu^N$} but due to the presence of long-range correlations (already present in the non-equilibrium SSEP), we have no explicit expression for it. 

Though, we can get some information on $\bar\mu^N$ by simple considerations. First of all, for any $\theta\in\R$, let   \nota{$\rho_\theta^{ss}$} be the \emph{stationary solution} of the hydrodynamic   equation corresponding to the value of $\theta$, namely the fast diffusion equation given in Definition \ref{defin:weaksolDirichlet} if $\theta < 1$, Definition \ref{defin:weaksolRobin} if $\theta=1$ and Definition \ref{defin:weaksolNeumann} if $\theta >1$. Simple computations show that this stationary solution satisfies
\begin{equation}\label{eq:stationarysol_PDE}
	\forall u\in [0,1],\qquad \act\big(\rho_\theta^{ss}(u)\big)=\begin{cases}
		\alpha +(\beta -\alpha )u & \mbox{ if }\theta <1,\\
		\alpha +(\beta -\alpha )\dfrac{\kappa u+1}{\kappa +2} & \mbox{ if }\theta =1,\\
		\dfrac{\alpha +\beta}{2} &\mbox{ if }\theta >1.
	\end{cases}
\end{equation}
The following result states that the active density field under the stationary state is close to this quantity.
\begin{lemma} The active density field under the non-equilibrium stationary measure $\bar\mu^N$ given by 
	$\bar a_x:=\bar \mu^N\big(h_x(\eta)\big)$, $x\in\Lambda_N$, satisfies
	\begin{equation} \label{approx_abar} \sup_{x\in\Lambda_N} \Big| \bar a_x - \act\big(\rho_\theta^{ss}(\tfrac x N)\big) \Big| \xrightarrow[N\to\infty]{}0.\end{equation}
\end{lemma}

\begin{proof} Recall the definition of the current $j_{x,x+1}(\eta )$ in \eqref{def:current}, of $h_x(\eta )$ in \eqref{def:hx}, and of the gradient condition \eqref{eq:gradientcondition}. Under the stationary state, we can make the following simple computation
\begin{equation}
\label{eq:gradh0}
0=\bar\mu^N\big(\mathcal{L}_N\eta_x\big) = \bar\mu^N\big( j_{x-1,x}(\eta )-j_{x,x+1}(\eta )\big)= \bar\mu^N\big( h_{x-1}(\eta )-2h_x(\eta )+h_{x+1}(\eta )\big)
\end{equation}
that holds for any $x\in \integers{2}{N-2}$. At the boundaries, the same reasoning gives
\begin{align*} \frac{\kappa}{N^\theta}\big\{\bar\mu^N(h_0(\eta)-h_1(\eta))\big\}&=\bar\mu^N(h_1(\eta)-h_2(\eta)),\\
	\frac{\kappa}{N^\theta}\big\{\bar\mu^N(h_{N-1}(\eta)-h_N(\eta))\big\}&=\bar\mu^N(h_{N-2}(\eta)-h_{N-1}(\eta)).
	\end{align*}
These relations, together with the convention \eqref{def:conventions2} are sufficient to get an explicit expression of the active density field, we find
\begin{equation}\label{eq:barax}
	\forall x\in\Lambda_N,\qquad \bar{a}_x = \alpha +(\beta -\alpha )\frac{\kappa (x-1)+N^\theta}{\kappa (N-2)+2N^\theta}
\end{equation}
One can then easily conclude the proof. \end{proof}

\begin{remark}\upshape
	We expect the stationary state $\bar\mu^N$ to look locally like a grand-canonical state, and in particular, if we define the density field $\nota{\bar\rho_x^N} = \bar\mu^N(\eta_x)$ under this measure, we expect a relation like \eqref{eq:relationdensities_pirho} to hold locally, \textit{i.e.} 
\begin{equation}
	\bar\rho_x^N \approx \frac{1}{2-\bar{a}_x},
\end{equation}
and from  \eqref{approx_abar} the latter is close to $\rho^{ss}_\theta\big(\frac xN\big)$. However we are unable to prove it. In Figure \ref{fig: simu alpha neq beta} we plot a numerical simulation of the density field $\bar\rho_x^N$ and of the profile $\rho_\theta^{ss}$ in the case $\theta =0$, and we can see that both are very close to each other, confirming our prediction.

\begin{figure}[hbtp]
	\centering
	\includegraphics[scale=.5]{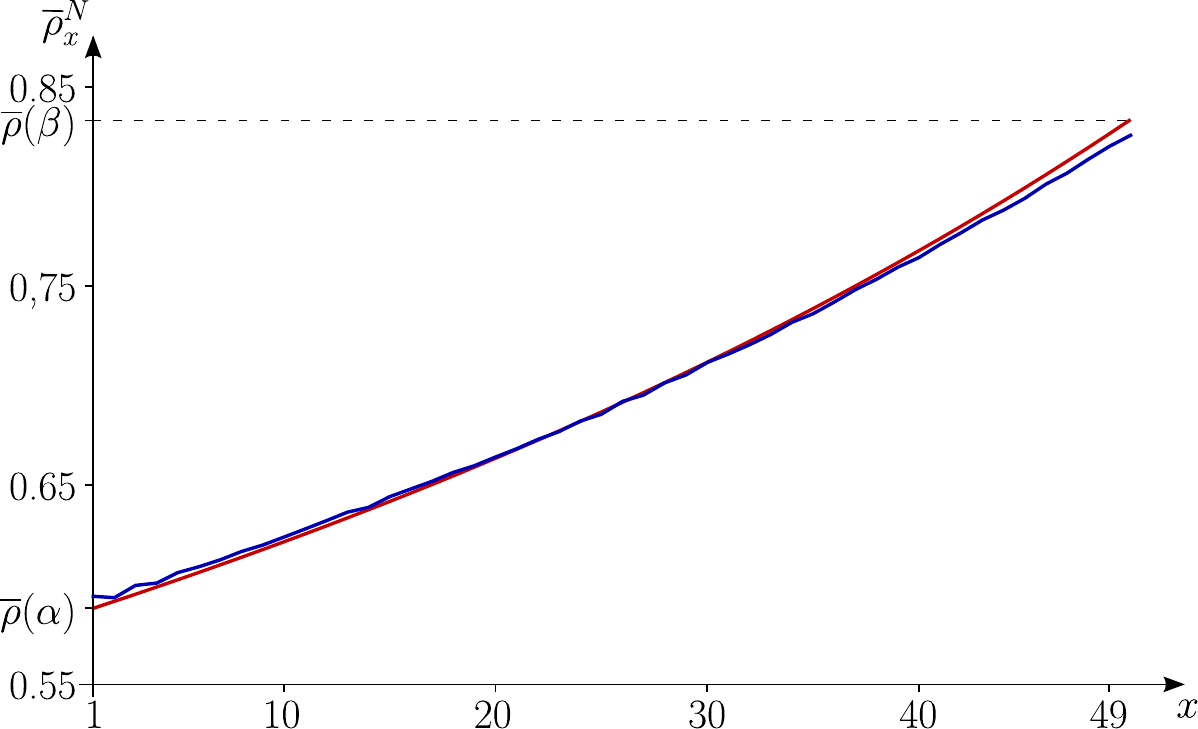}
	\caption{Numerical simulation of the density field $\bar\rho^N_x$ (in blue), and plot of the stationary profile $\rho_0^{ss}$  (in red) for $N=50$,  $\alpha =0.3$, $\beta= 0.8$ and $\theta =0$.}
	\label{fig: simu alpha neq beta}
\end{figure}

\end{remark}

Since we do not have any further information about this stationary measure $\bar\mu^N$, the next section is dedicated to constructing an explicit reference measure $\mu^N$  relying on the Markovian construction of the grand-canonical measures, and on the fact that the active density field is affine under the stationary state (see \eqref{eq:barax}). This reference measure will be crucial later on to prove the replacement lemmas in Sections \ref{sec:replacement_boundary} and \ref{sec: Replacement lemma bulk}, which are at the center of the proof of Theorem \ref{thm: hydrodynamic limit}.

\section{Reference measure and Dirichlet form estimates}
\label{sec:approx}

\subsection{Construction of a reference measure $\mu^N$}

First, recall that in the transition probabilities \eqref{transition grand-canonical} in the Markovian construction of the grand-canonical state $\pi_\rho$, the probability that an occupied site follows another occupied site is equal to the density of active particles $\mathfrak{a}(\rho)$. This is quite intuitive since if site $x$ is occupied, then a particle at site $x+1$ is automatically active. Besides, we know from \eqref{eq:barax} that under the stationary measure, the active density is affine. 

From these observations, we are now constructing a measure $\mu^N$ being the law of an inhomogeneous Markov chain started at a Bernoulli random variable with parameter $\bar\rho (\alpha )$, and using a suitable active density field for the transition probabilities. In particular, we will ensure that jumps inside the bulk are ``quasi-reversible'' (in the sense of Lemma \ref{lem:quasi-reversibility} below), and that exchanges with the reservoirs are reversible. Thanks to these properties, we will be able to use the measure $\mu^N$  as a reference measure (in the entropy method developed in the next sections) for any value of $\theta$.

More precisely, let us define the active density field
\begin{equation}
	\label{def:ax}
	\forall x\in \Lambda_N ,\qquad a_x = \frac{\beta -\alpha}{N-2}(x-1)+\alpha = \varepsilon_N(x-1)+\alpha\qquad\mbox{ with } \qquad \varepsilon_N=\frac{\beta -\alpha}{N-2},
\end{equation}
and let $\mu^N$ be the measure under which $(\eta_x)_{1\le x\le N-1}$ is an inhomogeneous Markov chain on $\lbrace 0,1\rbrace$, starting at $\eta_1\sim\mathrm{Ber}(\bar\rho (\alpha))$ and with transition probabilities
\begin{equation}
	\label{eq:transitionsmuN}
	\mu^N(\eta_{x+1}=1|\eta_x=1)=a_{x+1}\quad\mbox{ and }\quad \mu^N(\eta_{x+1}=1|\eta_x=0)=1\quad\mbox{ for }x\in\integers{1}{N-2}.
\end{equation}
This procedure builds a measure $\mu^N$ which is concentrated on the ergodic component $\mathcal{E}_N$ since as soon as a site is empty, the next one is occupied. We obtain the following explicit formula for the measure $\mu^N$
\begin{equation}
	\label{eq:muN_explicit}
	\mu^N(\eta )= \bar\rho (\alpha )^{\eta_1}\big( 1-\bar\rho (\alpha)\big)^{1-\eta_1}\times\prod_{x=2}^{N-1}\Big(\eta_{x-1}a_x^{\eta_x}(1-a_x)^{1-\eta_x} + (1-\eta_{x-1})\eta_x\Big)
\end{equation}
for all $\eta\in\Omega_N$. Let \nota{$\rho_x^N:=\mu^N(\eta_x=1)$} be the density profile under the measure $\mu^N$. By the Markovian construction and the relation \eqref{eq:transitionsmuN}, it satisfies the recurrence relation
\begin{equation}
	\label{eq:recurrencerelationmuN}
	\rho_1^N=\bar\rho (\alpha )\quad\mbox{ and }\quad \rho_x^N = 1-\rho_{x-1}^N+a_x\rho_{x-1}^N\qquad \mbox{ for all }x\in\integers{2}{N-1}.
\end{equation}

\subsection{Technical estimates on $\mu^N$}

The rest of this section is dedicated to proving some technical results on this reference measure. We first claim that under $\mu^N$, the relation \eqref{def:rhobar} between total and active densities is asymptotically satisfied.

\begin{proposition}
\label{prop: relation rho,rhoact muN}
	There exists a constant $C_0>0$, depending only on $\alpha,\beta$ such that for all $x\in\Lambda_N$, we have
	\begin{equation}
		\left|\rho_x^N -\frac{1}{2-a_x}\right| \le \frac{C_0}{N}, \qquad \text{for all } x \in \Lambda_N.
	\end{equation}
\end{proposition}

\begin{proof}
	For $x\in\Lambda_N$, define
	\[\delta_x:=\rho_x^N - \frac{1}{2-a_x}.\]
	Using \eqref{eq:recurrencerelationmuN}, we obtain that for $x\ge 2$,
	\begin{align*}
        \delta_x = 1-\rho_{x-1}^N+a_x\rho_{x-1}^N-\frac{1}{2-a_x} & = \rho_{x-1}^N(a_x-1)+1-\frac{1}{2-a_x}\\
     & = \delta_{x-1}(a_x-1)+ (a_x-1)\left( \frac{1}{2-a_{x-1}}-\frac{1}{2-a_x}\right)\\
        & = (a_x-1)\left( \delta_{x-1} +\frac{a_{x-1}-a_x}{(2-a_{x-1})(2-a_x)}\right).
    \end{align*}
	Note that: $|a_x-1| \le c$ with $c=|\alpha\wedge\beta -1|<1$, $|a_x-2|\ge 1$ and besides $a_x-a_{x-1}=\varepsilon_N$ defined in \eqref{def:ax}. Therefore we get that
	\[ |\delta_x| \le c|\delta_{x-1}|+c|\varepsilon_N|.\]
	By induction, we deduce that
	\[ |\delta_x|\le c^{x-1}|\delta_1| + |\varepsilon_N| \sum_{k=1}^{x-1}c^k \le c^{x-1}|\delta_1|+ \frac{\varepsilon_N}{1-c}.\]
	But by definition of $\bar\rho (\alpha )$ we have $\delta_1=0$ so only the second term on the right hand side remains, and it is clearly of order $\mathcal{O}(\frac1N)$.
\end{proof}

We now state and prove the property of \emph{local equilibrium} satisfied by $\mu^N$. In other words, $\mu^N$ is close, locally, to the grand-canonical state associated to the local density. For that purpose, when $\ell\ge 1$, we define the box
\begin{equation}
	\label{def:lambdaxell}
	\Lambda_x^\ell = \begin{cases}
		\integers{1}{2\ell +1} & \mbox{ if } x\in \integers{1}{\ell}, \\
		\integers{x-\ell}{x+\ell} & \mbox{ if } x\in \integers{\ell +1}{N-\ell -1},\\
		\integers{N-2\ell -1}{N-1} & \mbox{ if } x\in \integers{N-\ell}{N-1}.
	\end{cases}
\end{equation}
This is a box of size $2\ell +1$ that contains $x$, but is not necessarily centered around it when $x$ is too close to the boundary.

\begin{proposition}
	\label{prop:localequilibrium}
	Fix $\ell\ge 1$, take $u\in (\frac1N ,1)$,  and define $x=\lfloor uN\rfloor$. Then, there exists a constant $C_1=C_1(\alpha ,\beta ,\ell )>0$ (independent of $u$) such that for any local configuration $\sigma\in \lbrace 0,1\rbrace^{2\ell +1}$, we have
	\begin{equation}
		\big|\mu^N(\eta_{|\Lambda_x^\ell}=\sigma )-\pi_{\varrho (u)}(\eta_{|\integers{-\ell}{\ell}}=\sigma)\big| \le \frac{C_1}{N}
	\end{equation}
	where
	\begin{equation}
		\label{def:varrho}
		 \varrho (u):=\frac{1}{2-(\alpha +(\beta -\alpha)u)} =\lim_{N\to+\infty} \rho_x^N
	\end{equation}
	is the approximated local density around the point $x$ under $\mu^N$. 
\end{proposition}

\begin{proof}
	Let $y_-=\min\Lambda_x^\ell$ and $y_+=\max\Lambda_x^\ell$. Both distributions $\mu^N$ and $\pi_{\varrho (u)}$ can be built by a Markovian construction according to \eqref{eq:transitionsmuN}, with the difference that one has constant transition rates, and the other does not. Assuming that two sequences $u_k,v_k\in (0,1)$ satisfy $|u_k-v_k|\le\delta$ for any $k\in\integers{y_-}{y_+}$, then the difference
	\[ \left|\prod_{k=y_-}^{y_+} u_k-\prod_{k=y_-}^{y_+}v_k\right| \le C_\ell \delta\]
	is also of order $\delta$. So by the Markovian construction of both $\mu^N$ and $\pi_{\varrho (u)}$, it is enough to show that 
	\begin{align}
        & \big| \rho_{y_-}^N-\varrho (u)\big| \le \frac{C}{N}, \label{cond1}\\
        & \big| a_y - \frak{a}(\varrho (u))\big| \le \frac CN \qquad \forall y\in [\![y_-+1,y_+]\!], \label{cond2}
    \end{align}
	where $C>0$ is some constant that depends on $\alpha ,\beta ,\ell$, but not on $N$. Writing 
	\[\big| \rho_{y_-}^N-\varrho (u)\big| \le \left| \rho_{y_-}^N -\frac{1}{2-a_{y_-}}\right| + \left| \frac{1}{2-a_{y_-}}-\varrho (u)\right| ,\]
	we see that \eqref{cond1} is satisfied using Proposition \ref{prop: relation rho,rhoact muN} to bound the first term, and the fact that $y_-$ is at distance at most $\ell$ of $x$ to bound the second one. By definition of $a_y$ and $\varrho (u)$, the condition \eqref{cond2} is also clearly satisfied.
\end{proof}

We now claim that the marginal of $\mu^N$ with respect to two distant boxes is roughly a product measure:

\begin{corollary}
	\label{cor:2BEdec}
	Fix $\ell\ge 1$, and take $u,v\in (\frac1N,1)$ such that 
	\begin{equation}
	\label{eq:bounduv}
	v-u\ge \frac{(\log N)^2 +1}{N}.
	\end{equation}
	Defining $x=\lfloor uN\rfloor$ and $y=\lfloor vN\rfloor$, we have $x,y\in\Lambda_N$ with $y-x\ge (\log N)^2$. Then, there exists a constant $C_2=C_2(\alpha ,\beta ,\ell )>0$ such that for any $u,v$ satisfying \eqref{eq:bounduv}, and any local configurations $\sigma ,\sigma'\in\lbrace 0,1\rbrace^{2\ell +1}$, we have
	\begin{equation}
		\big| \mu^N(\eta_{|\Lambda_x^\ell}=\sigma , \eta_{|\Lambda_y^\ell}=\sigma ')-\pi_{\varrho (u)}(\eta_{|\integers{-\ell}{\ell}}=\sigma)\pi_{\varrho (v)}(\eta_{|\integers{-\ell}{\ell}}=\sigma')\big| \le \frac{C_2}{N}.
	\end{equation}
\end{corollary}

\begin{proof}
	It is clear that $(a_x)_{x\in\Lambda_N}$ is bounded away from 0 since $\alpha ,\beta >0$. The proof of this corollary is then straightforward thanks to the decorrelation estimate given in Theorem \ref{thm:decaycorrelations}, which yields, since $y-x >\log N$ by \eqref{eq:bounduv}, that
	\begin{equation}
		|\mu^N(\eta_{|\Lambda_x^\ell}=\sigma, \; \eta_{|\Lambda_y^\ell}=\sigma')-\mu^N(\eta_{|\Lambda_x^\ell}=\sigma)\mu^N(\eta_{|\Lambda_y^\ell}=\sigma')|\leq \frac{C}{N}.
	\end{equation}
	for some constant $C>0$. Then, a simple application of Proposition \ref{prop:localequilibrium} yields the result.
\end{proof}

	\begin{lemma} \label{lem:quasi-reversibility} The measure $\mu^N$ satisfies the following ``quasi-reversibility'' relation in the bulk. For any $x\in\integers{1}{N-2}$ we define	\begin{equation}\label{def:OmegaNx}
			\Omega_N^x :=\{ \eta\in\mathcal{E}_N\; :\; c_{x,x+1}(\eta ) > 0\}.
		\end{equation}
	 Then, there exists a constant $C_3=C_3(\alpha,\beta)>0$ such that, for any $x\in\integers{1}{N-2}$ and any $\eta\in\Omega_N^x$,
		\begin{equation}\label{eq:quasi-reversibility}
			\bigg|1-\frac{c_{x,x+1}(\eta^{x,x+1})}{c_{x,x+1}(\eta )}\frac{\mu^N(\eta^{x,x+1})}{\mu^N(\eta )}\bigg| \leqslant \frac{C_3}N.
		\end{equation}
		Moreover, the measure $\mu^N$ is reversible with respect to the boundary dynamics.
	\end{lemma}
\begin{proof}Denote by $\{\bullets\bulletx \circs\bullets\}$ the event $ \{\eta\in \mathcal{E}_N,\;(\eta_{x-1}, \eta_x, \eta_{x+1}, \eta_{x+2})=(1,1,0,1)\}$, and   similarly $\{\bullets\circx \bullets\bullets\}$ the event $ \{\eta\in \mathcal{E}_N,\;(\eta_{x-1}, \eta_x, \eta_{x+1}, \eta_{x+2})=(1,0,1,1)\}$.  For simplicity, we abuse a little bit our notation and use it also when $x$ is either $1$ or $N-2$, in which case the first/last particle should not be present but rather represents the reservoir. Then, one easily checks that  \[\Omega_N^x=\{\bullets\bulletx \circs\bullets\}\cup \{\bullets\circx \bullets\bullets\}, \qquad \text{and} \qquad \{\bullets\bulletx \circs\bullets\}=\{\eta^{x,x+1},\; \eta\in \{\bullets\circx \bullets\bullets\} \}.\] 
	Let us prove \eqref{eq:quasi-reversibility} in the case where $\eta \in \{\bullets\bulletx \circs\bullets\}$ (the second case where $\eta \in \{\bullets\circx \bullets\bullets \}$ is similar). One can check that, in this case: 
	\begin{equation}\label{eq:cxx} c_{x,x+1}(\eta) = \begin{cases} \alpha & \text{ if } x=1 \\ 1 & \text{ if } x \in \integers{2}{N-2} \end{cases} \qquad c_{x,x+1}(\eta^{x,x+1}) = \begin{cases}
		1 & \text{ if } x \in \integers{1}{N-3} \\ \beta & \text{ if } x =N-2,
	\end{cases}  \end{equation} and besides, recalling that $\bar\rho(\alpha)=(2-\alpha)^{-1}$, the expressions of $a_x$ given in \eqref{def:ax} and the Markov construction, we have
\begin{equation}\label{eq:quotient} \frac{\mu^N(\eta^{x,x+1})}{\mu^N(\eta)} = \begin{cases} \displaystyle \vphantom{\Bigg\{} \frac{1-\bar\rho(\alpha)}{\bar\rho(\alpha)}\frac{a_3}{1-a_2} = \alpha + \mathcal{O}_{\alpha,\beta}\Big(\frac1N\Big) & \text{ if } x=1 \\ \displaystyle \frac{(1-a_x)a_{x+2}}{a_x(1-a_{x+1})} = 1 + \mathcal{O}_{\alpha,\beta}\Big(\frac1N\Big) & \text{ if } x \in \integers{2}{N-3}\vphantom{\Bigg\{} \\ \displaystyle \frac{1-a_{N-2}}{a_{N-2}(1-a_{N-1})} = \frac1\beta  + \mathcal{O}_{\alpha,\beta}\Big(\frac1N\Big) \vphantom{\Bigg(} & \text{ if } x= N-2, \end{cases} \end{equation} where we denote by $\mathcal{O}_{\alpha,\beta}(\varepsilon_N)$ a quantity which is bounded by $C\varepsilon_N$ with $C>0$ depending only on $\alpha,\beta$. We easily deduce the claim \eqref{eq:quasi-reversibility} from \eqref{eq:cxx} and \eqref{eq:quotient}.

The reversibility at the boundaries is left to the reader.
\end{proof}

We finally show that our reference measure is regular enough in an entropic sense.
Given two probability measures $\nu$, $\mu$ on $\Omega_N,$ define the relative entropy
\begin{equation}
	H(\nu |\mu) = \sum_{\eta\in\Omega_N} \nu (\eta )\log\left(\frac{\nu (\eta )}{\mu(\eta )}\right)=\E_{\mu}\left[\frac{\diff\nu}{\diff\mu}\log \frac{\diff\nu}{\diff\mu}\right].
\end{equation}
We now give a crude entropy bound with respect to our reference measure $\mu^N$.
\begin{lemma}\label{lem:boundentropy}
There exists a constant $C_4=C_4(\alpha ,\beta )>0$ such that for any probability measure $\nu$ which is concentrated on the ergodic component $\mathcal{E}_N$,  we have
\[H(\nu |\mu^N)\le C_4N.\]
\end{lemma}

\begin{proof}
	We obviously have that
	\begin{equation}
		\label{ax_boundedaway}
		\forall x\in\Lambda_N,\qquad m\le a_x\le M
	\end{equation}
	with $m=\alpha\wedge\beta$ and $M=\alpha\vee\beta$. Therefore, in formula \eqref{eq:muN_explicit}, we see that for any ergodic configuration $\eta\in\mathcal{E}_N$,
	\[ \mu^N(\eta )\ge \big( 1-\bar\rho (\alpha)\big)\times \big( m\wedge (1-M)\big)^{N-2}.\]
	In particular, as $\nu (\eta )\le 1$, we have $\log (\frac{\diff\nu}{\diff\mu^N})\le C_4(N-2)$ for some constant $C_4$ depending only on $\alpha$ and $\beta$. Injecting this in the definition of the relative entropy, we get the result.
\end{proof}

\subsection{Dirichlet estimate}

Fix a function $f:\Omega_N \rightarrow \R$, define the Dirichlet form with respect to $\mu^N$ as
\begin{equation}\label{def: DN}
\frak{D}_N (f) = \underbrace{\sum_{x=1}^{N-2}\frak{D}_0^x(f)}_{\displaystyle =:\frak{D}_0(f)} + \frac{\kappa}{N^\theta}\big(\frak{D}_\ell (f) +\frak{D}_r(f)\big)
\end{equation}
where
\begin{subequations}\label{def: Dirichlet forms}
\begin{align}
 \frak{D}_0^x(f) &= \int_{\Omega_N} c_{x,x+1}(\eta )\big[\sqrt{f(\eta^{x,x+1})}-\sqrt{f(\eta )}\big]^2\diff\mu^N(\eta ) \label{def: D0x}\\
\frak{D}_\ell (f) &= \int_{\Omega_N} b_\ell (\eta )\big[ \sqrt{f(\eta^1)}-\sqrt{f(\eta )}\big]^2\diff\mu^N (\eta )\label{def: D ell}\\
\frak{D}_r(f) &= \int_{\Omega_N}b_r(\eta ) \big[\sqrt{f(\eta^{N-1})}-\sqrt{f(\eta )}\big]^2\diff\mu^N (\eta ) \label{def: Dr}
\end{align}
\end{subequations}
and the  bulk and boundary rates have been defined in \eqref{eq:bulkjumprate} and  \eqref{def:boundaryrates}. Thanks to the technical estimates obtained in the previous section, we are in a position to estimate the spectral radius of the generator $\mathcal{L}_N$ with the Dirichlet form $\frak{D}_N$.

\begin{proposition}\label{prop: estimate Dirichlet form}
	There exists a constant $C_5=C_5(\alpha ,\beta )>0$ such that for any probability density function $f:\Omega_N \longrightarrow [0,+\infty ]$ with respect to the measure $\mu^N$,  we have
	\begin{equation}\label{estimate Dirichlet}
		\mu^N (\sqrt{f}\mathcal{L}_N\sqrt{f}) \le - \frac{1}{4}\frak{D}_N(f) +\frac{C_5}{N}.
	\end{equation}
\end{proposition}

To prove this estimate,  we will use repeatedly the following classical estimate, whose proof can be found in \cite[Lemma 5.1]{bernardin20192low}.
\begin{lemma}\label{lemma: from BGJO}
Let $T:\eta\mapsto T_\eta\in \Omega_N$ be a configuration transformation,  and let $c:\Omega_N\longrightarrow [0,+\infty [$ be a non-negative local function.  Let $f$ be a density with respect to a probability measure $\mu$.  Then,  we have that
\begin{multline}
\label{error term BGJO}
\int_{\Omega_N} c(\eta )\big[ \sqrt{f(T_\eta )}-\sqrt{f(\eta )}\big]\sqrt{f(\eta )}\diff\mu (\eta )\\ \le -\frac{1}{4}\int_{\Omega_N} c(\eta ) \big[ \sqrt{f(T_\eta )}-\sqrt{f(\eta )}\big]^2\diff\mu (\eta )+\frac{1}{8}\int_{\Omega^{\mu c}_N} c(\eta )\left( 1-\frac{c(T_\eta )}{c(\eta )}\frac{\mu (T_\eta )}{\mu (\eta )}\right)^2[ f(T_\eta )+f(\eta )]\diff\mu (\eta ),
\end{multline}
where we defined $\Omega^{\mu c}_N:=\{\eta\in\Omega_N : \;\mu(\eta) c(\eta)>0\}$.
\end{lemma}

\begin{proof}[Proof of Proposition \ref{prop: estimate Dirichlet form}]
	The quantity in the left hand side of \eqref{estimate Dirichlet} is a sum of three terms,  each one coming from one of the generators $\mathcal{L}_0$,  $\mathcal{L}_\ell $ and $\mathcal{L}_r$.  Let us treat each of these terms separately,  beginning with the one coming from $\mathcal{L}_0$.  By definition,  it reads
	\[\mu^N(\sqrt{f}\mathcal{L}_0\sqrt{f}) = \sum_{x=1}^{N-2} \int_{\Omega^x_N}c_{x,x+1}(\eta )\big[ \sqrt{f(\eta^{x,x+1})}-\sqrt{f(\eta )}\big]\sqrt{f(\eta )}\diff\mu^N(\eta )\]
	where $\Omega_N^x$ has been defined in \eqref{def:OmegaNx}.
	We now apply Lemma \ref{lemma: from BGJO} to each of these integrals to get
	\begin{multline}
		\label{eq:boundDir1}
		\mu^N (\sqrt{f}\mathcal{L}_0\sqrt{f}) \le -\frac14\mathfrak{D}_0(f) \\
        + \frac18 \sum_{x=1}^{N-2}\int_{\Omega_N^x}c_{x,x+1}(\eta )\left( 1-\frac{c_{x,x+1}(\eta^{x,x+1})}{c_{x,x+1}(\eta )}\frac{\mu^N(\eta^{x,x+1})}{\mu^N (\eta)}\right)^2\big[ f(\eta^{x,x+1})+f(\eta )\big]\diff\mu^N(\eta).
	\end{multline}
%
	Now, in \eqref{eq:boundDir1} we directly use the quasi-reversibility relation \eqref{eq:quasi-reversibility} proved in Lemma \ref{lem:quasi-reversibility}, and the fact that $c_{x,x+1}(\eta) \leqslant 1$.
	This allows us to obtain from \eqref{eq:boundDir1} the bound
	\begin{equation}
		\label{eq:boundDir1computed}
		\mu^N(\sqrt{f}\mathcal{L}_0\sqrt{f})\le -\frac14\mathfrak{D}_0(f)
		+\frac{C_3^2}{N^2}\sum_{x=1}^{N-2}\int_{\Omega_N^x}\big[ f(\eta^{x,x+1})+f(\eta )\big]\diff\mu^N(\eta).
	\end{equation}
	Since $f$ is a density with respect to $\mu^N$, and since $\mu^N(\eta^{x,x+1})/\mu^N(\eta )$ is uniformly bounded in $x,N$ (see \eqref{eq:quotient}), the integral $\int_{\Omega_N^x}f(\eta^{x,x+1})\diff\mu^N(\eta )$ is bounded uniformly in $x$ as well. This yields as wanted that for some constant $C_5>0$ depending only on $\alpha$ and $\beta$, we have
	\begin{equation}
		\label{eq:boundDir2}
		\mu^N(\sqrt{f}\mathcal{L}_0\sqrt{f}) \le -\frac{1}{4}\frak{D}_0(f) +\frac{C_5}{N}.
	\end{equation}
	Let us now deal with the term coming from the generator $\mathcal{L}_\ell$ of the left boundary. A similar application of Lemma \ref{lemma: from BGJO} yields that 
	\begin{equation}
		\label{eq:bounddirLl}
		\mu^N(\sqrt{f}\mc{L}_\ell\sqrt{f})\le -\frac{1}{4}\mathfrak{D}_\ell (f)+\frac{1}{8}\int_{\Omega_N^0} b_\ell (\eta)\left( 1-\frac{b_\ell (\eta^1)}{b_\ell (\eta)}\frac{\mu^N(\eta^1)}{\mu^N(\eta)}\right)^2\big[ f(\eta^1)+f(\eta )\big]\diff\mu^N(\eta )
	\end{equation}
	where $\Omega_N^0:= \Omega_N^{\mu^Nb_\ell}$. In $\Omega_N^0$,  there are only two possible configurations on $\lbrace 1,2\rbrace$,  namely $\bulletone\bullets$ and $\circone\bullets$ and we go from one to the other by the transformation $\eta\longmapsto\eta^1$. But note that
	\[ b_\ell (\circone\bullets)\mu^N(\circone\bullets)-b_\ell(\bulletone\bullets)\mu^N(\bulletone\bullets) = \alpha \big( 1-\bar\rho (\alpha )\big) - (1-\alpha ) \bar\rho (\alpha )a_1 =0.\]
	This equality implies that the integral in the right hand side of \eqref{eq:bounddirLl} vanishes, so we obtain
	\begin{equation}
		\label{eq:bounddirLl2}
		\mu^N(\sqrt{f}\mathcal{L}_\ell\sqrt{f})\le - \frac14\mathfrak{D}_\ell (f).
	\end{equation}
	A similar computation for the last term coming from the generator $\mathcal{L}_r$ shows that
	\begin{equation}
		\label{eq:bounddirLr}
		\mu^N(\sqrt{f}\mathcal{L}_r\sqrt{f}) \le -\frac14 \mathfrak{D}_r(f).
	\end{equation}
	Putting \eqref{eq:boundDir2}, \eqref{eq:bounddirLl2} and \eqref{eq:bounddirLr} together, we deduce the result.
\end{proof}

\section{Proof of Theorem \ref{thm: hydrodynamic limit}}
\label{sec:HDL}

\subsection{Tightness and absolute continuity}

We now have the main ingredients needed to carry on with the proof of the hydrodynamic limit. Recall that we defined  in Section \ref{sec:HDLstatement} the distribution $\Q^N=\Prob_{\nu_0^N}\circ (m^N)^{-1}$ of the boundary-driven FEP empirical measure's trajectory. The proof of the tightness of $(\Q^N)_{N\ge 1}$ is quite standard and relies on \emph{Aldous criterion} (cf. \cite[Section 4.1]{KL}) which gives a necessary and sufficient condition for a sequence of measures to be tight in the Skorokhod topology. Nevertheless, during the proof one has to distinguish between the values $\theta\ge 1$ and $\theta <1$. Both are treated similarly, with the difference that the latter requires to approximate functions by compactly supported functions in order to get rid of boundary terms that can diverge. We omit this proof and we refer the reader to \cite[Section 4]{baldasso2017exclusion}, where they treat only the case $\theta\ge 0$, but with similar arguments we can extend it to $\theta <0$. 

Since we are dealing with an exclusion process, following classical arguments (see \emph{e.g.}~\cite[page 57]{KL}), it is straightforward to show that any limit point of $(\Q^N)_{N\ge 1}$ is concentrated on trajectories $(m_t)_{t\geq 0}$ which are absolutely continuous with respect to the Lebesgue measure and write $m_t(\mathrm{d} u)=\rho_t(u)\mathrm{d} u$, where the  profile $\rho_t(\cdot )$ takes its values in $[0,1]$. We actually have a stronger result as the process we consider starts and remains in the ergodic component, so that the profile $\rho_t(\cdot )$ takes its values in $\big[\frac12,1\big]$. Indeed, if $G : [0,1]\longrightarrow \R_+$ is a non-negative $\mathcal{C}^1$ function and $\eta\in\mathcal{E}_N$, we can write
\begin{equation*}
	\frac1N\sum_{x=1}^{N-1}\eta_xG\Big(\frac xN\Big) + \frac{1}{N}\sum_{x=1}^{N-1}\eta_xG\Big(\frac{x+1}{N}\Big) =  \frac{2}{N}\sum_{x=1}^{N-1}\eta_x G\Big( \frac{x}{N}\Big) +\mathcal{O}\Big(\frac1N\Big)
\end{equation*}
on the one hand, but on the other hand this quantity is also equal to
\begin{equation*}
	\frac1N\sum_{x=2}^{N-1}(\eta_x+\eta_{x-1})G\Big(\frac xN\Big) + \mathcal{O}\Big(\frac1N\Big).
\end{equation*}
Since $\eta$ is ergodic we have $\eta_x+\eta_{x-1}\ge 1$, so we can bound below this latter sum and get
\begin{equation*}
	\frac1N\sum_{x=1}^{N-1} \eta_xG\Big(\frac xN\Big) \ge \frac{1}{2N}\sum_{x=1}^{N-1}G\Big( \frac xN\Big) +\mathcal{O}\Big(\frac1N\Big).
\end{equation*}
The sum on the right hand side of this inequality converges to $\frac12\int_0^1G(u)\diff u$ as $N$ goes to infinity, so we conclude that
\begin{equation*}
	\frac1N\sum_{x=1}^{N-1} \eta_xG\Big(\frac xN\Big) -\frac12\int_0^1G(u)\diff u + \ge \mathcal{o}(1).
\end{equation*}
Therefore, for any $t\in [0,T]$
\begin{equation*}
	1 = \mathbb{P}_{\nu_0^N}\big( \eta (t)\in\mathcal{E}_N\big) \le \mathbb{P}_{\nu_0^N}\left( \langle m_t^N,G\rangle -\frac12\int_0^1G(u)\diff u \ge \mathcal{o}(1)\right)
\end{equation*}
and the probability on the right hand side is then equal to 1. An application of the Portmanteau Theorem then yields that $\Q$-almost surely
\begin{equation*}
	\int_0^1 \left(\rho_t(u) -\frac12\right) G(u)\diff u \ge 0.
\end{equation*}
Since this holds for any non-negative function $G$, we deduce that $\rho_t$ takes values bigger than $\frac12$ almost everywhere.

The following two subsections consist in showing that there is a unique limit point to the sequence $(\Q^N)_{N\ge 1}$ by showing that the profile $\rho$ is a weak solution to the hydrodynamic equations, which is known to be unique (as proved in Appendix \ref{sec:appuniqueness}).

\subsection{Weak formulation}

In order to prove that the density $\rho$ of the limiting measure $m$ is a weak solution of the hydrodynamic equations given in Theorem \ref{thm: hydrodynamic limit}, the first step is to prove that it satisfies a weak formulation. By construction, as detailed in Section  \ref{sec:nu0N} at time $t=0$,  this density coincides with the chosen initial profile, meaning that any limit point  $\Q$ of $(\Q^N)_{N\ge 1}$ satisfies that for any $\delta >0$, and any continuous function $G: [0,1]\longrightarrow\R$,
\begin{equation}\label{eq:initialEmpMeasure}
\Q\left( \left| \langle m_0,G\rangle - \int_0^1 \rho^\mathrm{ini}(u)G(u)\diff u\right| >\delta\right)=0
\end{equation}
by \eqref{eq:initialprofile}. Fix a test function $G\in\mathcal{C}^{1,2}([0,T]\times [0,1])$, it is well-known (see \cite[Lemma 5.1,  Appendix 1.5]{KL}) that 
\begin{equation}\label{Dynkin martingale}
\bb{M}_t^N(G): = \langle m_t^N,G_t\rangle -\langle m_0^N,G_0\rangle - \int_0^t\langle m_s^N,\partial_tG_s\rangle\diff s-\int_0^tN^2\mathcal{L}_N\langle m_s^N,G_s\rangle\diff s
\end{equation}
defines a mean-zero martingale. Recall the definitions of the instantaneous currents in \eqref{def:current} and \eqref{def:currentboundaries}, of the function $h_x$ in \eqref{def:hx} and of the two gradient decompositions \eqref{eq:gradientcondition} and \eqref{eq:gradientboundaries}.
 Recall that for a function $g(\eta )$, we simply write $g(s):=g(\eta (s))$.  As a consequence,  after two successive summations by parts,  the term inside the second integral of \eqref{Dynkin martingale} writes
\begin{align*}
N^2\mathcal{L}_N\langle m_s^N,G_s\rangle  = & \; NG_s\left(\frac{1}{N}\right)j_{0,1}(s)-NG_s\left( \frac{N-1}{N}\right) j_{N-1,N}(s)\\
& +\nabla_N^+G_s(0)h_1(s)-\nabla_N^-G_s(1)h_{N-1}(s) +\frac{1}{N}\sum_{x\in\Lambda_N} \Delta_NG_s\Big( \frac{x}{N}\Big) h_x(s)
\end{align*}
where we defined the \emph{discrete gradients} by
\begin{equation}\label{eq:discrete}
	\nabla_N^+G\Big( \frac{x}{N}\Big) = N\left( G\Big( \frac{x+1}{N}\Big) - G\Big( \frac{x}{N}\Big) \right)\quad\mbox{ and }\quad \nabla_N^-G\Big( \frac{x}{N}\Big) = N\left( G_s\Big( \frac{x}{N}\Big) - G\Big( \frac{x-1}{N}\Big) \right),
\end{equation}
and the \emph{discrete Laplacian} by
\begin{equation*}
	\Delta_NG\Big( \frac{x}{N}\Big) =  N^2\left( G\Big( \frac{x+1}{N}\Big) +G\Big( \frac{x-1}{N}\Big) - 2G\Big( \frac{x}{N}\Big)\right).
\end{equation*}
The gradient condition \eqref{eq:gradientboundaries} at $x=0$ and $x=N-1$, and the conventions \eqref{def:conventions2} allow to rewrite Dynkin's martingale under the form
\begin{align}
	\label{Dynkin_computed}
	\mathbb{M}_t^N(G)  = & \; \langle m_t^N,G_t\rangle - \langle m_0^N,G_0\rangle - \int_0^t \langle m_s^N,\partial_tG_s\rangle\diff s - \int_0^t \frac1N\sum_{x\in\Lambda_N}\Delta_NG_s\Big(\frac xN\Big)h_x(s)\diff s \notag \\
&	+\int_0^t\Big\lbrace \nabla_N^-G_s(1)h_{N-1}(s)-\nabla_N^+G_s(0)h_1(s)\Big\rbrace\diff s \notag \vphantom{\sum_{x\in\Lambda_N}}\\
&	+\kappa N^{1-\theta}\int_0^t \left\lbrace G_s\left(\frac{N-1}{N}\right) \big( h_{N-1}(s)-\beta\big)-G_s\left(\frac 1N\right) \big(\alpha -h_1(s)\big)\right\rbrace\diff s. 
\end{align}
To obtain the weak formulation corresponding to the value of $\theta$, we need to replace local functions of the configuration by functions of the empirical measure. Take $\varepsilon >0$ and recall the definition of the active density $\act (\rho )$ at density $\rho$ given in \eqref{eq:act}, and of the box $\Lambda_x^\ell$ in \eqref{def:lambdaxell}. The first replacement lemma, that is true for any value of $\theta$ asserts that we can replace each $h_x(s)$ in \eqref{Dynkin_computed} by $\act (\eta_x^{\varepsilon N} (s))$, where $\eta_x^{\varepsilon N}(s)$ is the average density on the box $\Lambda_x^{\varepsilon N}$ in the configuration $\eta (s)$
\begin{equation}
	\eta_x^{\varepsilon N}(s):=\frac{1}{|\Lambda_x^{\varepsilon N}|} \sum_{y\in \Lambda_x^{\varepsilon N}}\eta_y (s).
\end{equation}
More precisely,  Lemma \ref{lemma: replacement bulk} below says that the error  we make, doing this replacement, vanishes when we let $N$ go to $+\infty$,  and then $\varepsilon$ to 0:

\begin{lemma}[Replacement lemma in the bulk]
\label{lemma: replacement bulk}
For any $t\in [0,T]$,  and for any continuous function $\varphi : [0,T]\longrightarrow\R$ we have that
\begin{equation}\label{expectation replacement bulk}
\limsup_{\varepsilon\to 0}\limsup_{N\to +\infty} \sup_{x\in\Lambda_N} \E_{\nu_0^N}\left[ \bigg| \int_0^t \varphi (s) \Big( h_x(s) - \act\big(\eta_x^{\varepsilon N}(s)\big)\Big) \diff s\bigg| \right] =0.
\end{equation}
\end{lemma}

Define $\Sigma_N^\varepsilon = \integers{\varepsilon N+1}{(1-\varepsilon )N-1}$. Note that for any $x\in\Sigma_N^\varepsilon$, up to a factor that goes to 1 with $N$, we have $\eta_x^{\varepsilon N}(s) = m_s^N *\iota_\varepsilon (\frac xN)$ where $\iota_\varepsilon = \frac{1}{2\varepsilon}\ind_{[-\varepsilon ,\varepsilon]}$ is an approximation of the unit and $*$ denotes the usual convolution operation. Moreover, we have $\eta_1^{\varepsilon N}(s) = m_s^N *\iota_{2\varepsilon} (\frac 1N)$ and $\eta_{N-1}^{\varepsilon N}(s) = m_s^N *\iota_{2\varepsilon} (\frac{N-1}{N})$ at the boundaries. Thanks to these relations,  the lemma indeed helps us to close the expression with respect to the empirical measure. 
Lemma \ref{lemma: replacement bulk} is a stronger version of \cite[Lemma 5.4]{blondel2020hydrodynamic}, and its proof has to be carefully adapted from the latter because the addition of boundary dynamics and the nature of the FEP's stationary states breaks down some  arguments based on  translation invariance. We postpone it to Section \ref{sec: Replacement lemma bulk}, and now complete the proof of the hydrodynamic limit for the different values of $\theta$.

\subsubsection{Case $\theta <1$}

In this paragraph only, we further assume that $G\in \mathcal{C}_c^{1,2}\big([0,T]\times (0,1)\big)$. As the function $G$ is of class $\mathcal{C}^2$ with respect to the space variable, we can replace the discrete gradients and Laplacian in \eqref{Dynkin_computed} by their continuous versions up to an error that vanishes with $N$, and as $G$ is compactly supported, the last two integrals of \eqref{Dynkin_computed} vanish if $N$ is chosen large enough. Moreover, the fact that for any smooth $G$
\begin{equation}\label{eq:errorLambdaNtoSigmaNeps}
	\int_0^t\frac1N\sum_{x\in\Lambda_N\setminus\Sigma_N^\varepsilon} \Delta_NG_s\Big(\frac xN\Big)h_x(s)\diff s \underset{\varepsilon\to 0}{=} \mathcal{O}(\varepsilon ),
\end{equation}
together with Lemma \ref{lemma: replacement bulk}, yield that the Dynkin martingale rewrites
\begin{multline}
	\label{Dynkin_computed_Dirichlet}
	\mathbb{M}_t^N(G)= \langle m_t^N,G_t\rangle - \langle m_0^N,G_0\rangle - \int_0^t \langle m_s^N,\partial_tG_s\rangle\diff s \\ - \int_0^t \frac1N\sum_{x\in\Sigma_N^\varepsilon}\partial_u^2G_s\Big(\frac xN\Big) \act\Big( m_s^N*\iota_\varepsilon \Big(\frac xN\Big)\Big)\diff s +\mathcal{o}_{N,\varepsilon}(1)
\end{multline}
where $\mathcal{o}_{N,\varepsilon}(1)$ is a (random) error term that vanishes in probability as $N$ goes to $+\infty$, and $\varepsilon$ goes to 0. In Proposition \ref{prop:dynkin} of Appendix \ref{sec:appmartingale}, we prove that
\begin{equation*}
	\limsup_{N\to +\infty} \mathbb{P}_{\nu_0^N} \left( \sup_{t\in [0,T]}\big|\mathbb{M}_t^N(G)\big| >\delta \right) =0
\end{equation*}
for any $\delta >0$, so we obtain that
\begin{multline}
	\limsup_{\varepsilon\to 0}\limsup_{N\to +\infty}\Prob_{\nu_0^N}\bigg( \sup_{t\in [0,T]} \bigg| \langle m_t^N,G_t\rangle - \langle m_0^N,G_0\rangle - \int_0^t \langle m_s^N,\partial_tG_s\rangle\diff s\\ 
	- \int_0^t\frac1N\sum_{x\in\Sigma_N^\varepsilon}\partial_u^2G_s\Big(\frac xN\Big) \act\Big( m_s^N*\iota_\varepsilon\Big(\frac xN\Big)\Big)\diff s\bigg| >\delta\bigg)=0
\end{multline}
for any $\delta >0$. Now that everything is expressed in terms of the empirical measure $m^N$, an application of the Portmanteau Theorem shows that 
\begin{multline}
	\limsup_{\varepsilon\to 0}\Q \bigg( \sup_{t\in [0,T]}\bigg| \langle\rho_t,G_t\rangle - \langle\rho_0,G_0\rangle -\int_0^t\langle\rho_s,\partial_tG_s\rangle\diff s \\
	-\int_0^t\int_{\varepsilon}^{1-\varepsilon} \partial_u^2G_s(u)\act\big( \rho_s*\iota_\varepsilon (u)\big)\diff u\diff s\bigg| >\delta \bigg)=0.
\end{multline}
Letting $\varepsilon$ go to 0, we get that $\Q$--almost surely
\begin{equation}
	\langle\rho_t,G_t\rangle - \langle \rho_0,G_0\rangle - \int_0^t\langle\rho_s,\partial_tG_s\rangle\diff s - \int_0^t \big\langle\act (\rho_s),\partial_u^2G_s\big\rangle\diff s =0
\end{equation}
for all $t\in [0,T]$ and all $G\in\mathcal{C}_c^{1,2}([0,T]\times [0,1])$. Since, as mentioned in \eqref{eq:initialEmpMeasure}, we have that $\rho_0=\rho^\mathrm{ini}$, $\Q$--almost surely, we recognize immediately the weak formulation \eqref{eq:weakformulation_Dirichlet}. Then, recall Definition \ref{defin:weaksolDirichlet}: we also need to prove the last point (iii), namely that $\rho$ satisfies the Dirichlet boundary conditions \eqref{eq:DirichletBC}. This is done thanks to the following result.

\begin{lemma}[Replacement lemma at the boundaries]
	\label{lemma:replacement_boundary}
	If $\theta <1$, then for all $t\in [0,T]$ we have
	\begin{equation}
		\lim_{N\to +\infty} \E_{\nu_0^N}\left[\left| \int_0^t \big( h_1(s) -\alpha\big)\diff s\right|\right]=0.
	\end{equation}
	The same holds true if we replace $h_1(s)$ by $h_{N-1}(s)$ and $\alpha$ by $\beta$ in this expression.
\end{lemma}

Here also, the proof of Lemma \ref{lemma:replacement_boundary} has to be handled carefully, and therefore we postpone it to Section \ref{sec:replacement_boundary}.

Now, if we combine Lemma \ref{lemma:replacement_boundary} with the fact that, for any $t\in [0,T]$,
\[ \limsup_{\varepsilon\to 0}\limsup_{N\to +\infty} \E_{\nu_0^N} \left[ \left| \int_0^t\big(h_1(s)-\act \big( \eta_1^{\varepsilon N}(s)\big)\big)\diff s\right|\right] =0\]
(which is direct consequence of Lemma \ref{lemma: replacement bulk} with $\varphi\equiv 1$), we get that
\begin{equation}
	\limsup_{\varepsilon\to 0}\limsup_{N\to +\infty}\E_{\nu_0^N}\left[\left|\int_0^t \big(\act\big(\eta_1^{\varepsilon N}(s)\big)-\alpha\big)\diff s\right|\right] =0.
\end{equation}
Using Markov's inequality and the Portmanteau Theorem to pass to the limit over $N$ as before, we get that
\begin{equation*}
	\limsup_{\varepsilon\to 0}\Q\big( \big|\act\big(\rho_t*\iota_{2\varepsilon} (0)\big) - \alpha\big| >\delta \big) =0
\end{equation*}
for any $t\in [0,T]$ and any $\delta >0$. This proves that, $\Q$--almost surely, we have
\begin{equation*}
	\act\big( \rho_t(0)\big) = \alpha\qquad\Longleftrightarrow\qquad \rho_t(0)=\bar\rho (\alpha )
\end{equation*}
for almost every $t\in [0,T]$, and we obtain the corresponding result on the right boundary by repeating this proof.

To sum up, we have proved that the limit density profile satisfies points (ii) and (iii) of Definition \ref{defin:weaksolDirichlet}. The proof of Theorem \ref{thm: hydrodynamic limit} in the case $\theta<1$ will be concluded as soon as we prove point (i), namely that $\mathfrak{a}(\rho)$ belongs to $L^2([0,T],\mathcal{H}^1)$. This property follows from an \emph{energy estimate} which holds true for any value of $\theta$, and we will give its complete proof in full generality in Section \ref{sec:energy} below.

We have completed the case $\theta <1$ in Theorem \ref{thm: hydrodynamic limit}. Let us carry on with the other possible values of $\theta$.

\subsubsection{Case $\theta =1$}

We get back to a test function $G\in\mathcal{C}^{1,2}([0,T]\times [0,1])$ that is not necessarily compactly supported, and we assume that $\theta =1$. In particular, if we replace the discrete derivatives of $G$ by their continuous versions and we make use of Lemma \ref{lemma: replacement bulk}, the last two integrals in \eqref{Dynkin_computed} no longer vanish, and we obtain instead that Dynkin's martingale rewrites
\begin{align}
	\label{Dynkin_computed_Robin}
	\mathbb{M}_t^N&(G)= \langle m_t^N,G_t\rangle - \langle m_0^N,G_0\rangle - \int_0^t \langle m_s^N,\partial_tG_s\rangle\diff s \notag \\ 
&	+\int_0^t\Big\lbrace \partial_uG_s(1)\act\Big( m_s^N*\iota_{2\varepsilon}\Big(\frac{N-1}{N}\Big)\Big)- \partial_uG_s(0)\act\Big( m_s^N*\iota_{2\varepsilon} \Big(\frac1N\Big)\Big)\Big\rbrace\diff s \notag\\
&	+\kappa \int_0^t \Big\lbrace G_s\Big(\frac{N-1}{N}\Big)\Big( \act\Big( m_s^N*\iota_{2\varepsilon}\Big( \frac{N-1}{N}\Big)\Big) -\beta\Big) - G_s\Big(\frac1N\Big)\Big(\alpha - \act\Big(m_s^N*\iota_{2\varepsilon} \Big(\frac1N\Big)\Big)\Big)\Big\rbrace\diff s \notag\\
&	- \int_0^t \frac1N\sum_{x\in\Sigma_N^\varepsilon}\partial_u^2G_s\Big(\frac xN\Big) \act\Big( m_s^N*\iota_\varepsilon \Big(\frac xN\Big)\Big)\diff s +\mathcal{o}_{N,\varepsilon}(1)
\end{align}
where $\mathcal{o}_{N,\varepsilon}(1)$ is a (random) error term that vanishes in probability as $N$ goes to $+\infty$, and $\varepsilon$ to 0. Using the same procedure as before, we derive from this expression that any limit point $\Q$ is concentrated on trajectories of measures $m$ with density $\rho$ with respect to the Lebesgue measure, satisfying
\begin{align*}
	\langle\rho_t,G_t\rangle - \langle\rho_0,G_0\rangle& - \int_0^t \langle \rho_s,\partial_tG_s\rangle\diff s - \int_0^t \big\langle \act(\rho_s),\partial_u^2G_s\big\rangle\diff s\\
&	+ \int_0^t\Big\lbrace \partial_uG_s(1)\act\big( \rho_s(1)\big)-\partial_uG_s(0)\act\big(\rho_s(0)\big)\Big\rbrace\diff s\\
&	+ \kappa \int_0^t\Big\lbrace G_s(1)\big( \act\big( \rho_s(1)\big)-\beta\big)- G_s(0)\big( \alpha -\act\big(\rho_s(0)\big)\big)\Big\rbrace\diff s=0
\end{align*}
for any $t\in [0,T]$ and any test function $G\in\mathcal{C}^{1,2}([0,T]\times [0,1])$. We recognize immediately the weak formulation \eqref{eq:weakformulation_Robin} so $\rho$ is a weak solution of the fast diffusion equation \eqref{eq:fast_diffusion_equation_Robin} with Robin boundary conditions in the sense of Definition \ref{defin:weaksolRobin}. Together with the energy estimate given in the next Section \ref{sec:energy}, this proves the case $\theta =1$ in Theorem \ref{thm: hydrodynamic limit}.

\subsubsection{Case $\theta >1$} 

Since the function $G$ is bounded, it is clear that the last integral in \eqref{Dynkin_computed} is of order $\mathcal{O}(N^{1-\theta})$. In particular, when $\theta >1$ this term vanishes and everything can be done like in the case $\theta =1$ by taking $\kappa =0$. This proves the result for $\theta >1$ also.

\subsection{Energy estimate} \label{sec:energy}

Finally, in order to match our definition of weak solution in the different cases, we need to prove that $\rho$ satisfies an energy estimate, in the sense that $\act (\rho )$ belongs to the Sobolev space $L^2\big( [0,T],\mathcal{H}^1\big)$. We already know that $\rho$ takes its values in $\big[\frac12,1\big]$ almost everywhere, so it implies that $\act (\rho )$ takes its values in $[0,1]$ almost everywhere and then $\act (\rho )\in L^2\big( [0,T]\times [0,1]\big)$. We can thus define the linear functional 
\begin{equation*}
	\begin{array}{l|rcl}
		\mathcal{l} : & \mathcal{C}_c^{0,1}\big( [0,T]\times (0,1)\big) & \longrightarrow & \R \\
			& G & \longmapsto & \dlangle \act (\rho ),\partial_uG\drangle . \end{array}
\end{equation*}
If we show that this functional is $\Q$-almost surely continuous, then since $\mathcal{C}_c^{0,1}\big( [0,T]\times (0,1)\big)$ is dense in $L^2\big( [0,T]\times [0,1]\big)$, we will be able to extend it to a $\Q$-almost surely continuous functional on $L^2\big( [0,T]\times [0,1]\big)$. Then, we can invoke Riesz representation Theorem to deduce that there exists a function $\partial_u\act (\rho )\in L^2\big( [0,T]\times [0,1]\big)$ such that $\mathcal{l}(G)=-\dlangle \partial_u\act (\rho ),G\drangle$. This implies, as desired, that $\act (\rho )\in L^2\big( [0,T],\mathcal{H}^1\big)$.

Our goal is therefore to prove that the functional $\mathcal{l}$ is $\Q$-almost surely continuous, and as it is linear, it suffices to show that it is $\Q$-almost surely bounded. This is a consequence of Lemma \ref{lemma:energyestimate} that we state and prove below.

\begin{lemma}\label{lemma:energyestimate}
	There exists a constant $c>0$ such that
	\begin{equation}\label{eq:energyestimate}
		\E^\Q \left[ \sup_{G\in\mathcal{C}_c^{0,1}([0,T]\times (0,1))}\Big\lbrace \mathcal{l}(G)-c \dlangle G,G\drangle\Big\rbrace\right] <+\infty ,
	\end{equation}
	where $\E^\Q$ denotes the expectation with respect to the measure $\Q$.
\end{lemma}

\begin{proof}
The space $\mathcal{C}_c^{0,1}\big( [0,T]\times (0,1)\big)$ endowed with the topology of the norm $\| \cdot \|_\infty +\|\partial_u\cdot\|_\infty$ is separable, so we may consider a dense sequence $\lbrace G^j\rbrace_{j\in\N}$ in it. By the monotone convergence theorem, it is sufficient to prove that there exist positive constants $c$ and $K$ such that for any $k\in\N$,
\begin{equation*}\label{eq:energyestimate}
	\E^\Q \left[ \max_{j\le k}\Big\lbrace \mathcal{l}(G^j)-c \dlangle G^j,G^j\drangle\Big\rbrace\right] \le K.
\end{equation*}
For now, let $c$ be any positive real number. By approximation and Lebesgue differentiation Theorem, together with Fatou Lemma, the expectation above is bounded by
\begin{equation*}
	\liminf_{\varepsilon\to 0}\E^\Q \left[ \max_{j\le k}\int_0^T \left\lbrace \int_\varepsilon^{1-\varepsilon} \act\big( \rho_s*\iota_\varepsilon (u)\big)\partial_uG_s^j(u)\diff u - c\| G_s^j\|_{L^2}^2\right\rbrace\diff s\right]
\end{equation*}
The map
\begin{equation*}
	m_\cdot \in\mathcal{D}\big( [0,T],\mathcal{M}_+\big) \longmapsto \max_{j\le k}\int_0^T \left\lbrace \int_\varepsilon^{1-\varepsilon} \act\big( m_s*\iota_\varepsilon (u)\big)\partial_uG_s^j(u)\diff u - c\| G_s^j\|_{L^2}^2\right\rbrace\diff s
\end{equation*}
is bounded and lower semi-continuous in the Skorokhod topology, so using once again Fatou Lemma we can bound last expectation from above by
\begin{equation*}
	\liminf_{\varepsilon\to 0}\liminf_{N\to +\infty} \E_{\nu_0^N} \left[ \max_{j\le k}\int_0^T \left\lbrace \int_\varepsilon^{1-\varepsilon} \act\big( m_s^N*\iota_\varepsilon (u)\big)\partial_uG_s^j(u)\diff u - c\| G_s^j\|_{L^2}^2\right\rbrace\diff s\right] .
\end{equation*}
Using Lemma \ref{lemma: replacement bulk}, we reduce the problem to the study of
\begin{equation*}
	\liminf_{\varepsilon\to 0}\liminf_{N\to +\infty} \E_{\nu_0^N}\left[ \max_{j\le k} \int_0^T \left\lbrace \frac1N \sum_{x\in\Lambda_N} h_x(s)\partial_uG_s^j\Big(\frac xN\Big) - c \| G_s^j\|_{L^2}^2\right\rbrace\diff s\right].
\end{equation*}
Note that the sum should be a sum over $x\in\Sigma_N^\varepsilon$, but we have replaced it by a sum over $x\in\Lambda_N$ because as in \eqref{eq:errorLambdaNtoSigmaNeps}, the error we make is of order $\mathcal{O}(\varepsilon )$. Now, we use the entropy inequality \cite[Appendix A.1.8]{KL} to let our reference measure $\mu^N$ come into play, and also Jensen's inequality to bound this expectation above by
\begin{equation*}
	\frac{H(\nu_0^N|\mu^N)}{N}+\frac1N\log \E_{\mu^N}\left[\exp\left( \max_{j\le k} \int_0^T \left\lbrace  \sum_{x\in\Lambda_N} h_x(s)\partial_uG_s^j\Big(\frac xN\Big) - cN \| G_s^j\|_{L^2}^2\right\rbrace\diff s\right)\right].
\end{equation*}
Using Lemma \ref{lem:boundentropy} together with the inequality $\exp (\max_{j\le k}a_j)\le \sum_{j\le k}\exp (a_j)$, this expression is in turn bounded above by
\begin{equation*}
	C_4 + \frac1N\log \E_{\mu^N}\left[ \sum_{j=1}^k \exp\left( \int_0^T\left\lbrace \sum_{x\in\Lambda_N}h_x(s)\partial_uG_s^j\Big(\frac xN\Big) - cN\| G_s^j\|_{L^2}^2\right\rbrace\diff s\right)\right]
\end{equation*}
Now, thanks to the inequality
\begin{equation*}
	\limsup_{N\to +\infty}\frac{1}{N}\log (u_N+v_N) \le \max\left\lbrace \limsup_{N\to +\infty} \frac{1}{N}\log u_N,  \limsup_{N\to +\infty} \frac{1}{N}\log v_N\right\rbrace ,
\end{equation*}
we are left to show that there are positive constants $c$ and $K$ such that 
\begin{equation}\label{energyestimatetodo}
	\liminf_{N\to +\infty} \frac1N\log\E_{\mu^N}\left[ \exp\left(\int_0^T \left\lbrace \sum_{x\in\Lambda_N}h_x(s)\partial_uG_s\Big(\frac xN\Big)-cN\|G_s\|_{L^2}^2\right\rbrace\diff s\right)\right] \le K
\end{equation}
for any function $G\in\mathcal{C}_c^{0,1}\big( [0,T]\times (0,1)\big)$. Now, by Feynman-Kac Formula, the term in the limit can be bounded above by
\begin{equation}\label{sup_energyestimate}
	\int_0^T \sup_f\left\lbrace N\mu^N\big( \sqrt{f}\mathcal{L}_N\sqrt{f}\big) + \frac1N\sum_{x\in\Lambda_N}\partial_uG_s\Big(\frac xN\Big) \mu^N(h_xf) - c\|G_s\|_{L^2}^2 \right\rbrace\diff s
\end{equation}
where the supremum is carried over all density functions $f$ with respect to the measure $\mu^N$. We are already able to estimate the first term inside this supremum thanks to Proposition \ref{prop: estimate Dirichlet form}, so let us focus on the second one. As the test function $G$ is of class $\mathcal{C}^1$ with respect to the space variable, we can replace its space derivative by a discrete gradient (recall definitions \eqref{eq:discrete}) up to an error of order $\mathcal{O}\big(\frac1N\big)$, so that
\begin{equation*}
	\frac1N\sum_{x\in\Lambda_N}\partial_uG_s\Big(\frac xN\Big) \mu^N(h_xf)  = \int_{\Omega_N}\frac1N \sum_{x\in\Lambda_N} \nabla_N^-G_s\Big(\frac xN\Big) h_x(\eta ) f(\eta )\diff\mu^N(\eta ) + \mathcal{O}\Big(\frac1N\Big)
\end{equation*}
using the fact that each function $h_x$ is bounded by 1, and that $f$ is a density with respect to $\mu^N$. Now, let us perform a summation by parts in the sum
\begin{align*}
	\frac1N \sum_{x\in\Lambda_N} \nabla_N^-G_s\Big(\frac xN\Big) h_x(\eta ) & = \sum_{x=1}^{N-1}\left( G_s\Big(\frac xN\Big)-G_s\Big( \frac{x-1}{N}\Big)\right) h_x(\eta ) \\
	& = \sum_{x=1}^{N-2} \big( h_x(\eta )-h_{x+1}(\eta )\big)G_s\Big( \frac xN\Big) +h_{N-1}(\eta )G_s\Big(\frac{N-1}{N}\Big) -h_1(\eta )G_s(0)\\
	& = \sum_{x=1}^{N-2}j_{x,x+1}(\eta )G_s\Big( \frac xN\Big) 
\end{align*}
recalling the gradient condition \eqref{eq:gradientcondition}, and choosing $N$ large enough so that the boundary terms vanish since $G_s$ has compact support included in $(0,1)$. Now, recalling the definition \eqref{def:current} of the current, we are left to study the term
\begin{equation*}
	\sum_{x=1}^{N-2}G_s\Big(\frac xN\Big)\int_{\Omega_N}c_{x,x+1}(\eta )(\eta_x-\eta_{x+1}) f(\eta )\diff\mu^N(\eta ).
\end{equation*}
Note that these integrals are once again integrals over the set $\Omega_N^x$ defined in \eqref{def:OmegaNx}. Split it into two halves, and perform the change of variable $\eta \leadsto \eta^{x,x+1}$ in the second half to write it as 
\begin{multline*}
	\frac12 \sum_{x=1}^{N-2}G_s\Big(\frac xN\Big)\int_{\Omega_N^x}c_{x,x+1}(\eta )(\eta_x -\eta_{x+1})f(\eta )\diff\mu^N(\eta ) \\+\frac12 \sum_{x=1}^{N-2}G_s\Big(\frac xN\Big) \int_{\Omega_N^x} c_{x,x+1}(\eta^{x,x+1})(\eta_{x+1}-\eta_x)f(\eta^{x,x+1})\frac{\mu^N(\eta^{x,x+1})}{\mu^N(\eta )}\diff\mu^N(\eta ).
\end{multline*}
Recall the ``quasi-reversibility'' relation \eqref{eq:quasi-reversibility}. Using this, together with the fact that the integral $\int_{\Omega_N^x}f(\eta^{x,x+1})\diff\mu^N(\eta )$ is uniformly bounded in $x$, we thus get that this expression reads
\begin{equation*}
	\frac12\sum_{x=1}^{N-2}G_s\Big( \frac xN\Big) \int_{\Omega_N^x}c_{x,x+1}(\eta )(\eta_{x+1} -\eta_x)\big[ f(\eta^{x,x+1})-f(\eta )\big]\diff\mu^N(\eta ) + \mathcal{O}(1).
\end{equation*}
By Young's inequality we have that, for any $A>0$,
\begin{multline*}
	G_s\Big(\frac xN\Big) (\eta_{x+1}-\eta_x)\big[ f(\eta^{x,x+1})-f(\eta )\big] \le \frac{1}{2A}\big[ \sqrt{f(\eta^{x,x+1})}-\sqrt{f(\eta )}\big]^2 \\
	+ \frac A2 G_s\Big(\frac xN\Big)^2 \underbrace{(\eta_{x+1}-\eta_x)^2}_{=1}\big[ \sqrt{f(\eta^{x,x+1})}+\sqrt{f(\eta )}\big]^2
\end{multline*}
so that the remaining sum is bounded by
\begin{multline*}
	\frac{1}{4A}\sum_{x=1}^{N-2} \int_{\Omega_N^x}c_{x,x+1}(\eta )\big[ \sqrt{f(\eta^{x,x+1})}-\sqrt{f(\eta )}\big]^2 \diff\mu^N(\eta ) \\
	+ \frac A4 \sum_{x=1}^{N-2}G_s\Big( \frac xN\Big)^2\int_{\Omega_N^x}c_{x,x+1}(\eta )\big[ \sqrt{f(\eta^{x,x+1})}+\sqrt{f(\eta )}\big]^2\diff\mu^N(\eta ).
\end{multline*}
The first term above is a piece of the total Dirichlet form, so it can be bounded by $\frac{1}{4A}\mathfrak{D}_N(f)$. All the integrals in the second term are bounded above by a constant $\tilde{C}$ uniformly in $x$, this can be seen using the fact that $c_{x,x+1}(\eta )\le 1$, the inequality $(a+b)^2\le 2(a^2+b^2)$, and the fact that $f$ is a density with respect to $\mu^N$. To sum up, we have obtained the following estimation
\begin{equation*}
	\frac1N\sum_{x\in\Lambda_N}\partial_uG_s\Big(\frac xN\Big) \mu^N(h_xf) \le \frac{1}{4A}\mathfrak{D}_N(f) + \tilde{C}\frac{A}{4} \sum_{x=1}^{N-2}G_s\Big(\frac xN\Big)^2 + \widehat{C}
\end{equation*}
for any $A>0$, and for some constants $\tilde{C},\widehat{C}>0$ that do not depend on $G$. This, together with Proposition \ref{prop: estimate Dirichlet form} allow us to bound \eqref{sup_energyestimate} by
\begin{equation*}
	\int_0^T\sup_f\left\lbrace C_5+ \left(\frac{1}{4A} - \frac{N}{4}\right)\mathfrak{D}_N(f)+ \tilde{C}\frac{A}{4} \sum_{x=1}^{N-2}G_s\Big(\frac xN\Big)^2 + \widehat{C} - c\|G_s\|_{L^2}^2 \right\rbrace\diff s.
\end{equation*}
Choosing $A=\frac1N$ to remove the dependence with respect to the density $f$ gives that this expression is bounded above by
\begin{equation*}
	T(C_5+\widehat{C}) + \int_0^T\left(\frac{\tilde{C}}{4}\frac{1}{N}\sum_{x=1}^{N-2}G_s\Big(\frac xN\Big)^2 -c\|G_s\|_{L^2}^2\right)\diff s.
\end{equation*}
As $N$ goes to $+\infty$, the integral above converges to $(\tilde{C}/4-c)\dlangle G,G\drangle $, which is non-positive if we choose $c>\tilde{C}/4$. Making this choice of $c$ then proves that the limit in \eqref{energyestimatetodo} is bounded, and it concludes the proof of Lemma \ref{lemma:energyestimate}.
\end{proof}

It remains now to prove the replacement Lemmas \ref{lemma: replacement bulk} and \ref{lemma:replacement_boundary}, which require significant work. Section \ref{sec: Replacement lemma bulk} below is devoted to prove the former, and Section \ref{sec:replacement_boundary} the latter.

\section{Fixing the profile at the boundary for $\theta <1$: Proof of Lemma \ref{lemma:replacement_boundary}}
\label{sec:replacement_boundary}

The proof we give uses similar ideas to the one in \cite{BPGN2020}, namely we use the entropy inequality to be reduced to consider only expectations with respect to our reference measure $\mu^N$, and then we apply the entropy and Dirichlet estimates we proved to get the result. Though, some difficulties arise due to the fact that the reference measure is not product.

Fix $t\in [0,T]$. We want to prove that
\begin{equation*}
	\lim_{N\to +\infty}\E_{\nu_0^N}\left[\left|\int_0^t X(\eta_s)\diff s\right|\right] =0
\end{equation*}
where $X:\Omega_N\longrightarrow\R$ is a local function with is either $h_1-\alpha$ or $h_{N-1}-\beta$. We will only consider the former case because the latter is treated in the exact same way. For this, write
\begin{equation*}
	\E_{\nu_0^N}\left[\left|\int_0^t X(\eta_s)\diff s\right|\right]=\int_{\Omega_N}\E_\eta\left[\left|\int_0^t X(\eta_s)\diff s\right|\right]\diff\nu_0^N(\eta )
\end{equation*}
where $\E_\eta$ denotes the expectation under the law of the process starting from the configuration $\eta$. Recall the definition of our reference measure $\mu^N$, then the entropy inequality and Jensen's inequality allow us to bound it by  
\begin{equation*}
	\frac{H(\nu_0^N|\mu^N)}{\gamma N} + \frac{1}{\gamma N}\log\E_{\mu^N}\left[ \exp\left(\gamma N\left|\int_0^t X(\eta_s)\diff s\right|\right)\right]
\end{equation*}
for any $\gamma >0$. Since $\nu_0^N$ is concentrated on the ergodic component, thanks to Lemma \ref{lem:boundentropy}, the first term is bounded by $C_4\gamma^{-1}$ and vanishes as $\gamma\to +\infty$. It remains only to estimate the second term.  By the inequality $e^{|x|}\le e^x+e^{-x}$ together with the inequality
\begin{equation*}
	\limsup_{N\to +\infty}\frac{1}{N}\log (u_N+v_N) \le \max\left\lbrace \limsup_{N\to +\infty} \frac{1}{N}\log u_N,  \limsup_{N\to +\infty} \frac{1}{N}\log v_N\right\rbrace ,
\end{equation*}
it is sufficient to consider this term without the absolute value.  The Feynman-Kac formula permits to bound it by 
\begin{equation}\label{sup à estimer boundary}
	t\sup_f \left\lbrace  \mu^N(Xf) + \frac{N}{\gamma}\mu^N\big(\sqrt{f}\mathcal{L}_N\sqrt{f}\big)\right\rbrace .
\end{equation}
where the supremum is taken over all density functions $f$ with respect to the measure $\mu^N$. The second term inside this supremum can be estimated by Proposition \ref{prop: estimate Dirichlet form},  we therefore focus on $\mu^N(Xf)$. Note that $X$ has mean 0 under $\mu^N$, indeed
\begin{align*}
	\mu^N(X)=\mu^N\big( \alpha\eta_1+(1-\alpha)\eta_1\eta_2-\alpha \big) & = \alpha\bar\rho (\alpha )+(1-\alpha )\bar\rho (\alpha )\alpha -\alpha = 0
\end{align*}
since $(1-\alpha )\bar\rho (\alpha )=1-\bar\rho (\alpha)$. Denote by $\frak{f}$ the conditional expectation of $f$ under $\mu^N$ with respect to the coordinates $(\eta_1,\eta_2)$, it is a function on $\lbrace 0,1\rbrace^2$. Since $X$ is $(\eta_1,\eta_2)$-measurable and has mean 0 under $\mu^N$, we have that
\begin{align*}
	\mu^N(Xf) = \mu^N(X\frak{f}) & = \int_{\eta '\in\lbrace 0,1\rbrace^2}X(\eta ')\frak{f}(\eta ')\diff\mu^N(\eta ')\\
	& = \int_{\eta '\in\lbrace 0,1\rbrace^2}X(\eta ')\big[\frak{f}(\eta ')-\frak{f}(\bulletone\bullets )\big] \diff\mu^N(\eta ') + \frak{f}(\bulletone\bullets )\int_{\eta '\in\lbrace 0,1\rbrace^2}X(\eta ')\diff\mu^N(\eta ')\\
	& = \int_{\eta '\in\lbrace 0,1\rbrace^2}X(\eta ')\big[\frak{f}(\eta ')-\frak{f}(\bulletone\bullets )\big] \diff\mu^N(\eta ').
\end{align*}
As $\mu^N$ is concentrated on the ergodic component, this integral is in fact a sum of two terms:
\begin{equation*}
	\mu^N(Xf) = X(\circone\bullets )\big[ \frak{f}(\circone\bullets )-\frak{f}(\bulletone\bullets )\big]\mu^N(\circone\bullets )+ X(\bulletone\circs )\big[ \frak{f}(\bulletone\circs )-\frak{f}(\bulletone\bullets )\big]\mu^N(\bulletone\circs )
\end{equation*}
Use now twice Young's inequality 
\begin{equation}\label{Young}
	a-b=\sqrt{A}(\sqrt{a}+\sqrt{b})\times\frac{\sqrt{a}-\sqrt{b}}{\sqrt{A}}\le \frac{A}{2}(\sqrt{a}+\sqrt{b})^2+\frac{(\sqrt{a}-\sqrt{b})^2}{2A}
\end{equation}
which is valid for any $a,b,A>0$,  and the inequality $X(\eta ')\le 2$ to get that
\begin{subequations}
\begin{align}
	\mu^N(Xf) & \le A \big[ \sqrt{\frak{f}(\circone\bullets )}+\sqrt{\frak{f}(\bulletone\bullets )}\big]^2 \mu^N(\circone\bullets ) + A\big[ \sqrt{\frak{f}(\bulletone\circs )}+\sqrt{\frak{f}(\bulletone\bullets )}\big]^2\mu^N(\bulletone\circs ) \label{term +}\\
	& \quad + \frac{1}{A}\big[\sqrt{\frak{f}(\circone\bullets )}-\sqrt{\frak{f}(\bulletone\bullets )}\big]^2\mu^N(\circone\bullets )\label{term - 1}\\
	& \quad +\frac{1}{A}\big[ \sqrt{\frak{f}(\bulletone\circs )}-\sqrt{\frak{f}(\bulletone\bullets )}\big]^2\mu^N(\bulletone\circs ).\label{term - 2}
\end{align}
\end{subequations}
First of all, thanks to the inequality $(a+b)^2\le 2(a^2+b^2)$, we see that the right hand side of \eqref{term +} is bounded above by
\begin{equation*}
	2A\frak{f}(\circone\bullets )\mu^N(\circone\bullets) + 2A\frak{f}(\bulletone\bullets )\mu^N(\bulletone\bullets )\frac{\mu^N(\circone\bullets)}{\mu^N(\bulletone\bullets)} + 2A\frak{f}(\bulletone\circs)\mu^N(\bulletone\circs )+2A\frak{f}(\bulletone\bullets )\mu^N(\bulletone\bullets)\frac{\mu^N(\bulletone\circs)}{\mu^N(\bulletone\bullets)},
\end{equation*}
but
\begin{equation*}
	\frac{\mu^N(\circone\bullets )}{\mu^N(\bulletone \bullets)} = \frac{1-\bar\rho (\alpha )}{\bar\rho (\alpha )\alpha} = \frac{1-\alpha}{\alpha } \qquad\mbox{ and }\qquad \frac{\mu^N(\bulletone\circs)}{\mu^N(\bulletone\bullets)}= \frac{\bar\rho (\alpha )(1-\alpha )}{\bar\rho (\alpha )\alpha} = \frac{1-\alpha}{\alpha}
\end{equation*}
so using the fact that $f$ is a density with respect to $\mu^N$, we get that the right hand side of \eqref{term +} can be bounded by
\begin{equation}\label{estim_term+}
	2A + 2A\frac{1-\alpha}{\alpha} + 2A +2A\frac{1-\alpha}{\alpha} = \frac{4A}{\alpha}.
\end{equation}
Now, note that \eqref{term - 1} is equal to
\begin{equation}
	\label{eq:qttyboundary1}
	\frac{1}{A}\frac{1}{\alpha}b_\ell (\circone\bullets)\big[\sqrt{\frak{f}(\circone\bullets )}-\sqrt{\frak{f}(\bulletone\bullets )}\big]^2\mu^N(\circone\bullets ).
\end{equation}
Identifying $\frak{f}$ as a function on $\Omega_N$ by $\frak{f}(\eta):=\frak{f}(\eta_1,\eta_2)$, \eqref{eq:qttyboundary1} can be bounded by $\frak{D}_\ell (\frak{f})/A\alpha$ where $\frak{D}_\ell $ has been defined in \eqref{def: D ell}. Lastly, let us bound \eqref{term - 2} from above by
\begin{multline*}
	\frac{2}{A}\big[ \sqrt{\frak{f}(\bulletone\circs )}-\sqrt{\frak{f}(\circone\bullets )}\big]^2\mu^N(\bulletone\circs )  + \frac{2}{A}\big[ \sqrt{\frak{f}(\circone\bullets )}-\sqrt{\frak{f}(\bulletone\bullets )}\big]^2\mu^N(\bulletone\circs )\\
	 = \frac{2}{A}\frac{1}{\alpha}c_{1,2}(\bulletone\circs )\big[ \sqrt{\frak{f}(\bulletone\circs )}-\sqrt{\frak{f}(\circone\bullets )}\big]^2\mu^N(\bulletone\circs ) + \frac{2}{A}\frac{1}{\alpha} b_\ell (\circone\bullets )\big[ \sqrt{\frak{f}(\circone\bullets )}-\sqrt{\frak{f}(\bulletone\bullets )}\big]^2\mu^N(\circone\bullets )\frac{\mu^N(\bulletone\circs )}{\mu^N(\circone\bullets)},
\end{multline*}
but
\begin{equation*}
	\frac{\mu^N(\bulletone\circs)}{\mu^N(\circone\bullets)} = \frac{\bar\rho (\alpha )(1-\alpha)}{1-\bar\rho (\alpha)}=1
\end{equation*}
and we recognize portions of Dirichlet forms, so we get that \eqref{term - 2} is bounded above by
\begin{equation*}
	\frac{2}{A\alpha}\big(\mathfrak{D}_0^1(\frak{f})+\mathfrak{D}_\ell(\frak{f})\big)
\end{equation*}
where $\mathfrak{D}_0^1$ has been defined in \eqref{def: D0x}. Putting all these estimates together proves
\begin{equation}\label{estimation mu(Xf)}
	\mu^N(Xf) \le \frac{4A}{\alpha} + \frac{3}{A\alpha}\mathfrak{D}_\ell(\frak{f})+\frac{2}{A\alpha}\mathfrak{D}_0^1(\frak{f}).
\end{equation}
But by definition of the total Dirichlet form $\mathfrak{D}_N(\frak{f})$ in \eqref{def: DN}, we clearly have
\begin{equation*}
	\mathfrak{D}_0^1(\frak{f})\le \mathfrak{D}_N(\frak{f})\qquad\mbox{ and }\qquad \frak{D}_\ell (\frak{f})=\frac{N^\theta}{\kappa }\frac{\kappa}{N^\theta}\frak{D}_\ell(\frak{f})\le \frac{N^\theta}{\kappa}\frak{D}_N(\frak{f}),
\end{equation*}
so finally
\begin{equation}\label{estimation mu(Xf)2}
	\mu^N(Xf) \le \frac{4A}{\alpha} + \frac{3N^\theta +2\kappa}{\kappa A\alpha}\mathfrak{D}_N(\frak{f}).
\end{equation}
Since the conditional expectation $\frak{f}$ is an average and the Dirichlet form $\mathfrak{D}_N$ is convex, we have $\mathfrak{D}_N(\frak{f})\le\mathfrak{D}_N(f)$. Using this, together with \eqref{estimation mu(Xf)2} and the result of Proposition \ref{prop: estimate Dirichlet form} (see  \eqref{sup à estimer boundary}) yields that the supremum is bounded by
\begin{equation*}
	\sup_f\left\lbrace\frac{4A}{\alpha} + \left(\frac{3N^\theta +2\kappa}{\kappa A\alpha}-\frac{N}{4\gamma}\right)\mathfrak{D}_N(f) +\frac{C_5}{\gamma}\right\rbrace.
\end{equation*}
Choosing
\begin{equation*}
	A=\frac{4\gamma}{\kappa\alpha}\frac{3N^\theta +2\kappa}{N}
\end{equation*}
removes the dependence with respect to $f$ and we deduce that \eqref{sup à estimer boundary} is bounded from above by
\begin{equation*}
	\frac{16\gamma T}{\kappa\alpha^2}\frac{3N^\theta +2\kappa}{N}+\frac{C_5T}{\gamma}.
\end{equation*}
As we chose $\theta <1$, it suffices to make $N$ go to $+\infty$ before $\gamma$ to deduce the result. The proof for the replacement on the right boundary follows the exact same steps, and proves Lemma \ref{lemma:replacement_boundary}. \hfill$\square$

\section{Replacement lemma in the bulk: Proof of Lemma \ref{lemma: replacement bulk}}
\label{sec: Replacement lemma bulk}

\subsection{Strategy}

Let us now turn to the proof of the main replacement lemma. Although the replacement lemma in the bulk follows the classical one-block and two-blocks estimates, the lack of translation invariance in the system and the fact that the stationary state is not product induce some technical challenges. In order to handle this problem, we repeatedly make use of Proposition \ref{prop:localequilibrium} stating that our reference measure is locally close to a grand-canonical state, which  is translation invariant. This allows us to reduce the present one-block estimate to the one of \cite{blondel2020hydrodynamic}, and together with the decorrelation estimate of Corollary \ref{cor:2BEdec}, we are also able to prove a two-blocks estimate.

Let us introduce another scaling parameter $\ell$ which will act as an intermediary between the microscopic and the macroscopic scales,  it has to be seen as a parameter smaller than $\varepsilon N$. If we add and subtract the quantity
\begin{equation*}
	\int_0^t \varphi (s)\Big( h_x^\ell (s)-\act\big( \eta_x^\ell (s)\big)\Big)\diff s,
\end{equation*}
where
\begin{equation}
	\label{def:hxell,etaxell}
	\eta_x^\ell = \frac{1}{|\Lambda_x^\ell |}\sum_{y\in\Lambda_x^\ell}\eta_y\qquad\mbox{ and }\qquad h_x^\ell (\eta ):= \frac{1}{|\Lambda_x^{{\ell -1}}|}\sum_{y\in \Lambda_x^{{\ell-1}}} h_y(\eta)
\end{equation}
inside the absolute value of \eqref{expectation replacement bulk}, then the triangle inequality allows us to reduce the proof of the replacement Lemma \ref{lemma: replacement bulk} to three steps, each one consisting in proving that one of the following expressions vanishes:
\begin{align}
	&\sup_{x\in\Lambda_N}\E_{\nu_0^N} \left[\bigg| \int_0^t \varphi(s) \big( h_x(s ) - h_x^\ell (s)\big)\diff s\bigg|\right] ,\label{step 1}\\
	&\sup_{x\in\Lambda_N}\E_{\nu_0^N} \left[\bigg|\int_0^t \varphi(s)\Big( h_x^\ell (s) -\act\big(\eta_x^\ell (s)\big)\Big)\diff s\bigg|\right] ,\label{step 2}\\
	&\sup_{x\in\Lambda_N}\E_{\nu_0^N} \left[ \bigg|\int_0^t\varphi (s) \Big( \act\big(\eta_x^\ell (s)\big) - \act\big( \eta_x^{\varepsilon N}(s)\big)\Big)\diff s\bigg|\right]. \label{step 3}
\end{align}
The first step consists in showing that we can replace each $h_x$ by its average $h_x^\ell$ over the box $\Lambda_x^{\ell -1}$ containing $x$. This is the purpose of Lemma \ref{lemma:firststep_bulk} stated and proved in the next section.

The second step consists in proving that we can replace the empirical average $h_x^\ell$ over a large microscopic box, that is of size $\ell$ independent of $N$, by the expected value of $h_x$ under the grand-canonical measure with density $\eta_x^\ell$, namely $\act(\eta_x^\ell)$. This is the content of the \emph{one-block estimate} given in Lemma \ref{lemma:1BE} below.

The third and last step consists in replacing $\act (\eta_x^\ell)$ by $\act(\eta_x^{\varepsilon N})$. Using the fact that the function $\act$ is 4-Lipschitz on $[\frac12,1]$, it is enough to show that we can replace $\eta_x^\ell$ by $\eta_x^{\varepsilon N}$. In other words, we prove that the density of particles over large microscopic boxes (of size $\ell$) is close to the density of particles over small macroscopic boxes (of size $\varepsilon N$). This is the aim of the \emph{two-blocks estimate} given in Lemma \ref{lemma:2BE} below.

To sum up, the replacements we make are the following.

\begin{figure}[!h]
    \centering
    \begin{tikzpicture}
    \draw (0,0) node[] {$h_x(\eta )$};
    \draw (3,0) node[] {$h_x^\ell (\eta )$};
    \draw (6,0) node[] {$\mathfrak{a}(\eta_x^\ell )$};
    \draw (9.1,0) node[] {$\mathfrak{a}(\eta_x^{\varepsilon N})$};
    \draw [->,thick] (.5,0)--(2.5,0);
    \draw [->,thick] (3.5,0)--(5.5,0);
    \draw [->,thick] (6.5,0)--(8.5,0);
    \draw (1.5,0) node[anchor=south] {Lemma \ref{lemma:firststep_bulk}};
    \draw (4.5,0) node[anchor=south] {Lemma \ref{lemma:1BE}};
    \draw (7.5,0) node[anchor=south] {Lemma \ref{lemma:2BE}};
    \end{tikzpicture}
\end{figure}

\subsection{First step}

\begin{lemma}
	\label{lemma:firststep_bulk}
	For all $t\in [0,T]$ and for any continuous function $\varphi : [0,T]\longrightarrow\R$, we have
	\begin{equation}\label{eq:firststep_bulk}
		\limsup_{\ell\to +\infty}\limsup_{N\to +\infty}\sup_{x\in\Lambda_N}\E_{\nu_0^N} \left[\bigg| \int_0^t \varphi(s) \big( h_x(s ) - h_x^\ell (s)\big)\diff s\bigg|\right] =0.
	\end{equation}
\end{lemma}

\begin{proof}
	Fix $x\in \Lambda_N$. Using the same method used before coupling the entropy inequality, Jensen's inequality, together with Feynman-Kac formula and the entropy bound of Lemma \ref{lem:boundentropy}, we see that the expectation in \eqref{eq:firststep_bulk} can be bounded above by
	\begin{equation}\label{firststep_supàestimer}
		\frac{C_4}{\gamma} + \int_0^t \sup_f \left\lbrace \varphi (s)\mu^N\big( (h_x-h_x^\ell) f\big) +\frac N\gamma \mu^N(\sqrt{f}\mathcal{L}_N\sqrt{f})\right\rbrace\diff s
	\end{equation}
	for all $\gamma >0$, where the supremum is taken over all density fonctions $f$ with respect to $\mu^N$. Note that
	\begin{align*}
		h_x(\eta )-h_x^\ell (\eta ) = \frac{1}{|\Lambda_x^{\ell -1}|} \sum_{y\in\Lambda_x^{\ell -1}} \big( h_x(\eta )-h_y (\eta )\big) & =  \frac{1}{|\Lambda_x^{\ell -1}|} \sum_{y\in\Lambda_x^{\ell -1}}  \sum_{z=y}^{x-1} \big( h_{z+1}(\eta )-h_z(\eta )\big)\\
		& = \frac{1}{|\Lambda_x^{\ell -1}|} \sum_{y\in\Lambda_x^{\ell -1}}  \sum_{z=y}^{x-1} -j_{z,z+1}(\eta )\\
		& =\frac{1}{|\Lambda_x^{\ell -1}|} \sum_{y\in\Lambda_x^{\ell -1}} \sum_{z=y}^{x-1} c_{z,z+1}(\eta )(\eta_{z+1}-\eta_z).
	\end{align*}
	recalling the gradient conditions \eqref{eq:gradientcondition} and the definition of the instantaneous current \eqref{def:current}. Thus, the term $I:=\varphi (s)\mu^N\big( (h_x-h_x^\ell )f\big)$ rewrites
	\begin{align*}
		I & = \frac{\varphi (s)}{|\Lambda_x^{\ell -1}|} \sum_{y\in\Lambda_x^{\ell -1}} \sum_{z=y}^{x-1} \int_{\Omega_N}c_{z,z+1}(\eta )(\eta_{z+1}-\eta_z)f(\eta )\diff\mu^N (\eta )\\
		& = \frac{\varphi (s)}{2|\Lambda_x^{\ell -1}|} \sum_{y\in\Lambda_x^{\ell -1}} \sum_{z=y}^{x-1}\int_{\Omega_N} c_{z,z+1}(\eta )(\eta_{z+1}-\eta_z)f(\eta )\diff\mu^N(\eta )\\
		& \qquad +\frac{\varphi (s)}{2|\Lambda_x^{\ell -1}|} \sum_{y\in\Lambda_x^{\ell -1}} \sum_{z=y}^{x-1}\int_{\Omega_N} c_{z,z+1}(\eta^{z,z+1})(\eta_z-\eta_{z+1})f(\eta^{z,z+1})\frac{\mu^N(\eta^{z,z+1})}{\mu^N(\eta )}\diff\mu^N(\eta )
	\end{align*}
	where to obtain the last line, we wrote the integral as twice its half, and we performed the change of variable $\eta\rightsquigarrow \eta^{z,z+1}$ in the second half. These integrals are actually integrals over the set $\Omega_N^z = \big\lbrace \eta\in\mathcal{E}_N\; :\; c_{z,z+1}(\eta )\neq 0\big\rbrace$, for which we have the ``quasi-reversibility'' relation \eqref{eq:quasi-reversibility} given in Lemma \ref{lem:quasi-reversibility} that reads
	\begin{equation*}
		c_{z,z+1}(\eta^{z,z+1})\frac{\mu^N(\eta^{z,z+1})}{\mu^N(\eta )}= c_{z,z+1}(\eta )+ \mathcal{O}\left(\frac 1N\right) .
	\end{equation*}
	Injecting this in the last expression of $I$, using the fact that there is at most one particle per site and that $f$ is a density with respect to $\mu^N$, we can see that $I$ writes under the form
	\begin{equation*}
		I=\frac{\varphi (s)}{2|\Lambda_x^{\ell -1}|} \sum_{y\in\Lambda_x^{\ell -1}} \sum_{z=y}^{x-1}\int_{\Omega_N^z}c_{z,z+1}(\eta )(\eta_{z+1}-\eta_z)\big[f(\eta )-f(\eta^{z,z+1})\big]\diff\mu^N(\eta ) +\mathcal{O}\left(\frac{\ell}{N}\right)
	\end{equation*}
	where it is important to mention that the error term depends neither on the function $f$, nor on the coordinate $z$
	. Now, to bound the remaining sum, let us use Young's inequality which for $A>0$ writes
	\begin{align*}
		\varphi (s)(\eta_{z+1}-\eta_z)\big[f(\eta )-f(\eta^{z,z+1})\big] & \leq \frac{1}{2A}\big[\sqrt{f (\eta )}-\sqrt{f (\eta^{z,z+1})}\big]^2 \\
		& \qquad\qquad+ \frac A2 \varphi (s)^2 \underbrace{(\eta_{z+1}-\eta_z)^2}_{=1} \big[\sqrt{f(\eta )}+\sqrt{f(\eta^{z,z+1})}\big]^2
	\end{align*}
	so that we get that this sum is bounded by
	\begin{multline*}
		\frac{1}{4A|\Lambda_x^{\ell -1}|} \sum_{y\in\Lambda_x^{\ell -1}}\sum_{z=y}^{x-1} \int_{\Omega_N}c_{z,z+1}(\eta )\big[ \sqrt{f(\eta^{z,z+1})}-\sqrt{f(\eta )}\big]^2\diff\mu^N(\eta )\\
		+\frac{A\varphi (s)^2}{4|\Lambda_x^{\ell -1}|}\sum_{y\in\Lambda_x^{\ell-1}}\sum_{z=y}^{x-1} \int_{\Omega_N}c_{z,z+1}(\eta )\big[\sqrt{f(\eta^{z,z+1})}+\sqrt{f(\eta )}\big]^2\diff\mu^N(\eta ).
	\end{multline*}
	The inequalities $c_{z,z+1}(\eta )\le 1$ and $(a+b)^2\le 2(a^2+b^2)$, together with the fact that $f$ is a density show that the second term in this expression is bounded above by $\tilde{C}A\ell\|\varphi\|_\infty^2$ where $\tilde{C}$ is a positive constant. In the first term, the integral is exactly equal to $\mathfrak{D}_0^z(f)$ defined in \eqref{def: D0x}, so the sum over $z$ is a piece of the total Dirichlet form by which it can be bounded. Hence we get that the first term is bounded above by $\mathfrak{D}_N(f)/4A$. To sum up, we have proved that
	\begin{equation*}
		I\le \frac{1}{4A}\mathfrak{D}_N(f)+ \tilde{C}A\ell\|\varphi\|_\infty^2 + \mathcal{O}\left(\frac{\ell}{N}\right),
	\end{equation*}
	and this holds for any $A>0$. Injecting this in \eqref{firststep_supàestimer} and using the result of Proposition \ref{prop: estimate Dirichlet form} we see that it can be bounded above by
	\begin{equation*}
		\frac{C_4}{\gamma} + t\sup_f\left\lbrace\frac{1}{4A}\mathfrak{D}_N(f)+\tilde{C}A\ell \|\varphi\|_\infty^2 + \mathcal{O}\left(\frac{\ell}{N}\right) - \frac{N}{4\gamma}\mathfrak{D}_N(f)+ \frac{C_5}{\gamma}\right\rbrace .
	\end{equation*}
	Making the choice $A = \frac{\gamma}{N}$ removes the dependence with respect to the function $f$, so that we have the bound
	\begin{equation*}
		\E_{\nu_0^N}\left[\left| \int_0^t\varphi (s)\big( h_x(s)-h_x^\ell (s)\big)\diff s\right|\right] \le \frac{C_4}{\gamma} + \frac{\tilde{C}T\ell\gamma\|\varphi\|_\infty^2}{N}+\mathcal{O}\left(\frac{\ell}{N}\right) + \frac{TC_5}{\gamma }
	\end{equation*}
	and this holds uniformly in $x\in\Lambda_N$ so we can take the supremum on the left hand side of this inequality. Letting $N$, then $\ell$ and finally $\gamma$ go to infinity, we obtain the desired result.
\end{proof}

\subsection{One-block estimate}

Using the triangle inequality and the fact that the function $\varphi$ is bounded, if we want to prove that \eqref{step 2} vanishes as $N$ and $\ell$ go to $+\infty$, it is sufficient to prove the following result.

\begin{lemma}[One-block estimate]
	\label{lemma:1BE}
	For any $t\in [0,T]$, we have that
	\begin{equation}
		\limsup_{\ell\to +\infty}\limsup_{N\to +\infty}\sup_{x\in\Lambda_N}\E_{\nu_0^N}\left[\int_0^t\big| V_x^\ell (s)\big|\diff s\right] =0,
	\end{equation}
	where $V_x^\ell$ is defined by
	\begin{equation}
		\label{def:Vxell}
		V_x^\ell(\eta )= h_x^\ell (\eta )-\act (\eta_x^\ell).
	\end{equation}
\end{lemma}

\begin{remark}\upshape
	More generally, we expect that this one-block estimate should hold for any local function $\psi :\Omega_N\longrightarrow \R$, stating that an empirical average of $\psi$ over a large microscopic box can be replaced by the average of $\psi$ under the grand-canonical measure associated to the empirical density over this box. Though, we state it directly for the local function $h_x$ (defined in \eqref{def:hx}) as it is sufficient for our purpose.
\end{remark}

\begin{proof}
	Fix $x\in\Lambda_N$ and recall the definition of the box $\Lambda_x^\ell$ in \eqref{def:lambdaxell}. Making use of the Feynman-Kac formula as before, and using Lemma \ref{lem:boundentropy} and Proposition \ref{prop: estimate Dirichlet form} we can bound the expectation of the statement by
	\begin{equation}\label{1BEsup1}
		\frac{C_4}{\gamma} + t\sup_f\left\lbrace \mu^N\big( |V_x^\ell |f\big) - \frac{N}{4\gamma}\mathfrak{D}_N(f)\right\rbrace+ \frac{tC_5}{\gamma}
	\end{equation}
	for any $\gamma >0$, where the supremum is taken over all density functions $f$ with respect to the probability measure $\mu^N$. Note that for any $x\in\Lambda_N$, both quantities $h_x^\ell$ and $\eta_x^\ell$ defined in \eqref{def:hxell,etaxell} depend only on the coordinates in the box $\Lambda_x^\ell$. Let us introduce two additional notations:
	\begin{itemize}
		\item We denote by $\widehat{\mu}_x^\ell$ the restriction of the measure $\mu^N$ to the box $\Lambda_x^\ell$ : 
		\begin{equation*}
			\forall \sigma\in\lbrace 0,1\rbrace^{\Lambda_x^\ell },\qquad \widehat{\mu}_x^\ell (\sigma )=\mu^N(\eta_{|\Lambda_x^\ell}=\sigma ).
		\end{equation*}

		\item If $g:\lbrace 0,1\rbrace^{\Lambda_x^\ell}\longrightarrow [0,+\infty ]$ is a density with respect to $\widehat{\mu}_x^\ell$, then we define the Dirichlet form on the box $\Lambda_x^\ell$ by
		\begin{equation}\label{def: DirForm on Lambda_x^ell}
			\frak{D}_x^\ell (g)=\sum_{\lbrace y,y+1\rbrace\subset\Lambda_x^\ell}I_y^\ell (g)
		\end{equation}
		where
		\begin{equation}\label{def:Iyell}
			I_y^\ell (g)=\int_{\lbrace 0,1\rbrace^{\Lambda_x^\ell}}c_{y,y+1}(\sigma )\big[\sqrt{g(\sigma^{y,y+1})}-\sqrt{g(\sigma)}\big]^2\diff\widehat{\mu}_x^\ell (\sigma).
		\end{equation}
	\end{itemize}
	Note in particular that if $f:\Omega_N\longrightarrow [0,+\infty ]$ is a density with respect to $\mu^N$, and $f_x^\ell = \mu^N(f|\Lambda_x^\ell )$ denotes its conditional expectation with respect to the coordinates in $\Lambda_x^\ell$, then we have that
	\begin{equation*}
		I_y^\ell (f_x^\ell)=\mathfrak{D}_0^y(f_x^\ell)
	\end{equation*}
	where $\mathfrak{D}_0^y$ has been defined in \eqref{def: D0x}, because $f_x^\ell$ can be seen either as a function on $\lbrace 0,1\rbrace^{\Lambda_x^\ell}$, or on $\Omega_N$. As a consequence, using the convexity of each $\mathfrak{D}_0^y$ and the fact that $f_x^\ell$ is a conditional expectation, 
	\begin{equation*}
		\frak{D}_x^\ell (f_x^\ell)=\sum_{\lbrace y,y+1\rbrace\subset\Lambda_x^\ell}\mathfrak{D}_0^y (f_x^\ell) \le \sum_{\lbrace y,y+1\rbrace\subset\Lambda_x^\ell}\mathfrak{D}_0^y (f)
	\end{equation*}
	and on the right hand side of the inequality we recognize a piece of the total Dirichlet form, so we can bound it above by $\mathfrak{D}_N(f)$ to obtain that
	\begin{equation}\label{1BEestimationDir}
		\mathfrak{D}_x^\ell (f_x^\ell) \le \mathfrak{D}_N(f).
	\end{equation}
	Note that this bound is extremely crude, since we are bounding $\mathcal{O}(\ell )$ pieces of the Dirichlet form by the total ($N$ pieces) Dirichlet form. Nevertheless, this is sufficient for our purpose here, and is much more convenient in a non translation invariant setting. Each function $V_x^\ell$ depends only on the coordinates inside $\Lambda_x^\ell$, so we have that
	\begin{equation*}
		\mu^N\big( |V_x^\ell |f\big) = \mu^N\big( |V_x^\ell |f_x^\ell\big) = \widehat{\mu}_x^\ell\big( |V_x^\ell |f_x^\ell\big) .
	\end{equation*}
	Using this together with \eqref{1BEestimationDir}, we can bound the supremum in \eqref{1BEsup1} by
	\begin{equation}
		\sup_f\left\lbrace \widehat{\mu}_x^\ell\big( |V_x^\ell |f_x^\ell\big) -\frac{N}{4\gamma}\mathfrak{D}_x^\ell (f_x^\ell)\right\rbrace \le \sup_{x\in\Lambda_N}\sup_g\left\lbrace \widehat{\mu}_x^\ell\big( |V_x^\ell |g\big) - \frac{N}{4\gamma}\mathfrak{D}_x^\ell (g)\right\rbrace
	\end{equation}
	where this time, the supremum is taken over all density functions $g:\lbrace 0,1\rbrace^{\Lambda_x^\ell}\longrightarrow [0,+\infty ]$ with respect to the measure $\widehat{\mu}_x^\ell$. As the left hand term inside the supremum is non-negative and bounded uniformly in $x$, say by some constant $K>0$, the regime where $\mathfrak{D}_x^\ell (g)$ is larger than $\frac{4K\gamma}{N}$ does not contribute to the supremum and we can restrict it to densities $g$ satisfying $\mathfrak{D}_x^\ell (g)\le \frac{4K\gamma}{N}$. We are thus left to estimate
	\begin{equation}\label{1BEsup2}
		\sup_{x\in\Lambda_N} \sup_{g : \mathfrak{D}_x^\ell (g)\le\frac{4K\gamma}{N}} \widehat{\mu}_x^\ell \big( |V_x^\ell |g\big)
	\end{equation}
	since the Dirichlet form is non-negative. At this stage, we have a uniform bound with respect to $x$ so we can take the supremum over $x\in\Lambda_N$ of the expectations on the left hand side of the inequality. If $x\in\Lambda_N$ is written $x=\lfloor uN\rfloor$ for some $u\in (\frac1N ,1)$, using Proposition \ref{prop:localequilibrium} we have that
	\begin{equation}\label{1BEeqmeasures}
		\forall \sigma\in\lbrace 0,1\rbrace^{\Lambda_x^\ell}, \qquad \big|\widehat{\mu}_x^\ell (\sigma ) - \widehat{\pi}_{\varrho (u)}^\ell (\sigma )\big| \leq \frac{C_1}{N}
	\end{equation}
	where $\varrho (u)$ has been defined in \eqref{def:varrho} and $\widehat{\pi}_{\varrho (u)}^\ell$ is the restriction of $\pi_{\varrho (u)}$ to any box of size $2\ell +1$. If we define a new Dirichlet form with respect to $\widehat{\pi}_{\varrho (u)}^\ell$ by
	\begin{equation*}
		\tilde{\mathfrak{D}}_u^\ell (g)=\sum_{\lbrace y,y+1\rbrace\subset\integers{-\ell}{\ell}} \int_{\lbrace 0,1\rbrace^{2\ell +1}} c_{y,y+1}(\sigma )\big[\sqrt{g(\sigma^{y,y+1})}-\sqrt{g(\sigma )}\big]^2\diff\widehat{\pi}_{\varrho (u)}^\ell (\sigma )
	\end{equation*}
	then inequality \eqref{1BEeqmeasures} together with the fact that $\mathfrak{D}_x^\ell (g)\leq \frac{4K\gamma}{N}$ implies that $\tilde{\mathfrak{D}}_u^\ell (g)\le \frac{K'}{N}$ for another constant $K'>0$ that depends only on $\alpha ,\beta ,\gamma$ and $\ell$. As a consequence, if we want to estimate \eqref{1BEsup2}, it suffices to estimate
	\begin{equation}
		\sup_{u\in [0,1]}\sup_{g : \tilde{\mathfrak{D}}_u^\ell (g)\le \frac{K'}{N}} \widehat{\pi}_{\varrho (u)}^\ell \big( |\tilde{V}_0^\ell |g\big)
	\end{equation}
	where $\tilde{V}_0^\ell : \lbrace 0,1\rbrace^{2\ell +1}\longrightarrow \R$ is the function defined by $\tilde{V}_0^\ell (\eta )=h_0^\ell (\eta )-\act(\eta_0^\ell)$. As $\varrho (u)$ is bounded away from $\frac12$ and $1$, we can at this stage follow the steps of the proof of \cite[Lemma 7.1]{blondel2020hydrodynamic} to conclude the proof.
\end{proof}

\subsection{Two-blocks estimate}

The two-blocks estimate hereafter states that the density of particles over large microscopic boxes and small macroscopic boxes are close. The strategy to prove this result is to show that the density of particles over any two large microscopic boxes, at small macroscopic distance, are close to each other. To do so, we choose those microscopic boxes far enough to be uncorrelated by Corollary \ref{cor:2BEdec}, and use the fact that they are macroscopically close to ensure, by Proposition \ref{prop:localequilibrium}, that the reference measure on them is close to one single grand-canonical state. Thus, the density of particles over these two boxes should not differ much.

\begin{lemma}\label{lemma:2BE}
	For any $t\in [0,T]$, we have
	\begin{equation}
		\limsup_{\ell\to +\infty}\limsup_{\varepsilon\to 0}\limsup_{N\to +\infty}\sup_{x\in\Lambda_N} \E_{\nu_0^N}\left[ \int_0^t \big| \eta_x^\ell (s)-\eta_x^{\varepsilon N}(s)\big|\diff s\right] =0.
	\end{equation}
\end{lemma}

\begin{proof}
	The ideas in the proof of the two-blocks estimate are similar to the ones used in the proof of the one-block estimate so we solely sketch some classical ones, and detail some others. Fix $x\in\Lambda_N$. Using once again the Feynman-Kac formula together with Lemma \ref{lem:boundentropy} and Proposition \ref{prop: estimate Dirichlet form}, we can bound the expectation of the statement by
	\begin{equation}\label{2BEsup1}
		\frac{C_4}{\gamma} + t\sup_f \left\lbrace \mu^N\big( |\eta_x^\ell - \eta_x^{\varepsilon N}|f\big) -\frac{N}{4\gamma}\mathfrak{D}_N(f)\right\rbrace +\frac{tC_5}{\gamma}
	\end{equation}
	where the supremum is taken over all density functions $f$ with respect to $\mu^N$. Let us divide the box of size $2\varepsilon N+1$ relative to $\eta_x^{\varepsilon N}$ into $p=\big\lfloor\frac{2\varepsilon N+1}{2\ell +1}\big\rfloor$ boxes of size $2\ell +1$, plus possibly two leftover blocks whose size is strictly less than $2\ell +1$. It permits to write that
	\begin{equation*}
		\eta_x^\ell - \eta_x^{\varepsilon N} = \frac1p \sum_{j=-p/2}^{p/2}(\eta_x^\ell - \eta_{x+j(2\ell +1)}^\ell)
	\end{equation*}
	plus potentially an error term that we omit since it will vanish as $N$ goes to $+\infty$. We can remove the terms for $|j|<\frac{(\log N)^2}{2\ell +1}$ because they have a contribution of order $\mathcal{O}\big( \frac{(\log N)^2}{\varepsilon N}\big)$ which vanishes with $N$. Doing this truncation, we are in the conditions of validity of Corollary \ref{cor:2BEdec} and this will be useful later on. Instead of the supremum of \eqref{2BEsup1}, we are left to estimate
	\begin{equation}\label{2BEsup2}
		\sup_f \left\lbrace \frac1p \sum_{j\in \mathfrak{T}} \mu^N\big( |\eta_x^\ell -\eta_{x+j(2\ell +1)}^\ell |f\big) - \frac{N}{4\gamma}\mathfrak{D}_N(f)\right\rbrace
	\end{equation}
	where
	\begin{equation*}
		\mathfrak{T}=\mathfrak{T}(N,\varepsilon ,\ell ):=\left\lbrace j\in\Z \; :\; \frac{(\log N)^2}{2\ell +1}\le |j|\le \frac p2\right\rbrace .
	\end{equation*}
	For the sake of simplicity, we define \nota{$y_j=x+j(2\ell +1)$}. The map $\eta\longmapsto |\eta_x^\ell -\eta_{y_j}^\ell|$ depends only on the coordinates in \nota{$\Lambda_{x,j}^\ell := \Lambda_x^\ell\cup\Lambda_{y_j}^\ell$}. If $f$ is a density with respect to $\mu^N$, we denote by $\nota{f_{x,j}^\ell := \mu^N(f|\Lambda_{x,j}^\ell)}$ its conditional expectation with respect to the coordinates in $\Lambda_{x,j}^\ell$. The objective will be, as before, to define a Dirichlet form on $\Lambda_{x,j}^\ell$ and to estimate it by the total Dirichlet form $\mathfrak{D}_N(f)$. We introduce the following notations:
	\begin{itemize}
		\item $\widehat{\mu}_{x,j}^\ell$ is the restriction of the measure $\mu^N$ to the box $\Lambda_{x,j}^\ell$ ;
		\item If $g:\lbrace 0,1\rbrace^{\Lambda_{x,j}^\ell}\longrightarrow\R$ is a density with respect to $\widehat{\mu}_{x,j}^\ell$, then we define the Dirichlet form on $\Lambda_{x,j}^\ell$ by
		\begin{equation}\label{def:Dxjell}
			\mathfrak{D}_{x,j}^\ell (g) := J_{x,y_j}(g)+\sum_{\lbrace z,z+1\rbrace \subset\Lambda_{x,j}^\ell}I_z^\ell (g)
		\end{equation}
		where $I_z^\ell$ has been defined in \eqref{def:Iyell} and $J_{x,y_j}$ is a term that permits to connect the two boxes by allowing a jump from one to the other, while being sure that one never leaves the ergodic component:
		\begin{equation}
			J_{x,y_j}(g)=\int_{\lbrace 0,1\rbrace^{\Lambda_{x,j}^\ell}} (\sigma_{x-1}\sigma_x\sigma_{x+1}+\sigma_{y_j-1}\sigma_{y_j}\sigma_{y_j+1})\big[\sqrt{g(\sigma^{x,y_j})}-\sqrt{g(\sigma )}\big]^2\diff\widehat{\mu}_{x,j}^\ell (\sigma ).
		\end{equation}
	\end{itemize}
	When $f$ is a density with respect to $\mu^N$, our first goal is to estimate
	\begin{equation*}
		\frac1p\sum_{j\in\mathfrak{T}}\mathfrak{D}_{x,j}^\ell (f_{x,j}^\ell)
	\end{equation*}
	by the total Dirichlet form $\mathfrak{D}_N(f)$. If we perform the same proof as in the one-block estimate, we can see that
	\begin{equation}\label{2BEestdir1}
		\frac1p \sum_{j\in\mathfrak{T}}\sum_{\lbrace z,z+1\rbrace\subset \Lambda_{x,j}^\ell}I_z^\ell (f_{x,j}^\ell)\le \mathfrak{D}_N(f)
	\end{equation}
	so we only have to estimate
	\begin{equation*}
		\frac1p\sum_{j\in\mathfrak{T}}J_{x,y_j}(f_{x,j}^\ell).
	\end{equation*}
	We can extend $J_{x,y_j}$ to a term $\bar{J}_{x,y_j}$ on the whole space by the formula
	\begin{equation*}
		\bar{J}_{x,y_j}(f)=\int_{\Omega_N}(\eta_{x-1}\eta_x\eta_{x+1}+\eta_{y_j-1}\eta_{y_j}\eta_{y_j+1})\big[ \sqrt{f(\eta^{x,y_j})}-\sqrt{f(\eta )}\big]^2\diff\mu^N(\eta )
	\end{equation*}
	so that, if we see $f_{x,j}^\ell$ either as a function on $\lbrace 0,1\rbrace^{\Lambda_{x,j}^\ell}$ or as a function on $\Omega_N$, we have the equality $J_{x,y_j}(f_{x,j}^\ell) = \bar{J}_{x,y_j}(f_{x,j}^\ell)$ and the convexity of $\bar{J}_{x,y_j}$ yields the bound
	\begin{equation*}
		J_{x,y_j}(f_{x,j}^\ell)\le \bar{J}_{x,y_j}(f).
	\end{equation*}
	Note that in the expression of $\bar{J}_{x,y_j}$, we integrate only over configurations that are not alternate, and for which the occupation variables at $x$ and $y_j$ are distinct. As a consequence, it is possible to make a particle lying at $x$ go to $y_j$ (or the converse) as explained in Appendix \ref{sec:apptransition}. More precisely, by Lemma \ref{lemma:longjumppath} defining the integer $n(j)= 3|j|(2\ell +1)-4$, we can find a deterministic sequence of sites $(z_k)_{1\le k\le n(j)}$ in $\integers{x}{y_j}$ such that 
	\begin{equation*}
		\eta^{(0)}=\eta,\qquad \eta^{(k+1)}=(\eta^{(k)})^{z_k,z_k+1},\qquad \eta^{(n(j))}=\eta^{x,y_j}
	\end{equation*}
	and for all $k$, if $c_{z_k,z_k+1}(\eta^{(k)})= 0$ then $\eta^{(k+1)}=\eta^{(k)}$. It allows us to write that 
	\begin{align*}
		\big[\sqrt{f(\eta^{x,y_j})}-\sqrt{f(\eta )}\big]^2 & = \left(\sum_{k=0}^{n(j)-1} \Big[\sqrt{f(\eta^{(k+1)})}-\sqrt{f(\eta^{(k)})}\Big]\right)^2\\
		& \le n(j)\sum_{j=0}^{n(j)-1} \Big[\sqrt{f(\eta^{(k+1)})}-\sqrt{f(\eta^{(k)})}\Big]^2
	\end{align*}
	using Cauchy-Schwarz inequality. In particular, a piece of the total Dirichlet form appears, and we get the bound
	\begin{equation}\label{2BEestdir2}
		\bar{J}_{x,y_j}(f)\le 3|j|(2\ell +1)\mathfrak{D}_N(f) \le \frac{3}{2} p(2\ell +1)\mathfrak{D}_N(f)\le \frac32(2\varepsilon N +1)\mathfrak{D}_N(f).
	\end{equation}
	If we inject \eqref{2BEestdir1} and \eqref{2BEestdir2} in the definition \eqref{def:Dxjell} of $\mathfrak{D}_{x,j}^\ell (f_{x,j}^\ell)$, we get that
	\begin{equation}
		\frac1p \sum_{j\in\mathfrak{T}}\mathfrak{D}_{x,j}^\ell (f_{x,j}^\ell) \le \frac{6\varepsilon N+5}{2}\mathfrak{D}_N(f)\le 6 \varepsilon N \mathfrak{D}_N(f)
	\end{equation}
	as $\varepsilon N\ge 1$. Therefore, the supremum \eqref{2BEsup2} can be bounded above by
	\begin{equation*}
		\sup_f\left\lbrace\frac1p\sum_{j\in\mathfrak{T}} \Big( \widehat{\mu}_{x,j}^\ell \big( |\eta_x^\ell -\eta_{y_j}^\ell |f_{x,j}^\ell\big) - \frac{1}{24\gamma \varepsilon}\mathfrak{D}_{x,j}^\ell (f_{x,j}^\ell)\Big)\right\rbrace .
	\end{equation*}
	which is easily seen to be bounded above by
	\begin{equation*}
		\sup_{x\in \Lambda_N}\sup_{j\in\mathfrak{T}}\sup_g \left\lbrace \widehat{\mu}_{x,j}^\ell\big( |\eta_x^\ell -\eta_{y_j}^\ell |g\big) - \frac{1}{24\gamma \varepsilon}\mathfrak{D}_{x,j}^\ell (g)\right\rbrace
	\end{equation*}
	where this time the supremum is taken over density functions $g:\lbrace 0,1\rbrace^{\Lambda_{x,j}^\ell}\longrightarrow [0,+\infty ]$ with respect to the measure $\widehat{\mu}_{x,j}^\ell$. At this stage, we have a bound that is uniform on $x$, so we can take the supremum of the expectations over $x\in\Lambda_N$ in the left hand side of the inequality. As before, since the left hand term inside this supremum is bounded above, say by some constant $K>0$, we can truncate the supremum to functions $g$ that satisfy $\mathfrak{D}_{x,j}^\ell (g)\le 24K\gamma \varepsilon$ which is a correct order as it will vanish when we will make $\varepsilon$ go to 0. Writing $x=\lfloor uN\rfloor$ and $y_j=\lfloor v_jN\rfloor$, since $|y_j-x|>(\log N)^2$ we are in position to apply Corollary \ref{cor:2BEdec} which allows to replace the measure $\widehat{\mu}_{x,j}^\ell$ above by the measure $\widehat{\pi}_{\varrho (u)}^\ell\otimes \widehat{\pi}_{\varrho (v_j)}^\ell$ defined by
	\begin{equation*}
		\forall (\sigma ,\sigma ')\in \lbrace 0,1\rbrace^{\Lambda_{x,j}^\ell},\qquad \widehat{\pi}_{\varrho (u)}^\ell\otimes \widehat{\pi}_{\varrho (v_j)}^\ell (\sigma ,\sigma ') = \widehat{\pi}_{\varrho (u)}^\ell (\eta_{|\integers{-\ell}{\ell}}=\sigma)\widehat{\pi}_{\varrho (v_j)}^\ell (\eta_{|\integers{-\ell}{\ell}}=\sigma ')
	\end{equation*}
	up to an error of order $\mathcal{O}\big( \frac1N\big)$. Using the fact that $|y_j-x| = |j|(2\ell +1)\le \frac{2\varepsilon N+1}{2}$, it is not hard to see that $|u-v_j|=\mathcal{O}\big( \varepsilon +\frac{1}{N}\big)$, and using a proof similar to the one of Proposition \ref{prop:localequilibrium} one can check that the measure $\widehat{\pi}_{\varrho (v_j)}^\ell$ is close to the measure $\widehat{\pi}_{\varrho (u)}^\ell$ with an error of order $\mathcal{O}\big( \varepsilon +\frac{1}{N}\big)$. Putting all these statements together, we have a constant $C=C(\alpha ,\beta ,\ell )>0$ such that
	\begin{equation}\label{2BEestmeasures}
		\forall (\sigma ,\sigma')\in\lbrace 0,1\rbrace^{\Lambda_{x,j}^\ell},\qquad \big| \widehat{\mu}_{x,j}^\ell (\sigma ,\sigma ')-\widehat{\pi}_{\varrho (u)}^\ell\otimes \widehat{\pi}_{\varrho (u)}^\ell (\sigma ,\sigma ')\big| \le C\left(\varepsilon +\frac{1}{N}\right).
	\end{equation}
	As in the proof of the one-block estimate, we can now define a new Dirichlet form $\tilde{\mathfrak{D}}_{u,2}^\ell (g)$ with respect to $\widehat{\pi}_{\varrho (u)}^\ell\otimes\widehat{\pi}_{\varrho (u)}^\ell$, and if $\mathfrak{D}_{x,j}^\ell (g)\le 24K\gamma \varepsilon$, inequality \eqref{2BEestmeasures} implies that $\tilde{\mathfrak{D}}_{u,2}^\ell (g)\le K'(\varepsilon +\frac1N)$ for some constant $K'>0$ that depends only on $\alpha ,\beta ,\gamma$ and $\ell$. Now that we have expressed everything in terms of a measure that no longer depends on $N$ and $\varepsilon$, we can take the limits to be left to prove that
	\begin{equation*}
		\limsup_{\ell\to +\infty}\sup_{u\in [0,1]}\sup_{g : \tilde{\mathfrak{D}}_{u,2}^\ell (g)=0} \frac{1}{\ell} \widehat{\pi}_{\varrho (u)}^\ell\otimes \widehat{\pi}_{\varrho (u)}^\ell \Big( \big| |\sigma |-|\sigma '|\big| g\Big) =0.
	\end{equation*}
	Decomposing now along the hyperplanes with a fixed number of particles, this amounts to proving that
	\begin{equation}\label{2BEconcentration}
		\frac{1}{\ell} \sum_{i=0}^k |2i-k|\nu_{u,k}^\ell (i)\xrightarrow[\ell\to +\infty]{} 0
	\end{equation}
	for all $u\in [0,1]$ and all $k\ge 2\ell +1$, where $\nu_{u,k}^\ell$ is the measure defined by
	\begin{equation*}
		\nu_{u,k}^\ell (i) = \widehat{\pi}_{\varrho (u)}^\ell\otimes \widehat{\pi}_{\varrho (u)}^\ell \big( |\sigma |=i\;\big|\; |\sigma |+ |\sigma '|=k\big).
	\end{equation*}
	Conditioning with respect to the possible values at the borders of the configurations $\sigma$ and $\sigma '$, it is straightforward to prove that $\nu_{u,k}^\ell$ is concentrated around $i\simeq k/2$, so \eqref{2BEconcentration} can easily be deduced. This concludes the proof of Lemma \ref{lemma:2BE}.
\end{proof}

\appendix
\section{Technical results}

\subsection{Irreducibility of the ergodic component}
\label{sec:appergodic}
\begin{proposition}\label{prop: ergodic component irreducible}
The ergodic component $\mathcal{E}_N$ is an irreducible component for the Markov process with generator $\mathcal{L}_N$.
\end{proposition}

\begin{proof}
First, note that if we take an ergodic configuration -- that is with isolated holes --  then, performing jumps authorized by the dynamics of the generator $\mathcal{L}_N$, we go to another configuration with isolated holes. Therefore, $\mathcal{E}_N$ is stable under the dynamics.  Indeed,  if we perform a jump between two sites $x\in\Lambda_N$ and $x+1\in\Lambda_N$,  then we go from a configuration of the form $\bullets\bullets\circs\bullets$ to one of the form $\bullet\circs\bullets\bullets$ on $\llbracket x-1,x+2\rrbracket $ and we do not create two consecutive holes.  If an exchange with a reservoir takes place,  we cannot create two consecutive holes either because we chose reservoirs that can absorb a particle only if it is followed by another particle.

In order to show that $\mathcal{E}_N$ is an irreducible component,  the strategy is to show that any configuration $\eta\in\mathcal{E}_N$ can be connected to the full configuration $\bs{1}$ ($\forall x\in\Lambda_N,\; \bs{1}_x=1$) with jumps authorized by the generator $\mathcal{L}_N$.  To go from a configuration $\eta\in\mathcal{E}_N$ to  $\bs{1}$,  the idea is to take all the particles from left to right,  creating a particle at site 1 as soon as it is possible.  For this,  at each step we choose the first empty site $x$ in $\eta$ starting from the left.  If this site is in contact with the left reservoir,  then we create a particle.  Otherwise,  we let the particle at $x-1$ jump to $x$,  which is indeed possible since $\eta_{x-2}=1$ by minimality of $x$.  Repeating it several times,  we end up in the full configuration $\bs{1}$.

When we consider ergodic configurations,  all the jumps are reversible so if we want to go from the full configuration $\bs{1}$ to any configuration $\eta\in\mathcal{E}_N$,  it is enough to follow backward the path described above.
\end{proof}

\subsection{Deterministic long-range jump of particles}
\label{sec:apptransition}
In this section, we prove the following lemma, that is used to prove the two-blocks estimate.
\begin{lemma}\label{lemma:longjumppath}
Consider a configuration $\eta$ with $\eta_0=1-\eta_\ell=1$ such that $\eta$ and $\eta^{0,\ell}$ are both ergodic. Set $n=3\ell-4$,  there exists a deterministic sequence of $n$ (potentially trivial) nearest-neighbour jumps getting from $\eta$ to $\eta^{0,\ell}$. In other words, we define $\eta^{(0)}=\eta$, there exist $x_1,\dots,x_n$ such that for $1\leq k\leq n-1$,   
\[\eta^{(k)}=\begin{cases}
(\eta^{(k-1)})^{x_k,x_k+1} & \mbox{ if }c_{x_k,x_k+1}(\eta^{(k-1)})\neq 0,\\
\eta^{(k-1)} &\mbox{otherwise},
\end{cases}
\]
 and $\eta^{(n)}=\eta^{0,\ell}$.
\end{lemma}

\begin{proof}
\newcommand{\circzero}{\mbox{\stackunder[1pt]{$\circ$}{\tiny $0$}}}
\newcommand{\bulletzero}{\mbox{\stackunder[1pt]{$\bullet$}{\tiny $0$}}}
\newcommand{\starzero}{\mbox{\stackunder[1pt]{$*$}{\tiny $0$}}}
\newcommand{\circl}{\mbox{\stackunder[1pt]{$\circ$}{\tiny $\ell$}}}
\newcommand{\bulletl}{\mbox{\stackunder[1pt]{$\bullet$}{\tiny $\ell$}}}
\newcommand{\stars}{\!*}
To get from $\eta$ to  $\eta^{0,\ell}$, three series of jumps are needed.
\begin{itemize}
\item 
First, we set $x_1=1$, $x_2=2$, \dots , $x_{\ell-1}=\ell-1$, to get from 
\[\bullets\bulletzero\bullets\stars\stars\stars\stars\stars\!\bullets\circl\bullets\qquad \mbox{ to }\qquad \bullets\bulletzero\stars\stars\stars\stars\stars\!\bullets\circs\bulletl\bullets,\]
where the `` $\stars\stars\stars\stars\stars$'' piece represents an ergodic piece of the configuration containing arbitrary particles and empty sites.
This sequence of jumps is allowed, because a particle can fully cross an ergodic segment, by ignoring jumps towards another particle, and by jumping over empty sites, which is always allowed because the segment being ergodic means there is a particle right behind it, any time it tries to jump over an empty site.
\item For the same reason, we can make another particle fully cross, that will be used afterwards  to make the right empty site travel back to the origin. We therefore let $x_\ell=0, x_{\ell+1}=1,\dots, x_{2\ell-2}=\ell-2$, to get from 
\[ \bullets\bulletzero\stars\stars\stars\stars\stars\!\bullets\circs\bulletl\bullets\qquad \mbox{ to }\qquad  \bullets\starzero\stars\stars\stars\stars\overline{\bullets\circs}\bullets\bulletl\bullets.\]
Now, we only need to make the marked pair  $\overline{\bullets\circs}$ travel back to the origin.
\item To do so, we simply let $x_{2\ell-1}=\ell-3,$ $x_{2\ell}=\ell-4,\dots,x_{3\ell-4}=0$. This case requires a bit more justifications, because empty sites can only travel freely along an ergodic segment until they would meet another empty site. However, the empty site at $\ell-2$ can travel over fully occupied clusters of particles until it stops at some site $ x+1$ because the ultimate particle of the cluster at site $x$ cannot jump over it. This means that locally, successive transitions occur from $\cdots \bullets \bullets \bullets \overline{\bullets\circs}\cdots $ to $\cdots\overline{\bullets\circs} \bullets \bullets \bullets\cdots $, thus effectively making the pair $\overline{\bullets\circs}$ jump leftwards one step at a time, until it encounters an empty site to its left, where it stops at $\cdots \bullets\circs \overline{\bulletx\circs}.$ At this points, the next two jumps  over $(x,x+1)$ and over $(x-1,x)$ are canceled, because the particle at site $x$ cannot jump towards either of the neighbouring empty sites. Taking those two cancelled jumps together, however, the piece $\overline{\bullets\circs}$ has in effect switched places with the identical piece to its left, and thus can be considered having moved two steps left.
We can apply this to make the pair travel over the whole segment ``$\stars\stars\stars\stars\stars$'', thus transitionning from 
\[ \bullets\starzero\stars\stars\stars\stars\overline{\bullets\circs}\bullets\bulletl\bullets \qquad \mbox{ to }\qquad \bullets\overline{\bulletzero\circs}\stars\stars\stars\stars\stars\bullets\bulletl\bullets .\]
The jump $(x_{3\ell-4},x_{3\ell-4}+1)=(0,1)$ is then allowed, and gets us back to the configuration $\eta^{0,\ell}= \bullets \circzero\bullets \stars\stars\stars\stars\stars\!\bullets\bulletl\bullets$ as wanted.
\end{itemize}
\end{proof}

\subsection{Exponential decay of spatial correlations under the Markovian construction}\label{Appendix:CorrDecay}

In this appendix, we prove the following result.

\begin{theorem}[Decorrelation estimate]
  \label{thm:decaycorrelations}
  Let $\rho : [0,1]\longrightarrow (\frac12,1]$ be any continuous profile taking values strictly over $\frac12$. For $x\in \llbracket 2,N-1\rrbracket$, define 
  \begin{equation}\label{def:axAppendix}
    a_x:= \frac{\rho\big( \frac{x}{N}\big)+\rho\big( \frac{x-1}{N}\big)-1}{\rho\big( \frac{x-1}{N}\big)},
  \end{equation}
  and define a measure $\nu_\rho^N$ to be the distribution of an inhomogeneous Markov chain on $\lbrace 0,1\rbrace$ started from $\eta_1 \sim \mathrm{Ber}\big( \rho (\frac1N)\big)$, and with transition probabilities
  \begin{equation}\label{transition_nurho}
    \nu_\rho^N (\eta_{x+1}=1|\eta_x=1)=a_{x+1}\qquad\mbox{ and }\qquad \nu_\rho^N(\eta_{x+1}=1|\eta_x=0)=1
  \end{equation}
  for $ x\in\llbracket 1, N-2\rrbracket$. Then, under the measure $\nu_\rho^N$, spatial correlations decay exponentially fast, meaning that there exist constants $C,c>0$ depending only on the profile $\rho$, such that
  \begin{equation}\label{eq:corrdecay}
    \forall x<y\in \Lambda_N,\qquad \big| \nu_\rho^N (\eta_x=1,\eta_y=1)-\nu_\rho^N(\eta_x=1)\nu_\rho^N(\eta_y=1)\big| \le Ce^{-c(y-x)}.
  \end{equation}
\end{theorem}

To do so, we will use some general results about inhomogeneous Markov chains given in \cite{InhomMC}. The first thing to do is to show that our Markov chain $( \eta_x)_{1\le x\le N-1}$ with law $\nu_\rho^N$ is \emph{uniformly elliptic} according to the following definition. 

\begin{definition}[Uniform ellipticity]\upshape
An inhomogeneous Markov chain $X=(X_x)_{x\ge 1}$ evolving in a state-space $\frak{S}$ with transition kernels $(\pi_{x,x+1})_{x\ge 1}$ is said to be \emph{uniformly elliptic} if there exists a probability measure $\mu_x$ on $\frak{S}$,  a measurable function $p_x:\frak{S}^2\longrightarrow [0,+\infty [$ and a constant $\varepsilon_0\in (0,1)$ called \emph{ellipticity constant} such that for all $x\ge 1$,
\begin{enumerate}[label=(\roman*)]
\item $\pi_{x,x+1}(u,v)=p_x(u,v)\mu_{x+1}(v)$;
\item $0\le p_x\le\dfrac{1}{\varepsilon_0}$;
\item $\displaystyle \int_{\frak{S}} p_x(u,v)p_{x+1}(v,w)\diff\mu_{x+1}(v) >\varepsilon_0.$
\end{enumerate}
\end{definition}

The transition kernels of our Markov chain $( \eta_x)_{1\le x\le N-1}$ write
\[\pi_{x,x+1}=\begin{pmatrix}
0 & 1 \\ 
1-a_{x+1} & a_{x+1}
\end{pmatrix}.\]
We will show that it is uniformly elliptic with $\mu_x$ being the law of $\eta_x$,  that is the law of a Bernoulli random variable with parameter $\rho \big( \frac{x}{N}\big)$,  and with $p_x$ being the function (written in matrix form)
\[p_x = \begin{pmatrix}
0 & \dfrac{1}{\rho\big( \frac{x+1}{N}\big)} \vphantom{\Bigg(} \\ 
\dfrac{1}{\rho\big( \frac{x}{N}\big)} & \dfrac{a_{x+1}}{\rho\big( \frac{x+1}{N}\big)}\vphantom{\Bigg(}
\end{pmatrix} .\]
With these definitions,  one can see that condition (i) is immediately satisfied.  By the hypothesis that $\rho$ is continuous and takes values in $\big( \frac{1}{2},1\big]$, we can find $\varepsilon \mm{\in (0,\frac12)}$ so that it takes values in $\big[\frac{1}{2}+\varepsilon ,1\big]$.  It is not difficult to deduce from it that the active density field $(a_x)_{2\le x\le N-1}$ defined in \eqref{def:axAppendix} takes values in $[2\varepsilon ,1]$.  We clearly have that $0\le p_x\le 2$ so condition (ii) holds.  It remains only to prove condition (iii),  and for that define
\begin{equation*}\varphi_x(u,w) = \int_{\lbrace 0,1\rbrace} p_x(u,v)p_{x+1}(v,w)\diff\mu_{x+1}(v) .\end{equation*}
Let us compute the four possible values of this function :
\begin{itemize}
\item $\varphi_x(0,0) = p_x(0,0)p_{x+1}(0,0)\Big( 1-\rho\big( \frac{x+1}{N}\big)\Big) + p_x(0,1)p_{x+1}(1,0)\rho\big( \frac{x+1}{N}\big) = \dfrac{1}{\rho\big( \frac{x+1}{N}\big)}\ge 1 \vphantom{\Bigg(}$.
\item $\varphi_x(1,0) =  p_x(1,0)p_{x+1}(0,0)\Big( 1-\rho\big( \frac{x+1}{N}\big)\Big) + p_x(1,1)p_{x+1}(1,0)\rho\big( \frac{x+1}{N}\big) = \dfrac{a_{x+1}}{\rho\big( \frac{x+1}{N}\big)}\ge 2\varepsilon\vphantom{\Bigg(}$.
\item $\varphi_x(0,1) =p_x(0,0)p_{x+1}(0,1)\Big( 1-\rho\big( \frac{x+1}{N}\big)\Big) + p_x(0,1)p_{x+1}(1,1)\rho\big( \frac{x+1}{N}\big) = \dfrac{a_{x+2}}{\rho\big( \frac{x+2}{N}\big)}\ge 2\varepsilon\vphantom{\Bigg(}$.
\item $\varphi_x(1,1) =  p_x(1,0)p_{x+1}(0,1)\Big( 1-\rho\big( \frac{x+1}{N}\big)\Big) + p_x(1,1)p_{x+1}(1,1)\rho\big( \frac{x+1}{N}\big) \ge (2\varepsilon )^2\vphantom{\Bigg(}$.
\end{itemize}
As a consequence,  we can choose a suitable ellipticity constant $\varepsilon_0$ for which our chain $( \eta_x)_{1\le x\le N-1}$ is uniformly elliptic.

Now,  define the $\sigma$-algebras $\mathcal{F}_1^x = \sigma\big( \eta_y,\; y\le x\big)$ and $\mathcal{F}_x^\infty = \sigma\big( \eta_y,\; y\ge x\big)$.  Define also
\begin{equation*}\omega (\ell ) =\sup_x \sup \Big\lbrace \big|\nu_\rho^N (A\cap B) - \nu_\rho^N(A)\nu_\rho^N (B)\big| \; : A\in\mathcal{F}_1^x, \; B\in \mathcal{F}_{x+\ell}^\infty\Big\rbrace.\end{equation*}
By \cite[Proposition 1.22]{InhomMC}, since the chain is uniformly elliptic,  then $\omega (\ell )$ is exponentially small with $\ell$.  More precisely,  there exist constants $C,c>0$ depending only on the profile $\rho$,  such that
\begin{equation*}\omega (\ell ) \le Ce^{-c\ell} .\end{equation*}
This proves Theorem \ref{thm:decaycorrelations}. \qed

\subsection{Uniqueness of weak solutions to the hydrodynamic equations}
\label{sec:appuniqueness}

In this appendix, we aim to prove that the hydrodynamic equations with the different boundary conditions we consider admit at most one solution. Before doing it, we make the following important observation.

Throughout this section, we assume that $\rho$ and $\tilde\rho$ are two weak solutions of the fast diffusion equation $\partial_t\rho=\partial_u^2\act (\rho )$ starting from the same initial condition $\rho^\mathrm{ini}$, and with the boundary conditions corresponding to each problem. Recall that we chose the profile $\rho^\mathrm{ini}$ to be continuous and to take values in $(\frac12 ,1]$, and in the case of Dirichlet boundary conditions we also chose $\rho_-,\rho_+\in (\frac 12,1]$. As a consequence, we can find $\varepsilon >0$ such that
\begin{equation*}
	\frac12 +\varepsilon \le \rho^\mathrm{ini} \le 1\qquad\mbox{ and }\qquad \frac12 +\varepsilon \le \rho_-,\rho_+\le 1.
\end{equation*}
Applying the usual theory (see for instance \cite[Section 3.1]{Vazquez}) with these data, we know that there is a maximum principle and then both solutions $\rho ,\tilde\rho$ satisfy the same bounds 
\begin{equation}\label{boundofsolutions}
	\frac12 +\varepsilon \le \rho ,\tilde\rho \le 1
\end{equation}
for any of the boundary conditions we consider. In the sequel, for $(t,u)\in [0,T]\times [0,1]$ we denote
\begin{equation*}
	w_t(u)=\rho_t(u)-\tilde\rho_t(u)\qquad\mbox{ and also }\qquad v_t(u)=\frac{1}{\rho_t(u)\tilde\rho_t(u)}
\end{equation*}
which are both well defined by \eqref{boundofsolutions}, so that $\act (\rho )-\act (\tilde\rho )=w\times v$.

We start by presenting the proof of the uniqueness of weak solutions in the case of Dirichlet boundary conditions, and then we will present the Robin case which also includes the Neumann one.

\subsubsection{Dirichlet case}

Assume that $\rho$ and $\tilde\rho$ are two weak solutions of \eqref{eq:fast_diffusion_equation_Dirichlet} in the sense of Definition \ref{defin:weaksolDirichlet}.
Take a test function $G\in\mathcal{C}_c^{1,2}\big( [0,T]\times (0,1)\big)$ and consider the weak formulation \eqref{eq:weakformulation_Dirichlet} that both $\rho$ and $\tilde\rho$ satisfy. Performing an integration by parts which is allowed by item (i) of the definition, that is $\act (\rho ),\act (\tilde\rho )\in L^2\big([0,T],\mathcal{H}^1\big)$, and making the difference of both equations, one gets that
\begin{equation}\label{eq:uniquenessDir1}
	\langle w_T,G_T\rangle - \int_0^T\langle w_t,\partial_tG_t\rangle\diff t = -\int_0^T\big\langle \partial_u\act (\rho_t) - \partial_u\act (\tilde\rho_t),\partial_uG_t\big\rangle\diff t.
\end{equation}
Note that both sides of this equality are well defined if we only assume that $G\in L^2\big( [0,T],\mathcal{H}_0^1\big)$ and $\partial_tG\in L^2\big( [0,T]\times [0,1]\big)$. In fact, by approximating such functions by smooth and compactly supported ones, and using a limit argument one can see that \eqref{eq:uniquenessDir1} still holds when $G\in L^2\big( [0,T],\mathcal{H}_0^1\big)$ and has an $L^2$ time derivative. The goal is now to suitably choose the test function in this latter functional space so that \eqref{eq:uniquenessDir1} gives the equality of $\rho$ and $\tilde\rho$, and for this we get inspired by \cite[Section 5.3]{Vazquez}. Define
\begin{equation*}
	G_t(u)= \int_t^T w_s(u)v_s(u)\diff s\qquad\qquad \forall (t,u)\in [0,T]\times [0,1]
\end{equation*}
and let us check that $G\in L^2\big( [0,T],\mathcal{H}_0^1\big)$ and $\partial_tG\in L^2\big( [0,T]\times [0,1]\big)$. Since 
\begin{equation*} \partial_tG_t(u)= - w_t(u)v_t(u) = \act\big( \tilde\rho_t(u)\big)-\act\big( \rho_t(u)\big)
\end{equation*}
the latter condition is clearly satisfied. Besides, 
\begin{equation*}\partial_uG_t(u) = \int_t^T \big( \partial_u\act (\rho_s)-\partial_u\act (\tilde\rho_s)\big) (u)\diff s
\end{equation*}
so $\partial_uG \in L^2\big( [0,T]\times [0,1]\big)$ by item (i) of Definition \ref{defin:weaksolDirichlet}. Finally, item (iii) of Definition \ref{defin:weaksolDirichlet} implies that $G_t(0)=G_t(1)=0$ for all $t\in [0,T]$, so putting all these properties together, we deduce that indeed $G$ belongs to $L^2\big( [0,T],\mathcal{H}_0^1)$ and has an $L^2$ time derivative.

Therefore, this function $G$ can be used as a test function in \eqref{eq:uniquenessDir1}, and it yields that
\begin{equation*}
	\int_0^T \langle w_t,w_tv_t\rangle\diff t = -\int_0^T \int_0^1 \big( \partial_u\act (\rho_t)-\partial_u\act (\tilde\rho_t)\big)(u)\left(\int_t^T \big( \partial_u\act (\rho_s)-\partial_u\act (\tilde\rho_s)\big)(u)\diff s\right)\diff u\diff t,
\end{equation*}
which rewrites, by integrating the right hand side,
\begin{equation*}
	\int_0^T \langle w_t^2,v_t\rangle\diff t +\frac12 \int_0^1\left(\int_0^T \big( \partial_u\act (\rho_t)-\partial_u\act (\tilde\rho_t)\big)(u) \diff t\right)^2\diff u =0.
\end{equation*}
As both term are non-negative, this identity yields that $w=0$, \textit{i.e.} $\rho = \tilde\rho$ almost everywhere in $[0,T]\times [0,1]$. It concludes the proof of the uniqueness of weak solutions in the case of Dirichlet boundary conditions. \hfill $\qed$

\subsubsection{Robin and Neumann cases}

We prove here the uniqueness of weak solutions in the case of Robin boundary conditions. We present only this proof because one can check that it can be adapted for Neumann boundary conditions by taking $\kappa =0$. But beforehand, we need the following three technical lemmas whose proofs can be found in \cite[Section 7.2]{BPGN2020}.

\begin{lemma}\label{lemma:UniqRobin1}
	Let $\sigma\in \mathcal{C}^{2,2}\big( [0,T]\times [0,1]\big)$ be a positive function, let $h\in\mathcal{C}^2\big( [0,1]\big)$ be a function such that $h(0)=h(1)=0$, and let $\lambda\ge 0$. Then, for any $t\in (0,T]$, the problem with Robin boundary conditions
	\begin{equation}
		\label{eq:RobinProblem_lemma}
		\begin{cases}
		\partial_t\varphi + \sigma\partial_u^2\varphi = \lambda\varphi  & \mbox{ on } [0,t)\times (0,1), \\ 
		\varphi_t(\cdot )=h(\cdot ), &  \\ 
		\partial_u\varphi_s(0)=\kappa\varphi_s(0) & \mbox{ for all }s\in [0,t),\\
		\partial_u\varphi_s(1)=-\kappa\varphi_s(1) & \mbox{ for all }s\in [0,t),
		\end{cases}
	\end{equation}
	admits a unique solution $\varphi\in \mathcal{C}^{1,2}\big( [0,t]\times [0,1]\big)$. Moreover, if $0\le h\le 1$, then we have
	\begin{equation*}
		0\le \varphi_s(u) \le e^{-\lambda (t-s)}\qquad \forall (s,u)\in [0,t]\times [0,1].
	\end{equation*}
\end{lemma}

\begin{lemma}\label{lemma:UniqRobin2}
	Let $\varphi$ be the unique solution to the problem \eqref{eq:RobinProblem_lemma}. Then, there exists a constant $K=K(\kappa ,h)>0$ such that
	\begin{equation*}
		\int_0^t\int_0^1\sigma_s(u)\big( \partial_u^2\varphi_s(u)\big)^2\diff u\diff s \le K.
	\end{equation*}
\end{lemma}

\begin{lemma}\label{lemma:UniqRobin3}
	Let $\sigma$ be a non-negative, bounded, measurable function on $[0,T]\times [0,1]$. Then, there exists a sequence $(\sigma^k)_{k\in\N}$ of positive functions in $\mathcal{C}^\infty\big( [0,T]\times [0,1]\big)$ such that
	\begin{equation*}
		\frac1k\le \sigma^k\le \|\sigma\|_\infty +\frac1k \qquad\mbox{ and }\qquad \left\|\frac{\sigma -\sigma^k}{\sqrt{\sigma^k}}\right\|_{L^2([0,T]\times [0,1])} \xrightarrow[k\to +\infty]{} 0.
	\end{equation*}
\end{lemma}

Assume that $\rho$ and $\tilde\rho$ are two weak solutions of \eqref{eq:fast_diffusion_equation_Robin} in the sense of Definition \ref{defin:weaksolRobin}. Take $t\in [0,T]$, a test function $G\in\mathcal{C}^{1,2}\big( [0,T]\times [0,1]\big)$ and consider the weak formulation \eqref{eq:weakformulation_Robin} that both $\rho$ and $\tilde\rho$ satisfy. Making the difference of both equations, one gets that
\begin{align*}
	\langle w_t,G_t\rangle - \int_0^t \langle w_s,\partial_tG_s\rangle\diff s = & \; \int_0^t \langle w_sv_s,\partial_u^2G_s\rangle\diff s\\
&	-\int_0^t \Big\lbrace w_s(1)v_s(1)\partial_uG_s(1)-w_s(0)v_s(0)\partial_uG_s(0)\Big\rbrace\diff s \\ & + \kappa\int_0^t \Big\lbrace -w_s(1)v_s(1)G_s(1)-w_s(0)v_s(0)G_s(0)\Big\rbrace\diff s
\end{align*}
and this rewrites exactly
\begin{multline}\label{eq:UniqRobin_diff}
	\langle w_t,G_t\rangle = \int_0^t \langle w_s,\partial_tG_s+v_s\partial_u^2G_s\rangle\diff s - \int_0^t w_s(1)v_s(1)\Big\lbrace \partial_uG_s(1)+\kappa G_s(1)\Big\rbrace \diff s\\
	+\int_0^t w_s(0)v_s(0)\Big\lbrace \partial_uG_s(0)-\kappa G_s(0)\Big\rbrace\diff s.
\end{multline}
The idea is, as before, to choose a suitable test function to deduce the uniqueness of solutions, and this will be done taking a function that satisfies a problem like \eqref{eq:RobinProblem_lemma}. We have no idea about the regularity of the function $v$, but thanks to \eqref{boundofsolutions} we know that $0< v\le 4$. Thanks to Lemma \ref{lemma:UniqRobin3} we can find a sequence of positive functions $(\sigma^k)_{k\in\N}$ in $\mathcal{C}^\infty\big( [0,T]\times [0,1]\big)$ such that
\begin{equation*}
	\frac1k\le \sigma^k\le 4+\frac1k\qquad\mbox{ and }\qquad \frac{\sigma^k-v}{\sqrt{\sigma^k}}\xrightarrow[k\to +\infty]{} 0 \quad\mbox{ in } L^2\big( [0,T]\times [0,1]\big).
\end{equation*}
Fix a function $h\in\mathcal{C}^2\big( [0,1]\big)$ with $h(0)=h(1)=0$, and consider the problem \eqref{eq:RobinProblem_lemma} when we replace $\sigma$ by $\sigma^k$ and $\lambda$ by 0. By Lemma \ref{lemma:UniqRobin1}, we know that it admits a unique solution $\varphi^k$ and that it belongs to $\mathcal{C}^{1,2}\big( [0,t]\times [0,1]\big)$. If we use this function as test function in \eqref{eq:UniqRobin_diff}, then the last two integrals in the right hand side vanish because
\begin{equation*}
	\partial_u\varphi_s^k(0)=\kappa \varphi_s^k(0)\qquad\mbox{ and }\qquad \partial_u\varphi_s^k(1)=-\kappa\varphi_s^k(1).
\end{equation*}
We can estimate the remaining term using Cauchy-Schwarz inequality in the following way
\begin{align*}
	\int_0^t \langle w_s,\partial_t\varphi_s^k + v_s\partial_u^2\varphi_s^k\rangle\diff s & = \int_0^t \langle w_s,\underbrace{\partial_t\varphi_s^k + \sigma_s^k\partial_u^2 \varphi_s^k}_{=0}\rangle\diff s + \int_0^t \big\langle w_s, (v_s-\sigma_s^k)\partial_u^2\varphi_s^k\big\rangle\diff s\\
	& \le \int_0^t \left\| w_s\frac{v_s-\sigma_s^k}{\sqrt{\sigma_s^k}}\right\|_{L^2([0,1])} \left\| \sqrt{\sigma_s^k}\partial_u^2\varphi_s^k\right\|_{L^2([0,1])}\diff s.
\end{align*}
By \eqref{boundofsolutions}, we have that $|w_s| = |\rho_s-\tilde\rho_s|\le 2$, and by Lemma \ref{lemma:UniqRobin2} we know that the second $L^2$ norm in the integral above is bounded by $\sqrt{K}$, where $K$ is some constant that depends only on $\kappa$ and $h$. Therefore, putting this in \eqref{eq:UniqRobin_diff} we deduce that
\begin{equation*}
	\langle w_t,\varphi_t^k\rangle \le 2\sqrt{K} \left\| \frac{v-\sigma^k}{\sqrt{\sigma^k}}\right\|_{L^2([0,T]\times [0,1])}.
\end{equation*}
Noticing that $\varphi_t^k=h$, and letting $k$ go to infinity, we get that
\begin{equation}\label{UniqRobin_ineqfinal}
	\langle w_t,h\rangle \le 0.
\end{equation}
Define $A_t = \big\lbrace u\in [0,1]\; : \; w_t(u)>0\big\rbrace$ for $t\in [0,T]$. We can approximate the function $\ind_{A_t}$ in $L^2\big( [0,1]\big)$ by a sequence of functions $(h^k)_{k\in\N}$ in $\mathcal{C}^2\big( [0,1]\big)$ such that $h^k(0)=h^k(1)=0$ for all $k\in\N$. As \eqref{UniqRobin_ineqfinal} holds for any of these functions, letting $k$ go to infinity shows that
\begin{equation*}
	\int_0^1 w_t^+(u)\diff u \le 0,
\end{equation*}
where $w^+=\max (0,w)$. As a consequence, $w_t\le 0$ \textit{i.e.}~$\rho_t\le \tilde\rho_t$ almost everywhere in $[0,1]$, and this is true for any $t\in [0,T]$. It shows that $\rho\le\tilde{\rho}$ almost everywhere in $[0,T]\times [0,1]$, and we can prove the converse inequality in the exact same way, so finally $\rho=\tilde\rho$ almost everywhere and it concludes the proof of the uniqueness of weak solutions.\hfill$\qed$

\subsection{Dynkin's martingale}
\label{sec:appmartingale}

\begin{proposition}\label{prop:dynkin}
	Let 
	\begin{equation*}
		G\in\begin{cases}
			\mathcal{C}_c^{1,2}\big( [0,T]\times (0,1)\big) & \mbox{ if } \theta <1,\\
			\mathcal{C}^{1,2}\big( [0,T]\times [0,1]\big) & \mbox{ if } \theta \ge 1
		\end{cases}
	\end{equation*}
	and consider the process defined by
	\begin{equation*}
		\forall t\ge 0,\qquad \bb{M}_t^N(G): = \langle m_t^N,G_t\rangle -\langle m_0^N,G_0\rangle - \int_0^t\langle m_s^N,\partial_tG_s\rangle\diff s-\int_0^tN^2\mathcal{L}_N\langle m_s^N,G_s\rangle\diff s
	\end{equation*}
	It is a mean-zero martingale with respect to the natural filtration of $(\eta (t))_{t\ge 0}$, and it satisfies
	\begin{equation}\label{martingalevanishes}
		\forall T>0,\quad \forall \delta >0,\qquad \mathbb{P}_{\nu_0^N}\left( \sup_{t\in [0,T]}\big| \mathbb{M}_t^N(G)\big| >\delta\right) \xrightarrow[N\to +\infty]{} 0.
	\end{equation}
\end{proposition}

\begin{proof}
	Thanks to \cite[Lemma 5.1, Appendix 1.5]{KL}, we know that this process is indeed a mean-zero martingale, and that its quadratic variation is given by
	\begin{equation*}
		\big\langle \mathbb{M}^N(G)\big\rangle_t = \int_0^t N^2\Big( \mathcal{L}_N\big( \langle m_s^N,G_s\rangle^2\big)-2\langle m_s^N,G_s\rangle\mathcal{L}_N\langle m_s^N,G_s\rangle\Big)\diff s.
	\end{equation*}
	Making simple, but tedious computations, one gets that
	\begin{multline*}
		\big\langle \mathbb{M}^N(G)\big\rangle_t = \int_0^t \frac{1}{N^2}\sum_{x=1}^{N-2} c_{x,x+1}(s)\nabla_N^+G_s\Big(\frac xN\Big) \diff s \\
		+ \frac{\kappa}{N^\theta}\int_0^t \left(b_\ell (s) G_s\Big(\frac1N\Big)^2+b_r(s)G_s\Big( \frac{N-1}{N}\Big)^2\right)\diff s.
	\end{multline*} 
	Using the fact that the local function $c_{x,x+1}$ is bounded by 1, and by smoothness of $G$, if $\theta \ge 1$
	\begin{equation}\label{eq:estim_varquadratique1}
		\forall t\ge 0,\qquad \big\langle \mathbb{M}^N(G)\big\rangle_t \le \frac{t\| \partial_uG \|_\infty}{N} + \frac{2\kappa t\| G\|_\infty^2}{N^\theta}
	\end{equation}
	whereas for $\theta <1$, as $G_s$ is compactly supported the boundary terms will vanish for $N$ large enough so that we simply have
	\begin{equation}\label{eq:estim_varquadratique2}
		\forall t\ge 0,\qquad \big\langle \mathbb{M}^N(G)\big\rangle_t \le \frac{t\| \partial_uG \|_\infty}{N}.
	\end{equation}
	But notice that applying successively Markov's and Doob's $L^2$ inequalities, we have
	\begin{align*}
		\mathbb{P}_{\nu_0^N}\left( \sup_{t\in [0,T]}\big| \mathbb{M}_t^N(G)\big| >\delta\right) & \le \frac{1}{\delta^2}\E_{\nu_0^N}\left[ \bigg( \sup_{t\in [0,T]}\big|\mathbb{M}_t^N(G)\big|\bigg)^2\right]\\
		& \le \frac{4}{\delta^2}\E_{\nu_0^N}\Big[ \big|\mathbb{M}_T^N(G)\big|^2\Big]\\
		& \le \frac{4}{\delta^2} \E_{\nu_0^N}\Big[ \big\langle \mathbb{M}^N(G)\big\rangle_T\Big]
	\end{align*}
	so inequalities \eqref{eq:estim_varquadratique1} and \eqref{eq:estim_varquadratique2} imply \eqref{martingalevanishes}.
\end{proof}

\section*{Declarations}

\subsection*{Fundings}

This project is partially supported by the ANR grants
	MICMOV (ANR-19-CE40-0012) and CONVIVIALITY (ANR-23-CE40-0003)  of the French National Research Agency (ANR).

\subsection*{Competing interests}

The authors have no competing interests to declare that are relevant to the content of this article.

\subsection*{Data availability}

No datasets were generated or analysed during the current study.

\bibliographystyle{alpha}
\bibliography{reference.bib}

\end{document}